%% file: DimensionEstimator_arxiv.tex
\documentclass[english]{jocgarxiv}  % use this to get page numbers
\usepackage{amsmath,amsfonts}

\title{%
	\MakeUppercase{Minimax Rates for Estimating the Dimension of a Manifold}%
}

\author{%
	Jisu Kim,%
	\thanks{\affil{Carnegie Mellon University}, 
		\email{jisuk1@andrew.cmu.edu}, supported by Samsung Scholarship and partially supported by NSF CAREER Grant DMS 1149677.}\,
	Alessandro Rinaldo,%
	\thanks{\affil{Carnegie Mellon University}, 
		\email{arinaldo@cmu.edu}, partially supported by NSF CAREER Grant DMS 1149677.}\,
	and Larry Wasserman%
	\thanks{\affil{Carnegie Mellon University}, 
		\email{larry@stat.cmu.edu}}
}

\usepackage[T1]{fontenc}
\usepackage[utf8]{inputenc}
\usepackage{geometry}
\usepackage{mathrsfs}
\usepackage{amsthm,amsmath}
\usepackage{amstext}
\usepackage{amssymb}
\usepackage{graphicx}
\usepackage{setspace}
\usepackage{esint}
\usepackage{color}
\usepackage[authoryear]{natbib}
\usepackage[all]{xy}
\usepackage[subrefformat=parens,labelformat=parens]{subcaption}
\makeatletter
\numberwithin{equation}{section}
\numberwithin{figure}{section}
  \theoremstyle{plain}
  \newtheorem*{lem*}{\protect\lemmaname}
  \theoremstyle{plain}
  \newtheorem*{prop*}{\protect\propositionname}
  \theoremstyle{plain}
  \newtheorem*{cor*}{\protect\corollaryname}
  \theoremstyle{remark}
  \newtheorem*{claim*}{\protect\claimname}
 \theoremstyle{definition}
 \newtheorem{defn}{\protect\definitionname}
 \theoremstyle{remark}
 \newtheorem{rem}{\protect\remarkname}
\theoremstyle{plain}
\newtheorem{thm}{\protect\theoremname}
  \theoremstyle{plain}
  \newtheorem{lem}[thm]{\protect\lemmaname}
  \theoremstyle{plain}
  \newtheorem{prop}[thm]{\protect\propositionname}
  \theoremstyle{plain}
  	
  \theoremstyle{remark}
  \newtheorem{claim}[thm]{\protect\claimname}

\usepackage{chngcntr}

\makeatother

\let\hat\widehat

\usepackage{babel}
  \providecommand{\claimname}{Claim}
  \providecommand{\corollaryname}{Corollary}
  \providecommand{\definitionname}{Definition}
  \providecommand{\lemmaname}{Lemma}
  \providecommand{\propositionname}{Proposition}
  \providecommand{\remarkname}{Remark}
\providecommand{\theoremname}{Theorem}

\begin{document}
\bibliographystyle{plainnat}

\maketitle

%\begin{center}
%{\bf\Large Minimax Rates for Estimating the Dimension of a Manifold}\\
%{\bf Jisu Kim, Alessandro Rinaldo and Larry Wasserman}\\
%{\bf Carnegie Mellon University}\\
%{\bf January 4 2016}
%\end{center}

\input{DimensionEstimator_main}

\bibliography{reference}

\input{DimensionEstimator_appendix}

\end{document}

%% file: DimensionEstimator_main.tex
%\begin{quote}
%{\em 
%Many algorithms in machine learning and computational geometry
%require, as input, the intrinsic
%dimension of the manifold that supports
%the probability distribution of the data. 
%This parameter is rarely known and
%therefore has to be estimated from the data. 
%We characterize the statistical difficulty of this problem.
%Specifically, we derive upper and lower bounds on the
%minimax rate for estimating the dimension.
%First we consider the problem of 
%testing the hypothesis that the
%support of the data-generating probability distribution is a
%well-behaved manifold of intrinsic dimension $d_1$ versus the
%alternative that it is of different dimension $d_2$. With an i.i.d. sample of
%size $n$, we provide an upper bound on the sum of type I and II errors of order
%$\approx n^{-\left(d_{2}/d_{1}-1\right)n}$  based on the
%traveling salesman path through the data points, where $\epsilon$ is an
%arbitrarily small positive number. We also demonstrate a lower bound of
%$\approx n^{-(2d_{2}-2d_{1})n}$, by applying Le
%Cam's lemma with a specific set of $d_{1}$-dimensional probability
%distributions. We then extend these results to get minimax rates
%for estimating the dimension of a well-behaved manifolds.
%We obtain an upper bound of order
%$\approx n^{-(\frac{1}{m-1})n}$ and a lower bound of
%order $\approx n^{-2n}$, where $m$ is the
%embedding dimension.}
%\end{quote}

\begin{abstract}
Many algorithms in machine learning and computational geometry require, as input, the intrinsic dimension of the manifold that supports the probability distribution of the data. This parameter is rarely known and therefore has to be estimated. We characterize the statistical difficulty of this problem by deriving upper and lower bounds on the minimax rate for estimating the dimension. First, we consider the problem of testing the hypothesis that the support of the data-generating probability distribution is a well-behaved manifold of intrinsic dimension $d_1$ versus the alternative that it is of dimension $d_2$, with $d_{1}<d_{2}$. With an i.i.d. sample of size $n$, we provide an upper bound on the probability of choosing the wrong dimension of $O\left( n^{-\left(d_{2}/d_{1}-1-\epsilon\right)n} \right)$, where $\epsilon$ is an arbitrarily small positive number. The proof is based on bounding the length of the traveling salesman path through the data points. We also demonstrate a lower bound of $\Omega \left( n^{-(2d_{2}-2d_{1}+\epsilon)n} \right)$, by applying Le Cam's lemma with a specific set of $d_{1}$-dimensional probability distributions. We then extend these results to get minimax rates for estimating the dimension of well-behaved manifolds. We obtain an upper bound of order $O \left( n^{-(\frac{1}{m-1}-\epsilon)n} \right)$ and a lower bound of order $\Omega \left( n^{-(2+\epsilon)n} \right)$, where $m$ is the embedding dimension.
\end{abstract}

\section{Introduction}

Suppose that
$X_1,\ldots, X_n$ is an i.i.d. sample from a distribution $P$
whose support  is an unknown, well-behaved, manifold $M$
of dimension $d$ in $\mathbb{R}^m$, where $1\leq d\leq m$.
Manifold learning refers broadly to a suite of techniques from statistics and
machine learning
aimed at estimating $M$ or some of its features 
based on the data.

Manifold learning procedures are widely used in high dimensional data analysis,
mainly to alleviate the curse of dimensionality.
Such algorithms map the data to a new, lower dimensional coordinate
system
\citep{Bellman1961,LeeV2007.ch1,HastieTF2009.ch14}, with little
loss in accuracy.
Manifold learning can greatly reduce the dimensionality of the data. 

Most manifold learning techniques require, as input, the intrinsic dimension of
the manifold. However, this quantity is almost never known in advance and
therefore has to be estimated from the data. 

%%
%%estimated. Most manifold learning methods need intrinsic dimension
%%as an external and user-defined parameter\citep{LeeV2007.ch3},
%%including {[} TODO : suggests several manifold learning methods which
%%assumes dimension a priori {]}. However, dimension is usually not
%%observable from data.. Sometimes intrinsic dimension is given from
%%scientific reason, but usually it is not. Then intrinsic dimension
%%is not apparent from data. For example, when data lie on space-filling
%%curve\citep{Steele1997.ch2,Buchin2008.ch2}, it is hard to tell whether
%%intrinsic dimension is $1$ or higher.
%%

Various intrinsic dimension estimators have been proposed and analyzed; 
\citep[see, e.g.,][]{LeeV2007.ch3,Koltchinskii2000,Kegl2003,LevinaB2004,HeinA2005,
RaginskyL2005,LittleJM2009,LittleMR2011,SricharanRH2010,
RozzaLCCC2012,CamastraS2015}. However,
characterizing the intrinsic statistical hardness of estimating the dimension remains an open problem.

The traditional way of measuring the difficulty of a statistical problem
is to bound its {\it minimax risk,} which in the present setting is loosely described as the worst possible statistical performance of an optimal dimension estimator. Formally, given a class of probability distribution
$\mathcal{P}$, the minimax risk $R_n = R_n(\mathcal{P})$ is defined as
\begin{equation}
\label{eq:intro.minimax}
R_n = \inf_{\hat{d}_{n}}\sup_{P\in \mathcal{P}} \mathbb{E}_P\left[ 1(\hat{d}_{n} \neq d(P))\right].
\end{equation}
In Equation \eqref{eq:intro.minimax}, $d(P)$ is the dimension of the support of $P$, $\mathbb{E}_P$ denotes the
expectation with respect to the distribution $P$,
$1(\cdot)$ is the indicator function, and the infimum is over all estimators (measurable functions of the
data) 
$\hat{d}_{n} = \hat{d}_{n}(X_1,\ldots,X_n)$ of the dimension $d(P)$. 
The risk $\mathbb{E}_P[ 1(\hat{d}_{n} \neq d(P))]$ of a dimension estimator $\hat{d}_{n}$ is
the probability that $\hat{d}_{n}$ differs from the true dimension $d(P)$ of the
support of the data generating distribution $P$.  The
minimax risk $R_{n}(\mathcal{P})$, which is a function of both the sample size $n$
and the class $\mathcal{P}$, quantifies the intrinsic hardness of the dimension
estimation problem, in the sense that {\it any dimension estimator} cannot have
a risk smaller than $R_n$ uniformly over every $P \in \mathcal{P}$.

The purpose of this paper is to obtain upper and lower bounds on the
minimax risk $R_{n}$ in \eqref{eq:intro.minimax}. We impose several
regularity conditions on the set of manifolds supporting the
distribution in the class $\mathcal{P}$, in order to make the problem
analytically tractable and also to avoid pathological cases, such as
space-filling manifolds. We first assume that the manifold supporting
the data generating distribution $P$ has two possible dimensions,
$d_{1}$ and $d_{2}$. This assumption is then relaxed to any dimension
$d(P)$ between $1$ and the embedding dimension $m$. Our main result is
the following theorem. See Section \ref{sec:definition.regular} for
the definition of the class $\mathcal{P}$ of probability distributions
supported on well-behaved manifolds in $\mathbb{R}^m$.

\begin{thm}
\label{thm:intro.bound}
%{(Bounds for the minimax rate)}
The minimax risk $R_n$ in \eqref{eq:intro.minimax} satisfies{\color{blue},}
$a_n \leq R_n \leq b_n$, where
\begin{align}
a_n &= (C_{K_{I}}^{(\ref*{prop:multidimension.lower})})^n \min\{ \tau_{\ell}^{-4} n^{-2},1\}^n,\label{eq:intro.lower.bound}\\
b_n &= (C_{K_{I},K_{p},K_{v},m}^{(\ref*{prop:multidimension.maximumrisk})})^n \max\left\{1, \tau_{g}^{-(m^2 - m)n}\right\} n^{-\frac{n}{m-1}},\label{eq:general.upper.bound}
\end{align}
and the constants
$\tau_{\ell}$,
$\tau_{g}$,
$C_{K_{I}}^{(\ref*{prop:multidimension.lower})}$ and
$C_{K_{I},K_{p},K_{v},m}^{(\ref*{prop:multidimension.maximumrisk})}$ depend on 
$\mathcal{P}$ and are defined in Section \ref{sec:multidimension}.
\end{thm}

We now make a few remarks about the previous theorem.
\begin{itemize}
\item 
Since the dimension $d(P)$ is a discrete quantity, the minimax rate $R_{n}$ in
\eqref{eq:intro.minimax} is superexponential in sample size. This result seems
at odds with the exponential rate obtained by \citep[Proposition
2.1]{Koltchinskii2000}. These different rates are due to different model
assumptions. In \citep{Koltchinskii2000} the data generating
distribution is the convolution of a probability distribution supported on a manifold with a
noise distribution supported on a set of full dimension $m$. In contrast, here
we assume that the data are generated from a probability distribution supported
on a manifold. Under our noiseless model, distributions supported on
manifolds with different dimension are more easily distinguishable, hence the minimax
rate $R_{n}$ converges to $0$ faster than under the model with noise assumed by \citep{Koltchinskii2000}. 

\item The key quantities that appear in the lower bound
\eqref{eq:intro.lower.bound} and the upper bound
\eqref{eq:general.upper.bound} are the global reach $\tau_{g}$ and
the local reach $\tau_{\ell}$ of the manifold, which are defined in
Section \ref{sec:definition.regular}.  These reach parameters can be
roughly thought as the inverse of the usual notion of curvature
\citep[see, e.g.][]{Federer1959}, and they affect the performance of
any dimension estimator: a manifold with low reach may appear more
space-filling than a manifold of the same dimension but with higher
reach, thus making the task of resolving the dimension harder.
Indeed, our analysis shows formally that the minimax risk $R_n$ in
\eqref{eq:intro.minimax} decreases in the values of the
reaches. Given their crucial role, we have attempted to make the
dependence of the minimax risk $R_n$ on both $\tau_{g}$ and
$\tau_{\ell}$ as explicit as possible.

\item There is a gap between the lower bound
\eqref{eq:intro.lower.bound} and the upper bound
\eqref{eq:general.upper.bound}.
Nonetheless, as far as we are aware, these are the most precise
bounds on $R_n$ that are available.
\end{itemize}

This paper is organized as follows. In Section \ref{sec:definition.regular}, we
formulate and discuss regularity
conditions on distributions and their supporting manifolds. In Section
\ref{sec:upper}, we provide an upper bound on the minimax rate
by considering the traveling salesman path through the points. In Section
\ref{sec:lower}, we derive a lower bound on the
minimax rate by applying Le Cam's lemma with a specific set of
$d_{1}$-dimensional and $d_{2}$-dimensional probability distributions. In Section
\ref{sec:multidimension}, we extend our upper bound and lower bound for the case where the intrinsic dimension varies from $1$ to $m$.

\section{Definitions and Regularity Conditions}
\label{sec:definition.regular}

In this section, we define the set $\mathcal{P}$ of probability distributions
that we consider in bounding the minimax risk $R_{n}$ in \eqref{eq:intro.minimax}.
Such distributions are supported on manifolds whose dimension $d$ is between
$1$ and $m$, where $m$ is the dimension of the embedding
space.
In particular, we require that the supporting manifolds have a uniform lower
bound on their reach parameters $\tau_g$ and $\tau_l$.
The resulting class of distributions is denoted by
\begin{equation}
\label{eq:definition.distribution}
\mathcal{P} = 
\bigcup_{d=1}^m
\mathcal{P}_{\tau_{g},\tau_{\ell},K_{I},K_{v},K_{p}}^{d}.
\end{equation}
In the rest of this section, we will  make the definition
$\mathcal{P}_{\tau_{g},\tau_{\ell},K_{I},K_{v},K_{p}}^{d}$ precise. Readers who
are not interested in the details may skip the rest of the section.
All the proofs for this section are in Section \ref{sec:proof.definition.regular}.

\subsection{Notation and Basic Definitions}

% For
%positive integers $n_{1},n_{2},d$ such that $1\leq n_{1}\leq n_{2}\leq
%d$, the coordinate projection map
%$\Pi_{n_{1}:n_{2}}:\mathbb{R}^{d}\rightarrow\mathbb{R}^{n_{2}-n_{1}+1}$
%is defined by
%$\Pi_{n_{1}:n_{2}}(x_{1},\ldots,x_{d})=(x_{n_{1}},x_{n_{1}+1},\ldots,x_{n_{2}})$.
%We let $S_{n}$
%denote the permutation group on $\{1,\ldots,n\}$. For any product set $J^n
%\subset \mathbb{R}^n$, $S_{n}$ acts on $J^{n}$ and its subsets by
%applying a coordinate change, i.e. for $\sigma\in S_{n}$ and $x\in
%J^{n}$, $\sigma x:=(x_{\sigma(1)},\ldots,x_{\sigma(n)})$, and, for any
%$A\subset J^{n}$, $S_{n}A:=\{\sigma x\in J^{n}:\ \sigma\in
%S_{n},\ x\in A\}$. Finally, for a metric space $(A,d_A)$ and $x \in
%A$, we let $B_{A}(x,r)=\{y\in A:\,d_{A}(y,x)<r\}$ be the ball with
%center $x$ and radius $r$. 
For the reader's convenience, we provide a list of the notation used throughout
the paper in Table \ref{tab:definition.notation}.

\begin{table}[!ht]
	\begin{center}
	\begin{tabular}{c|l}
		\hline 
		\textbf{Notation} & \textbf{Definition}\tabularnewline
		\hline 
		$1(\cdot)$  & indicator function.\tabularnewline
		$d$, $d_{1}$, $d_{2}$  & dimension of a manifold.\tabularnewline
		$\hat{d}_{n}$ & dimension estimator.\tabularnewline
		$dist_{A}(\cdot,\cdot)$  & distance function on the set $A$.\tabularnewline
		$dist_{A,||\cdot||}(\cdot,\cdot)$  & distance function on the set $A$ induced by the norm $||\cdot||$.\tabularnewline
		$\exp_{p}(\cdot)$ & exponential map on point $p\in M$.\tabularnewline
		$\ell(\cdot,\cdot)$ & loss function.\tabularnewline
		$n$  & size of the sample.\tabularnewline
		$m$  & dimension of the embedding space.\tabularnewline
		$p$, $q$ & points on the manifold $M$.\tabularnewline
		$vol_{A}(\cdot)$  & volume function of $A$.\tabularnewline
		$B_{A}(x,r)$  & open ball with center $x$ and radius $r$, $\{y\in A:\,dist_{A}(y,x)<r\}$. \tabularnewline
		$C_{a_{1},\ldots,a_{k}}$ & constant that depends only on $a_{1},\,\ldots,\,a_{k}$. \tabularnewline
		$I$ & cube $[-K_{I},\,K_{I}]^{m}$.\tabularnewline
		$K_{I}$, $K_{v}$, $K_{p}$ & fixed constants for regular conditions; see Definition \ref{def:regular.manifold}.\tabularnewline
		$M$  & manifold.\tabularnewline
		$P$  & data generating probability distribution.\tabularnewline
		$R_{n}$  & minimax risk $\underset{\hat{d}_{n}}{\inf}\underset{P\in\mathcal{P}}{\sup}\mathbb{E}_{P}\left[1\left(\hat{d}_{n}\neq d(P)\right)\right]$;
		see \eqref{eq:intro.minimax}, \eqref{eq:minimax.minimax}, and \eqref{eq:regular.minimax}.\tabularnewline
		$S_{n}$  & permutation group on $\{1,\,\ldots,\,n\}$.\tabularnewline
		$T$ & subset of $I^{n}\subset(\mathbb{R}^{d})^{n}$, used in Section \ref{sec:lower}.\tabularnewline
		$T_{p}M$  & tangent space of a manifold $M$ at $p$.\tabularnewline
		$X_{1},\,\ldots,\,X_{n}$  & sample points.\tabularnewline
		$\mathcal{M}$ & set of manifolds; see Definition \ref{def:regular.manifold}.\tabularnewline
		$\mathcal{P}$ & set of distributions; see Definition \ref{def:regular.manifold}.\tabularnewline
		$\gamma$ & path on a manifold $M$.\tabularnewline
		$\pi_{A}(\cdot)$ & projection function onto a closed set $A$.\tabularnewline
		$\sigma$ & permutation.\tabularnewline
		$\tau(M)$ & reach of a manifold $M$; see Definition \ref{def:definition.reach} and Lemma~\ref{lem:definition.reachball}.\tabularnewline
		$\tau_{g}$ & lower bound for global reach; see Definition \ref{def:regular.manifold}.\tabularnewline
		$\tau_{\ell}$ & lower bound for local reach; see Definition \ref{def:regular.manifold}.\tabularnewline
		$\omega_{d}$  & volume of the unit ball in $\mathbb{R}^{d}$, $\frac{\pi^{\frac{d}{2}}}{\Gamma\left(\frac{d}{2}+1\right)}$.\tabularnewline
		$\Pi_{n_{1}:n_{2}}$  & coordinate projection map: $\Pi_{n_{1}:n_{2}}(x_{1},\,\ldots,\,x_{d})=(x_{n_{1}},\,\ldots,\,x_{n_{d}})$.\tabularnewline
		\hline 
	\end{tabular}
	\end{center}
	\caption{Table of notations and definitions.}
	\label{tab:definition.notation} 
\end{table}

We now briefly review some notations from differential geometry. For a more detailed treatment, we
refer the reader to standard textbooks on this topic \citep[see,
e.g.,][]{Lee2000, Lee2003, Petersen2006, doCarmo1992}.
A topological manifold of dimension $d$ is a topological space $M$ and a family of
homeomorphisms $\varphi_{\alpha}:U_{\alpha}\subset\mathbb{R}^{d}\rightarrow
V_{\alpha}\subset M$ from an open subset of $\mathbb{R}^{d}$ to an open
subset of $M$ such that $\underset{\alpha}{\bigcup}\varphi_{\alpha}(U_{\alpha})=M$. A topological space $M$ is considered to be a $d$-dimensional manifold if there exists a family of homeomorphisms $\varphi_{\alpha}:U_{\alpha}\subset\mathbb{R}^{d}\rightarrow
V_{\alpha}\subset M$ such that $(M,\{\varphi_\alpha\}_{\alpha})$ is a manifold. If $M$ is a $d$-dimensional manifold,
such $d$ is unique and is called the dimension of a manifold. 
If, for any pair $\alpha$, $\beta$, with
$\varphi_{\alpha}(U_{\alpha})\cap\varphi_{\beta}(U_{\beta})\neq\emptyset$, 
$\varphi_{\beta}^{-1}\circ\varphi_{\alpha}:U_{\alpha}\cap U_{\beta}\rightarrow U_{\alpha}\cap
U_{\beta}$ is $C^{k}$, then $M$ is a
$C^{k}$-manifold. 

We assume that the topological manifold $M$ is embedded in
$\mathbb{R}^{m}$, i.e. $M\subset \mathbb{R}^{m}$, and the metric is
inherited from the metric of $\mathbb{R}^{m}$. For a topological
manifold $M\subset \mathbb{R}^{m}$ and for any $p,\,q\in M$, a path
joining $p$ to $q$ is a map $\gamma:[a,b]\rightarrow M$ for
some $a,b\in\mathbb{R}$ such that $\gamma(a)=p$, $\gamma(b)=q$. The length
of the curve $\gamma$ is defined as $Length(\gamma)=\int_{a}^{b}
||\gamma'(t)||_{2} dt$. A topological manifold $M$ is equipped with
the distance $dist_{M}:M\times M\rightarrow \mathbb{R}$ as
$dist_{M}(p,q)=\underset{\gamma:\text{ path joining }p\text{ and
  }q}{\inf}Length(\gamma)$. A path $\gamma:[a,b]\rightarrow M$
is a geodesic if for all $t,t'\in[a,b]$,
$dist_{M}(\gamma(t),\gamma(t')) = |t-t'|$.

Let $T_p M$ denote the tangent space to $M$ at $p$.
Given $p\in M$, there exist a set $0\in \mathcal{E}\subset T_{p}(M)$
and a mapping $\exp_{p}:\mathcal{E}\subset T_{p}M \rightarrow M$ such
that $t \rightarrow \exp_{p}(tv)$, $t\in(-1,1)$, is the unique
geodesic of $M$ which, at $t=0$, passes through $p$ with
velocity $v$, for all $v \in \mathcal{E}$. The map
$\exp_{p}:\mathcal{E}\subset T_{p}M \rightarrow M$ is called the exponential map on $p$. 

One of the key conditions that we impose in Section \ref{subsec:regular} is
about the reach.

\begin{defn}
\label{def:definition.reach}
For a compact $d$-dimensional topological manifold $M\subset \mathbb{R}^{m}$ (with
boundary), the {\em reach} of $M$, $\tau(M)$, is defined as the
largest value of  $r>0$ such that each $x\in\mathbb{R}^{m}$ with
$dist_{\mathbb{R}^{m}}(x,M)<r$ has  a unique projection $\pi_{M}(x)$ on
$M$, i.e.
\begin{align}
\label{eq:definition.reach}
\tau(M):=\sup\Bigl\{ r:\, & \forall x\in\mathbb{R}^{m}\text{ with }dist_{\mathbb{R}^{m}}(x,\,M)<r, \nonumber\\
	& \exists!\pi_{M}(x)\in M\text{ s.t. }||x-\pi_{M}(x)||_{2}=\underset{y\in M}{\inf}||x-y||_{2}\Bigr\}.
\end{align}
\end{defn}

See \citep{Federer1959} for further details.
%For a manifold
%$M$, the reach of the manifold can be roughly understood as how much $M$ can be
%\textcolor{red}{UNCLEAR: lifted to outward direction of $M$}.
The reach $\tau(M)$ can be also considered as one kind of curvature, and can be understood as an inverse of other usual curvatures. See Figure \ref{fig:definition.reach}\subref{subfig:definition.reach.projection} for the illustration of Definition \ref{def:definition.reach}. There are several equivalent ways to define the reach $\tau(M)$ in
\eqref{eq:definition.reach} for the manifold $M$. 
The reach $\tau(M)$ is the maximum radius of a ball that 
can be rolled freely over the manifold $M$, as in Lemma
\ref{lem:definition.reachball}. See Figure \ref{fig:definition.reach}\subref{subfig:definition.reach.ball} for the illustration of Lemma~\ref{lem:definition.reachball}.

\begin{lem}
	\label{lem:definition.reachball}
	For a manifold $M\subset \mathbb{R}^{m}$, 
	\begin{align}
	\label{eq:definition.reachball}
	\tau(M)=\sup\bigl\{ r:\, & \forall x\in M,\,\forall y\in \mathbb{R}^{m}\text{ with }y-x\perp T_{x}M\text{ and }||y-x||_{2}=r,\nonumber \\
	& B_{\mathbb{R}^{m}}(y,\,r)\cap M=\emptyset\bigr\}.
	\end{align}
\end{lem}

\begin{proof}[Proof of Lemma~\ref{lem:definition.reachball}]
	\citep[See][Theorem 4.18]{Federer1959}.
\end{proof}

\begin{figure}
	\begin{center}
		\begin{subfigure}[b]{0.4\textwidth}
			\begin{center}
			\includegraphics{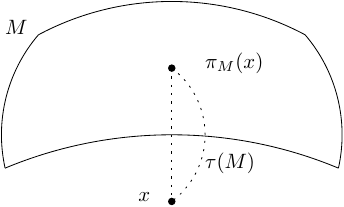}
			\end{center}
			\caption{reach $\tau(M)$ in Definition \ref{def:definition.reach}.}
			\label{subfig:definition.reach.projection}
		\end{subfigure}
		\begin{subfigure}[b]{0.4\textwidth}
			\begin{center}
			\includegraphics{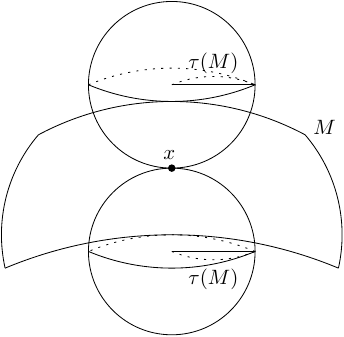}
			\end{center}
			\caption{reach $\tau(M)$ in Lemma \ref{lem:definition.reachball}.}
			\label{subfig:definition.reach.ball}
		\end{subfigure}
	\end{center}
	\caption{For a manifold $M$, there are several equivalent definitions for reach $\tau(M)$ in Definition \ref{def:definition.reach}. \protect\subref{subfig:definition.reach.projection} The reach $\tau(M)$ is the supremum value of $r$ such that for all $x\in\mathbb{R}^{m}$ with $dist_{\mathbb{R}^{m}}(x,M)<r$ has unique projection $\pi_{M}(x)$ to $M$, as in \eqref{eq:definition.reach}. \protect\subref{subfig:definition.reach.ball} The reach $\tau(M)$ is the maximum radius of a ball that you can roll over the manifold $M$, as in \eqref{eq:definition.reachball}.}
	\label{fig:definition.reach}
\end{figure}

\subsection{Minimax Theory}
\label{subsec:minimax}

The minimax rate is the risk of an estimator that performs best in the
worst case, as a function of the sample size \citep[see, e.g.][]{Tsybakov2008}. 
Let $\mathcal{P}$ be a collection of
probability distributions over the same sample space $\mathcal{X}$ and
let $\theta:\mathcal{P} \rightarrow \Theta$ be a function over
$\mathcal{P}$ taking values in some space $\Theta$, the parameter
space. We can think of $\theta(P)$ as the feature of interest of the
probability distribution $P$, such as its mean, or, as in our case,
the dimension of its support.  For the fixed sample size $n$, suppose $X =
(X_{1},\ldots,X_{n})$ is an i.i.d. 
(independent and identically distributed)
sample drawn from a fixed probability
distribution $P \in \mathcal{P}$. Thus $X$ takes values in the
$n$-fold product space $\mathcal{X}^n = \mathcal{X} \times \cdots
\times \mathcal{X}$ and is distributed as $P^{(n)}$, the $n$-fold
product measure. An estimator
$\hat{\theta}_{n}:\mathbb{R}^{n}\rightarrow\Theta$ is any measurable
function that maps the observation $X$ into the parameter space
$\Theta$. Let $\ell:\Theta\times\Theta\rightarrow\mathbb{R}$ be a loss
function, a non-negative bounded function that
measures how different two parameters are. Then for a fixed estimator
$\hat{\theta}_n$ and a fixed distribution $P$, the risk of
$\hat{\theta}_n$ is defined as
\[
\mathbb{E}_{P^{(n)}}\left[\ell\left(\hat{\theta}_{n}(X),\theta(P)\right)\right].
\]
Then for a fixed estimator $\hat{\theta}_n$, its maximum risk is the supremum of
its risk over every distribution $P\in\mathcal{P}$, that is, 
\begin{equation}
\label{eq:minimax.maximumrisk}
\underset{P\in
  \mathcal{P}}{\sup}\mathbb{E}_{P^{(n)}}\left[\ell\left(\hat{\theta}_{n}(X),\theta(P)\right)\right].
\end{equation}
 The minimax risk associated with $\mathcal{P}$, $\theta$, $\ell$ and $n$ is the
 maximal risk of any estimator that performs the best under the worst possible
 choice of $P$. Formally, the \emph{minimax risk} is
\begin{equation}
\label{eq:minimax.minimax}
R_{n} = \inf_{\hat{\theta}_{n}}\sup_{P\in \mathcal{P}}
\mathbb{E}_{P^{(n)}}\left[\ell\left(\hat{\theta}_{n}(X),\theta(P)\right)\right].
\end{equation}
The minimax risk $R_{n}$  in \eqref{eq:minimax.minimax} is often viewed as a
function of the sample size $n$, in
which case any positive sequence $\psi_n$ such that $\lim_{n \rightarrow \infty}
R_n/\psi_n$ remains bounded away from $0$ and $\infty$ is called a {\it minimax
rate.}
Notice that minimax rates are unique up to constants and lower order terms. 

To define a meaningful minimax risk, 
it is essential to have some constraint 
on the set of distributions $\mathcal{P}$ in \eqref{eq:minimax.maximumrisk} and
\eqref{eq:minimax.minimax}.
If $\mathcal{P}$ is too large, then
the minimax rate $R_{n}$ in \eqref{eq:minimax.minimax} will not
converge to $0$ as $n$ goes to $\infty$: this means that 
the problem is statistically ill-posed.
If $\mathcal{P}$ is too small, the
minimax estimator depends too much on the specific distributions in $\mathcal{P}$
and is not a useful measure of a statistical difficulty.

Determining the value of the minimax risk $R_{n}$ in
\eqref{eq:minimax.minimax} for a given problem requires two separate
calculations: an upper bound on $R_n$ and a lower
bound. In order to derive an upper bound, one analyzes the asymptotic
risk of a specific estimator $\hat{\theta}_n$. Lower bounds are
instead usually computed by measuring the difficulty of a multiple
hypothesis testing problem that entails identifying finitely many
distributions in $\mathcal{P}$ that are maximally difficult to
discriminate \citep[see, e.g.][Section 2.2]{Tsybakov2008}.

For the dimension estimation problem, we obtain an upper bound on
$R_{n}$ by analyzing the performance of an estimator based on the
length of the traveling salesman problem, as described in Section
\ref{sec:upper}. On the other hand, the calculation of the lower bound
presents non-trivial technical difficulties, because probability distributions supported on manifolds of different
dimensions are singular with respect to each other, and therefore trivially
discriminable. In order to overcome such an issue, we resort to
constructing mixtures of mutually singular distributions. We detail
this construction in Section \ref{sec:lower}.

There is a gap
between the lower and upper bounds we derive on the minimax
risk,
as it is often the case in such calculations.
Nonetheless, the derivation of the bounds is of use in
understanding the difficulty of the dimension estimation problem.

\subsection{Regularity conditions on the Distributions and their Supporting Manifolds}
\label{subsec:regular}

In our analysis we require various regularity conditions on the class $\mathcal{P}$ of
probability distributions appearing in the minimax risk \eqref{eq:intro.minimax}.
Most of these conditions are of a geometric nature and concern the properties of
the manifolds supporting the probability distributions in $\mathcal{P}$.
Altogether, our assumptions rule out manifolds
that are so complicated to make the dimension estimation problem unsolvable and,
 therefore, guarantee that the minimax risk
$R_{n}$ in \eqref{eq:minimax.minimax} converges to $0$ as $n$
goes to $\infty$. Such regularity assumptions are quite mild, and in fact allow
for virtually all types of
manifolds usually encountered in manifold learning problems.

Our first assumption is that the probability distributions in $\mathcal{P}$ are
supported over manifold contained inside a compact set, which, without loss of
generality, we take to be the cube $I := [-K_{I},K_{I}]^m$, for some $K_I>0$. See Figure \ref{fig:regular.boundingbox}.

\begin{figure}
	\begin{center}
		\includegraphics{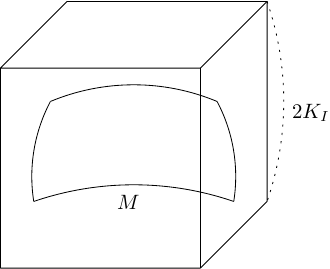}
	\end{center}
	\caption{A manifold $M$ is assumed to be contained inside the cube $I = [-K_{I},K_{I}]^m$, for some $K_I>0$. See Definition \ref{def:regular.manifold}.}
	\label{fig:regular.boundingbox}
\end{figure}

Second, to exclude manifolds that are arbitrarily complicated in the
sense of having unbounded curvatures or of being nearly self intersecting, we
assume that the  reach is uniformly bounded from below.  More precisely, we
will constrain both the global reach and the local reach as follows. Fix $\tau_{g},
\tau_{\ell} \in (0,\infty]$ with $\tau_{g}\leq \tau_{\ell}$. The global reach
condition for a manifold $M$ is that the usual reach $\tau(M)$ in
\eqref{eq:definition.reach} is lower bounded by $\tau_{g}$ as in Figure
\ref{fig:regular.reach}\subref{subfig:regular.reach.global}, and the local reach condition is that $M$ can
be covered by small patches whose reaches are lower bounded by $\tau_{\ell}$, as in
Figure \ref{fig:regular.reach}\subref{subfig:regular.reach.local}.  (See Definition
\ref{def:regular.manifold} below for more details.)
%{\color{blue}
%\begin{defn}
%	Fix $0<\tau_{g}\leq\tau_{\ell}<\infty$ and let $\kappa_{g}:=\frac{1}{\tau_{g}}$, $\kappa_{l}:=\frac{1}{\tau_{\ell}}$. A compact $d$-dimensional topological manifold $M\subset I$ (with boundary) is of {\em (global) reach} at least $\tau_{g}$, if for all $x\in\mathbb{R}^{m}$ with $dist_{\mathbb{R}^{m}}(x,M)<\tau_{g}$	has unique projection $\pi_{M}(x)$ to $M$, i.e. there uniquely exists $\pi_{M}(x)\in M$ such that $||x-\pi_{M}(x)||_{\mathbb{R}^{m}}=\underset{y\in M}{\inf}||x-y||_{\mathbb{R}^{m}}$. $M$ has {\em local reach} at least $\tau_{\ell}$ if for all $x\in M$, there exists neighborhood $U_{x}\subset M$ of $x$ such that $U_{x}$ is of global reach at least $\tau_{\ell}$. See Figure 2.1.
%\end{defn}
%
%\begin{defn}
%We define  $\mathcal{M}_{\tau_{g},\tau_{\ell},K_{I},K_{v}}^{d}$ to be the set of all
%$d$-dimensional topological manifolds in $I$ with
%global and local reach bounded by $\tau_{g}$ and $\tau_{\ell}$, respectively.
%\end{defn}
%}

\begin{figure}
	\begin{center}
		\begin{subfigure}[b]{0.4\textwidth}
			\begin{center}
				\includegraphics{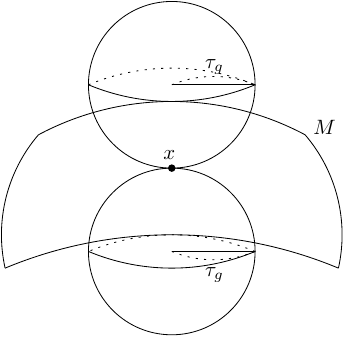}
			\end{center}
			\caption{global reach condition}
			\label{subfig:regular.reach.global}
		\end{subfigure}
		\begin{subfigure}[b]{0.4\textwidth}
			\begin{center}
				\includegraphics{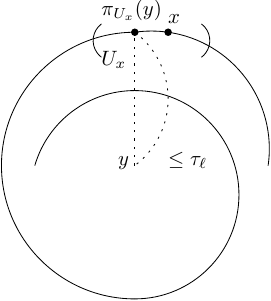}
			\end{center}
			\caption{local reach condition}
			\label{subfig:regular.reach.local}
		\end{subfigure}
	\end{center}
	\caption{A manifold $M$ with \protect\subref{subfig:regular.reach.global} {\em global reach} at least  $\tau_{g}$, or \protect\subref{subfig:regular.reach.local} {\em local reach} at least $\tau_{\ell}$. See Definition \ref{def:regular.manifold}.}
	\label{fig:regular.reach}
\end{figure}
%{\color{blue}
%\begin{rem}
%The above definition is equivalent
%to the following:
%$M$ is of global reach $\geq\tau_{g}$ if $\forall x\in M$, $\forall y\in M$ s.t. $y-x\perp T_{x}M$ and $||y-x||_{2}=\tau_{g}$, $B_{\mathbb{R}^{m}}(y,\tau_{g})\cap M=\emptyset$, where $T_{x}M$ is tangent space of $M$ at $x$.
%\end{rem}
%}

Third, we assume that the data are generated from a distribution $P$
supported on a manifold $M$ having a density with respect to the (restriction of
the) Hausdorff
measure on $M$ bounded from above by some positive constant $K_{p}$.

%{\color{blue}
%\begin{defn}
%Let $\mathcal{B}(I)$ be the Borel subsets of $I$ and $\mathcal{P}$ be a set of
%probability measures on $(I,\mathcal{B}(I))$. Fix $K_{p}\geq (2K_{I})^{m}$. Let
%$\mathcal{P}_{\tau_{g},\tau_{\ell},K_{I},K_{v},K_{p}}^{d}$ be the set of probability
%distributions $P$ supported on a $d$-dimensional manifold $M \in
%\mathcal{M}_{\tau_{g},\tau_{\ell},K_{I},K_{v}}^{d}$,  absolutely continuous with respect
%to the restriction $vol_M$ of the $d$-dimensional Hausdorff measure on $M$ and
%such that $\sup_{x \in M} \frac{dP}{dvol_{M}}(x) \leq K_{p}$.
%For all $P\in\mathcal{P}_{\tau_{g},\tau_{\ell},K_{I},K_{v},K_{p}}^{d}$, define $\dim(P):=d$.
%\end{defn}
%}

For manifolds without boundary, the above conditions suffice for our
analysis.  However, to deal with manifolds with boundary, we need
further assumptions, namely local geodesic completeness and essential
dimension.  A manifold $M$ is said to be complete if any geodesic can
be extended arbitrarily farther, i.e. for any geodesic path
$\gamma:[a,b]\to M$, there exists a geodesic
$\tilde{\gamma}:\mathbb{R}\to M$ that satisfies
$\tilde{\gamma}|_{[a,b]}=\gamma$. \citep[see, e.g.,][]{Lee2000,   Lee2003, Petersen2006, doCarmo1992}.
Accordingly, we define a manifold $M$ to be
locally (geodesically) complete, if any two points inside a geodesic
ball of small enough radius in the interior of $M$ can be joined by a
geodesic whose image also lies on the interior of $M$.

Fifth, we assume the manifold $M$ is of essential dimension $d$, in
volume sense. If we fix any point $p$ in the $d$-dimensional manifold
$M$, then the volume of a ball of radius $r$ grows in order of $r^{d}$
when $r$ is small. By extending this, fix $K_{v}\in (0,2^{-m}]$, and we
say that the manifold $M$ is of essential volume dimension $d$, if
the volume of a geodesic ball of radius $r$ around any point in $M$
is lower bounded by $K_{v}r^{d}\omega_{d}$, for some positive
constant $K_v$ and all $r$ small enough.

    We are now ready to formally define the class $\mathcal{P}$ of probability
    distributions that we will consider in our analysis of the minimax problem
    \eqref{eq:intro.minimax}.

\begin{defn}
\label{def:regular.manifold}
Fix $\tau_{g},\,\tau_{\ell}\in(0,\infty]$, $K_{I}\in[1,\infty)$,
$K_{v}\in(0,2^{-m}]$, with $\tau_{g}\leq\tau_{\ell}$. Let
$\mathcal{M}_{\tau_{g},\tau_{\ell},K_{I},K_{v}}^{d}$ be the set of compact
$d$-dimensional manifolds $M$ such that:
\begin{itemize}
         \item[(1)] $M \subset I:=[-K_{I},K_{I}]^{m}\subset\mathbb{R}^{m}$;
     \item[(2)] $M$ is of {\em global reach} at least $\tau_{g}$,
i.e. $\tau(M)\geq \tau_{g}$, and $M$ is of {\em local reach} at
least $\tau_{\ell}$, i.e. for all $p \in M$, there exists a
neighborhood $U_{p}$ in $M$ such that $\tau(U_{p})\geq
{\tau_{\ell}}$;
	     \item[(3)] $M$ is {\em locally (geodesically) complete} (with respect
          to $\tau_{g}$): for all $p \in {\rm int} (M)$ and for all
          $q_{1}, q_{2} \in B_{M}(p,\,2\sqrt{3}\tau_{g})$, there
          exists a geodesic $\gamma$ joining $q_{1}$ and $q_{2}$ whose
          image lies on $int M$;
     \item[(4)] $M$ is of {\em essential volume dimension $d$} (with respect to
	$K_{v}$ and $\tau_{g}$): if for all $p \in M$ and for all $r\leq
	\sqrt{3} \tau_{g}$, $vol_{M}(B_{M}(p,r))\geq K_{v} r^{d} \omega_{d}$.
\end{itemize}
%\end{defn}
%\begin{defn}
Let $\mathcal{P} = \mathcal{P}_{\tau_{g},\tau_{\ell},K_{I},K_{v},K_{p}}^{d}$ be the
set of Borel probability distributions $P$ such that:
\begin{itemize}
\item[(5)] $P$ is supported on a $d$-dimensional manifold $M\in \mathcal{M}_{\tau_{g},\tau_{\ell},K_{I},K_{v}}^{d}$;
\item[(6)] $P$ is absolutely continuous with respect to the 
restriction $vol_M$ of the $d$-dimensional Hausdorff 
measure on the supporting manifold $M$ and such that $\sup_{x \in M} \frac{dP}{dvol_{M}}(x) \leq K_{p}$.
\end{itemize}
For every $P\in \mathcal{P}_{\tau_{g},\tau_{\ell},K_{I},K_{v},K_{p}}^{d}$, 
denote the dimension of its distribution as $d(P)$.
\end{defn}

%\begin{defn}
%	\label{def:regular.distribution}
%	Fix $\tau_{g},\,\tau_{\ell}\in(0,\infty]$, $K_{I}\in[1,\infty)$,
%	$K_{v}\in(0,2^{-m}]$, $K_{p}\in[(2K_{I})^{m},\infty)$, with
%	    $\tau_{g}\leq\tau_{\ell}$. Let $\mathcal{B}(I)$ be the Borel subsets
%	    of $I$ and $\mathcal{P}^{I}$ be a set of probability measures on
%	    $(I,\mathcal{B}(I))$. i
%\end{defn}

\begin{rem}
For manifolds without boundary, the local completeness condition and the
essential volume dimension condition in Definition \ref{def:regular.manifold}
always hold. The Hopf Rinow Theorem \citep[see, e.g.][Theorem 16]{Petersen2006} implies
that any compact closed manifold without boundary is geodesic complete, which
implies it is locally complete  in the sense of (3) in Definition
\ref{def:regular.manifold}. Also, \citep[Lemma 5.3]{NiyogiSW2008}
implies that, for a $d$-dimensional manifold $M$ and all $0<r\leq 2\tau(M)$, 
\[
vol_{M}(B_{M}(p,r))\geq
r^{d}\left(1-\left(\frac{r}{2\tau(M)}\right)^{2}\right)^{\frac{d}{2}}\omega_{d},
\]
for all $ p\in M$.
Hence, when, for fixed $\tau_{g}>0$,  a $d$-dimensional manifold $M$ (without
boundary)  satisfies
$\tau(M)\geq \tau_{g}$, then for any $0<r\leq\sqrt{3}\tau_{g}$, $vol_{M}(B_{M}(p,r))\geq
2^{-d}r^{d}\omega_{d}$, so the essential volume dimension condition is satisfied.
\end{rem}

\begin{rem}
The notion of the local reach $\tau_{\ell}$ in Definition \ref{def:regular.manifold} is less standard than the global reach $\tau_{g}$, which is the usual definition of the reach in \citep[see, e.g.][]{Federer1959}. The local reach condition is only used in getting the lower bound of the minimax rate $R_{n}$ in Section \ref{sec:lower}, while the global reach condition is used in both Section \ref{sec:upper} and Section \ref{sec:lower}. In fact, the reach of the manifold is determined either by a bottleneck structure or an area of high curvature, as in \cite[Theorem 3.4]{AamariKCMRW2017}. And the global reach condition is imposing regularities on both cases, while the local reach condition is imposing regularities only on the latter case, i.e. on the local curvature. Setting the local reach $\tau_{\ell}$ equal to the global reach $\tau_{g}$ reduces to the model that has conditions only on the usual reach.
\end{rem}

The  regularity conditions in Definition \ref{def:regular.manifold}
imply further constraints on both the distributions in $\mathcal{P}$  and
their supporting manifolds, in Lemma
\ref{lem:regular.bounded.volume}, \ref{lem:regular.bounded.covernumber}, and
\ref{lem:regular.bounded.length}. Such properties are exploited in Section \ref{sec:upper} and \ref{sec:lower}. The proofs for  Lemma~\ref{lem:regular.bounded.volume}, \ref{lem:regular.bounded.covernumber}, and \ref{lem:regular.bounded.length} are in Appendix \ref{sec:proof.definition.regular}.

\begin{lem}
\label{lem:regular.bounded.volume}
Fix $\tau_{g}\in(0,\infty]$, and let $M$ be a $d$-dimensional manifold
with global reach $\geq\tau_{g}$. For $r\in(0,\tau_{g})$, let
$M_{r}:=\{x\in\mathbb{R}^{m}:\ dist_{\mathbb{R}^{m},}(x,M)<r\}$ be
an $r$-neighborhood of $M$ in $\mathbb{R}^{m}$. Then, the volume
of $M$ is upper bounded as 
\[
vol_{M}(M)\leq\frac{m!}{d!}r^{d-m}vol_{\mathbb{R}^{m}}(M_{r}).
\]
Further, fix $\tau_{\ell}\in(0,\infty]$, $K_{I}\in[1,\infty)$,
$K_{v}\in(0,2^{-m}]$, with $\tau_{g}\leq\tau_{\ell}$, and suppose
$M\in\mathcal{M}_{\tau_{g},\tau_{\ell},K_{I},K_{v}}^{d}$. Then the
volume of $M$ is upper bounded as
\[
vol_{M}(M) \leq C_{K_{I},m}^{(\ref*{lem:regular.bounded.volume})}\max\left\{1,\tau_{g}^{d-m}\right\},
\]
where $C_{K_{I},m}^{(\ref*{lem:regular.bounded.volume})}$
is a constant depending only on $K_{I}$ and $m$.
\end{lem}

\begin{lem}
\label{lem:regular.bounded.covernumber}
 Fix $\tau_{g},\,\tau_{\ell}\in(0,\infty]$, $K_{I}\in[1,\infty)$, $K_{v}\in(0,2^{-m}]$, with $\tau_{g}\leq\tau_{\ell}$. Let $M\in\mathcal{M}_{\tau_{g},\tau_{\ell},K_{I},K_{v}}^{d}$ and $r\in(0,2\sqrt{3}\tau_{g}]$. 
Then $M$ can be covered by $N$ radius $r$ balls $B_{M}(p_{1},r)$, $\ldots$, $B_{M}(p_{N},r)$, with 
\[
N \leq \left\lfloor \frac{2^{d}vol(M)}{K_{v} r^{d}\omega_{d}}\right\rfloor.
\]
\end{lem}

\begin{lem}
\label{lem:regular.bounded.length}
Fix $\tau_{g},\,\tau_{\ell}\in(0,\infty]$, $K_{I}\in[1,\infty)$, $K_{v}\in(0,2^{-m}]$, with $\tau_{g}\leq\tau_{\ell}$. Let $M\in\mathcal{M}_{\tau_{g},\tau_{\ell},K_{I},K_{v}}^{d}$ and let
$\exp_{p_{k}}:\mathcal{E}_{k} \subset \mathbb{R}^{m} \rightarrow {\mathcal{M}}$
be an exponential map, where  $\mathcal{E}_{k}$ is the domain of the exponential
map $\exp_{p_{k}}$ and $T_{p_{k}}M$ is identified with $\mathbb{R}^{d}$.
For all
$v,w\in \mathcal{E}_{k}$, let $R_{k}:=\max\{||v||,||w||\}$. Then 
\[
\|\exp_{p_{k}}(v)-\exp_{p_{k}}(w)\|_{\mathbb{R}^{m}}\leq
\frac{\sinh(\sqrt{2}R_{k}/\tau_{\ell})}{\sqrt{2}R_{k}/\tau_{\ell}}\|v-w\|_{\mathbb{R}^{d}}.
\]
\end{lem}

Under these regularity conditions,  the minimax risk $R_{n}$ is defined as 
\begin{equation}
\label{eq:regular.minimax}
R_n =
\inf_{\hat{d}_{n}}\sup_{P\in \mathcal{P}}
\mathbb{E}_{P^{(n)}}\left[1\left(\hat{d}_{n}(X) \neq d(P)\right)\right],
\end{equation}
where  in Section \ref{sec:upper} and \ref{sec:lower} we fix
$d_{1},d_{2}\in\mathbb{N}$ with $1\leq d_{1} < d_{2} \leq m$ and define
\begin{equation}
\label{eq:regular.distribution.binary}
\mathcal{P} = 
\mathcal{P}_{\tau_{g},\tau_{\ell},K_{I},K_{v},K_{p}}^{d_{1}}
\bigcup 
\mathcal{P}_{\tau_{g},\tau_{\ell},K_{I},K_{v},K_{p}}^{d_{2}},
\end{equation}
 and in Section \ref{sec:multidimension} we set instead
\begin{equation}
\label{eq:regular.distribution.multi}
\mathcal{P} = 
\overset{m}{\underset{d=1}{\bigcup}}\mathcal{P}_{\tau_{g},\tau_{\ell},K_{I},K_{v},K_{p}}^{d}.
\end{equation}

In \eqref{eq:regular.minimax}, $\hat{d}_{n}$  is any dimension estimator based on data
$X=(X_{1},\ldots,X_{n})$, and the loss function $\ell(\cdot,\cdot)$ is
$0-1$ loss, so for all $x,y\in\mathbb{R}$, $\ell(x,y)=1(x \neq y)$.

\vspace{11pt}

\section{Upper Bound for Choosing Between Two Dimensions}
\label{sec:upper}

In this section we provide an upper bound on the
minimax rate $R_{n}$ in \eqref{eq:regular.minimax}
when $d(P)$ can only take two known values. 
Fix
$d_{1},\,d_{2}\in\mathbb{N}$ with $1\leq d_{1}<d_{2}\leq m$, and assume that the
data are generated from a distribution $P \in \mathcal{P}$ such that either $d(P)
= d_{1}$ or $d(P) = d_{2}$ as in \eqref{eq:regular.distribution.binary}. In this case, the minimax risk
quantifies the statistical hardness of the hypothesis testing problem of
deciding whether the data originate from a $d_1$ or $d_2$-dimensional
distribution. In
Section \ref{sec:multidimension} we will relax this assumption and allow 
for the intrinsic dimension $d(P)$ to be any integer between $1$ and $m$ as in \eqref{eq:regular.distribution.multi}.
All the proofs for this section are in Section \ref{sec:proof.upper}.

Our strategy to derive an upper bound on $R_n$ is to choose a particular
estimator $\hat{d}_n$ and then derive a uniform upper bound on its risk over the class $\mathcal{P}$ in
    \eqref{eq:regular.distribution.binary}, i.e. an upper bound for the quantity
\begin{equation}
\label{eq:upper.maximumrisk}
\underset{P\in \mathcal{P}}{\sup}
\mathbb{E}_{P^{(n)}}\left[1\left(\hat{d}_{n}(X) \neq d(P)\right)\right],
\end{equation}
where $P^{(n)}$ denotes the $n$-fold product of $P$. 
This will in turn yield an upper bound on the minimax risk $R_{n}$, since
\begin{equation}
\label{eq:upper.maximumrisk.bound}
R_n =
\inf_{\hat{d}_{n}}\sup_{P\in \mathcal{P}}
\mathbb{E}_{P^{(n)}}\left[1\left(\hat{d}_{n}(X) \neq d(P)\right)\right] \leq
\sup_{P\in \mathcal{P}}
\mathbb{E}_{P^{(n)}}\left[1\left(\hat{d}_{n}(X) \neq d(P)\right)\right]. 
\end{equation}
%{\color{blue} since
%$$
%R_n =
%\inf_{\hat{\dim}^*_{n}}\sup_{P\in \mathcal{P}}
%\mathbb{E}_{P^{(n)}}\left[\ell\left(\hat{\dim}^*_{n}(X),\dim(P)\right)\right] \leq
%\sup_{P\in \mathcal{P}}
%\mathbb{E}_{P^{(n)}}\left[\ell\left(\hat{\dim}^*_{n}(X),\dim(P)\right)\right]. 
%$$
%}
%{\color{magenta}
%\begin{equation}
%\label{eq:upper.maximumrisk.bound}
%R_n =
%\inf_{\hat{d}_{n}}\sup_{P\in \mathcal{P}}
%\mathbb{E}_{P^{(n)}}\left[\ell\left(\hat{d}_{n}(X),d(P)\right)\right] \leq
%\sup_{P\in \mathcal{P}}
%\mathbb{E}_{P^{(n)}}\left[\ell\left(\hat{d}_{n}(X),d(P)\right)\right]. 
%\end{equation}
%}
Naturally, choosing an appropriate estimator is critical to get a sharp bound. In Section \ref{subsec:upper.estimator}, we define our dimension estimator $\hat{d}_{n}$ and analyze its risk. From that analysis, we derive an upper bound on the minimax risk $R_{n}$ in \eqref{eq:regular.minimax} in Section \ref{subsec:upper.minimax}.

\subsection{Dimension Estimator and its Analysis}
\label{subsec:upper.estimator}

Our dimension estimator $\hat{d}_{n}$ is based on the $d_{1}$-squared length of the TSP
(Traveling Salesman Path) generated by the data. The $d_{1}$-squared length of
the TSP generated by the data is the minimal $d_{1}$-squared length of
all possible paths passing through each sample point $X_{i}$ once, which is
\begin{equation}
\label{eq:upper.TSPpath}
\underset{\sigma\in S_{n}}{\min}\left\lbrace \overset{n-1}{\underset{i=1}{\sum}}\|X_{\sigma(i+1)}-X_{\sigma(i)}\|_{\mathbb{R}^{m}}^{d_{1}}\right\rbrace.
\end{equation}
Then, $\hat{d}_{n} = d_{1}$ if and only if the $d_{1}$-squared length of the TSP  is below a certain
threshold; that is
\begin{equation}
\label{eq:upper.dim.estimator}
\hat{d}_{n}(X):=\begin{cases}
d_{1}, & if\ \underset{\sigma\in S_{n}}{\min}\left\{ \overset{n-1}{\underset{i=1}{\sum}}\|X_{\sigma(i+1)}-X_{\sigma(i)}\|_{\mathbb{R}^{m}}^{d_{1}}\right\}\leq C_{K_{I},K_{v},m}^{(\ref*{lem:upper.lowerdim})}\max\left\{1,\tau_{g}^{d_{1}-m}\right\}, \\
d_{2}, & \text{otherwise.}
\end{cases}
\end{equation}
where $C_{K_{I},K_{v},m}^{(\ref*{lem:upper.lowerdim})}$ is a constant to be defined later. 

%%{\color{blue}Then in Proposition~\ref{prop:upper.bound}, it is shown
%%  that this estimator $\hat{\dim}_{n}$ is always correct when the
%%  intrinsic dimension is $d_{1}$, and makes error with probability at
%%  most $O\left(n^{-\left(\frac{d_{2}}{d_{1}}-1\right)n}\right)$ if
%%  intrinsic dimension is $d_{2}$. }

%\textcolor{red}{THE FOLLOWING PARAGRAPH SEEM TO BE REPEATED BELOW. PLEASE CHECK.
%Lemma~\ref{lem:upper.lowerdim} shows
%that the estimator $\hat{d}_{n}$ in \eqref{eq:upper.dim.estimator}
%makes an error with probability of order
%$O\left(n^{-\left(\frac{d_{2}}{d_{1}}-1\right)n}\right)$ if the intrinsic dimension
%is $d_{2}$.
%Also, $\hat{d}_{n}$ is
%always correct when the intrinsic dimension is $d_{1}$.
%Combining these lemmas, the maximum risk of
%$\hat{d}_{n}$ 
%is of order $O\left(n^{-\left(\frac{d_{2}}{d_{1}}-1\right)n}\right)$, 
%as shown in Proposition~\ref{prop:upper.bound}.}

We begin our analysis of the estimator $\widehat{d}_n$ with Lemma
\ref{lem:upper.higherdim}, which shows that 
$\hat{d}_{n}$ makes an error with probability
of order
$O\left(n^{-\left(\frac{d_{2}}{d_{1}}-1\right)n}\right)$ if the
correct dimension is $d_{2}$. 
Specifically, we demonstrate that, for any positive value $L$, the $d_{1}$-squared
length of a piecewise linear path from $X_{1}$ to $X_{n}$,
$\overset{n-1}{\underset{i=1}{\sum}}\|X_{i+1}-X_{i}\|_{\mathbb{R}^{m}}^{d_{1}}$,
is upper bounded by $L$ with
a very small probability of order
$O\left(n^{-\left(\frac{d_{2}}{d_{1}}-1\right)n}\right)$,
as in \eqref{eq:upper.path.upperdimension}. Hence the
$d_{1}$-squared length of the path is not likely to be bounded by
any such threshold $L$.

\begin{lem}
\label{lem:upper.higherdim}
Fix $\tau_{g},\,\tau_{\ell}\in(0,\infty]$, $K_{I}\in[1,\infty)$,
$K_{v}\in(0,2^{-m}]$, $K_{p}\in[(2K_{I})^{m},\infty)$,
$d_{1},\,d_{2}\in\mathbb{N}$, with $\tau_{g}\leq\tau_{\ell}$ and
$1\leq d_{1} < d_{2} \leq m$. Let $X_{1},\ldots,X_{n}\sim
P\in\mathcal{P}_{\tau_{g},\tau_{\ell},K_{I},K_{v},K_{p}}^{d_{2}}$.
Then for all $L>0$,
\begin{equation}
\label{eq:upper.path.upperdimension}
P^{(n)}\left[\overset{n-1}{\underset{i=1}{\sum}}\|X_{i+1}-X_{i}\|_{\mathbb{R}^{m}}^{d_{1}}\leq L\right]\leq\frac{\left(C_{K_{I},K_{p},m}^{(\ref*{lem:upper.higherdim})}\right)^{n-1}L^{\frac{d_{2}}{d_{1}}(n-1)}\max\left\{1,\tau_{g}^{(d_{2}-m)(n-1)}\right\}}{(n-1)^{\left(\frac{d_{2}}{d_{1}}-1\right)(n-1)}(n-1)!},
\end{equation}
where $C_{K_{I},K_{p},m}^{(\ref*{lem:upper.higherdim})}$ is a constant depending only on $K_{I}$, $K_{p}$, and $m$.

\end{lem}
\begin{proof}[Proof of Lemma~\ref{lem:upper.higherdim}]
in Appendix \ref{sec:proof.upper}.
\end{proof}
Next, Lemma~\ref{lem:upper.lowerdim} shows
that the estimator
$\hat{d}_{n}$ in
\eqref{eq:upper.dim.estimator} is always correct when the intrinsic dimension is $d_{1}$, as in \eqref{eq:upper.travel}. 
Specifically, the $d_{1}$-squared length of the TSP path 
in \eqref{eq:upper.TSPpath} is bounded
by some positive threshold 
$C_{K_{I},K_{v},m}^{(\ref*{lem:upper.lowerdim})}\max\left\{1,\tau_{g}^{d_{1}-m}\right\}$.
We take note that, when $d_{1}=1$, Lemma
\ref{lem:upper.lowerdim} 
is straightforward: the length of the TSP path 
in \eqref{eq:upper.TSPpath} is 
upper bounded by
the length of curve $vol_{M}(M)$, 
as in Figure \ref{fig:upper.curve}. This fact, combined with
Lemma~\ref{lem:regular.bounded.volume}, which shows that $vol_{M}(M)\leq C_{K_{I},m}^{(\ref*{lem:regular.bounded.volume})}\max\left\{1,\tau_{g}^{1-m}\right\}$, yields the result.
In particular, the constant $C_{K_{I},K_{v},m}^{(\ref*{lem:upper.lowerdim})}$ can be set as $C_{K_{I},K_{v},m}^{(\ref*{lem:upper.lowerdim})}=C_{K_{I},m}^{(\ref*{lem:regular.bounded.volume})}$.

\begin{figure}
\begin{center}
\includegraphics{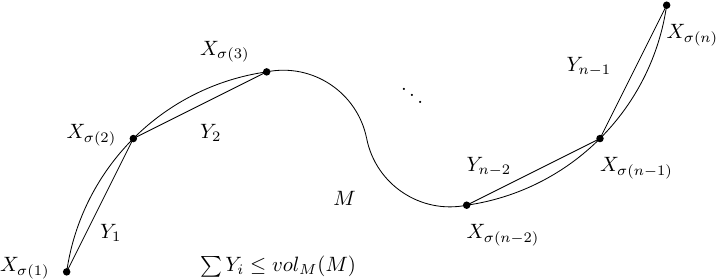}
\end{center}
\caption{When the manifold is a curve, the length of the TSP path
$\underset{\sigma \in S_{n}}{\min}\left\lbrace
\overset{n-1}{\underset{i=1}{\sum}}\|X_{\sigma(i+1)}-X_{\sigma(i)}\|_{\mathbb{R}^{m}}\right\rbrace$
in \eqref{eq:upper.TSPpath} is upper bounded by the length of
the curve $vol_{M}(M)$.}
\label{fig:upper.curve}
\end{figure}

When $d_{1}>1$, Lemma
\ref{lem:upper.lowerdim} is proved using Lemma~\ref{lem:regular.bounded.volume},
\ref{lem:regular.bounded.covernumber} and
\ref{lem:regular.bounded.length}, along with the H\"{o}lder continuity of
a $d_{1}$-dimensional space-filling curve
\citep{Steele1997.ch2,Buchin2008.ch2}.

%%{\color{blue}
%%When $d_{1}>1$, Lemma~\ref{lem:regular.bounded.volume},
%%\ref{lem:regular.bounded.covernumber}, and
%%\ref{lem:regular.bounded.length}, combined with H\"{o}lder continuity of
%%$d_{1}$-dimensional space-filling curve
%%\citep{Steele1997.ch2,Buchin2008.ch2}, is used to show Lemma
%%\ref{lem:upper.lowerdim}.
%%}

%\begin{lem*}
%5. Let $M\in\mathcal{M}_{\tau_{g},\tau_{\ell},K_{I},K_{v}}^{d_{1}}$ and $U_{p_{k}}$
%be a normal neighborhood of $p_{k}$ with chart map $\varphi_{k}:U_{p_{k}}\rightarrow\mathbb{R}^{d_{1}}$.
%Then there exists $K'_{\kappa_{g},K_{I}}>0$ which depends only on
%$\kappa_{g}$ and $K_{I}$ such that for all $p,q\in U_{k}$, 
%\begin{equation}
%\|p-q\|_{\mathbb{R}^{m}}^{d_{1}}\leq K'_{\kappa_{g},K_{I}}\|
%\varphi_{k}(p)-\varphi_{k}(q)\|_{\mathbb{R}^{d_{1}}}^{d_{1}}.
%\end{equation}
%\end{lem*}
%\begin{proof}
%in Appendix \ref{sec:proof.upper}.\end{proof}

\begin{lem}
\label{lem:upper.lowerdim}
Fix $\tau_{g},\,\tau_{\ell}\in(0,\infty]$, $K_{I}\in[1,\infty)$,
$K_{v}\in(0,2^{-m}]$, $d_{1}\in\mathbb{N}$, with
$\tau_{g}\leq\tau_{\ell}$. Let
$M\in\mathcal{M}_{\tau_{g},\tau_{\ell},K_{p},K_{v}}^{d_{1}}$ and
$X_{1},\ldots,X_{n}\in M$. Then 
\begin{equation}
\label{eq:upper.travel}
\underset{\sigma\in S_{n}}{\min} \overset{n-1}{\underset{i=1}{\sum}}
\|X_{\sigma(i+1)}-X_{\sigma(i)}\|_{\mathbb{R}^{m}}^{d_{1}} \leq 
C_{K_{I},K_{v},m}^{(\ref*{lem:upper.lowerdim})}\max\left\{1,\tau_{g}^{d_{1}-m}\right\},
\end{equation}
where $C_{K_{I},K_{v},m}^{(\ref*{lem:upper.lowerdim})}$ is a constant depending only on $K_{I}$, $K_{v}$, and $m$.
\end{lem}

\begin{proof}[Proof of Lemma~\ref{lem:upper.lowerdim}]
in Appendix \ref{sec:proof.upper}.
\end{proof}

Proposition~\ref{prop:upper.maximumrisk} below is the main result of this subsection and
follows directly from Lemma
\ref{lem:upper.higherdim} and Lemma
\ref{lem:upper.lowerdim} above. 
Indeed, when the intrinsic dimension
is $d_{2}$, the risk of our estimator $\hat{d}_{n}$,
is of order
$O\left(n^{-\left(\frac{d_{2}}{d_{1}}-1\right)n}\right)$ by
Lemma~\ref{lem:upper.higherdim} and the union bound. On the other hand, when the
intrinsic dimension is $d_{1}$, the risk of our estimator
$\hat{d}_{n}$ is $0$, because of Lemma~\ref{lem:upper.lowerdim}.

\begin{prop}
	\label{prop:upper.maximumrisk}
	Fix $\tau_{g},\,\tau_{\ell}\in(0,\infty]$, $K_{I}\in[1,\infty)$,
	$K_{v}\in(0,2^{-m}]$, $K_{p}\in[(2K_{I})^{m},\infty)$,
	$d_{1},\,d_{2}\in\mathbb{N}$, with $\tau_{g}\leq\tau_{\ell}$ and
	$1\leq d_{1} < d_{2} \leq m$. Let $\hat{d}_{n}$ be in \eqref{eq:upper.dim.estimator}. Then either for $d=d_{1}$ or $d=d_{2}$,
	\begin{align*}
	& \underset{P\in\mathcal{P}_{\tau_{g},\tau_{\ell},K_{I},K_{v},K_{p}}^{d}}{\sup}
	\mathbb{E}_{P^{(n)}}\left[\ell\left(\hat{d}_{n},d(P)\right)\right] \\
	& \leq 1(d=d_{2})\left(C_{K_{I},K_{p},K_{v},m}^{(\ref*{prop:upper.maximumrisk})}\right)^{n}
	\max\left\{1,\tau_{g}^{-\left(\frac{d_{2}}{d_{1}}m+m-2d_{2}\right)n}\right\}n^{-\left(\frac{d_{2}}{d_{1}}-1\right)n},
	\end{align*}
	where $C_{K_{I},K_{p},K_{v},m}^{(\ref*{prop:upper.maximumrisk})}\in(0,\infty)$ is a constant depending only on $K_{I}$, $K_{p}$, $K_{v}$, and $m$.
\end{prop}
\begin{proof}[Proof of Proposition~\ref{prop:upper.maximumrisk}]
	in Appendix \ref{sec:proof.upper}.
\end{proof}

As described so far, the convergence analysis of our dimension estimator is probable. This is enough for our purpose, which is to quantify the statistical difficulties, in particular the minimax rate, of the dimension estimation problem. However, our $\widehat{d}_n$ in \eqref{eq:upper.dim.estimator} is not completely data-driven but depends on the model parameters $\tau_{g}$, $K_{I}$, and $K_{v}$. Hence the model on which our convergence analysis is valid depends on the model parameters. When it comes to applying our dimension estimator $\widehat{d}_n$ to real data, we need to estimate the constant $C_{K_{I},K_{v},m}^{(\ref*{lem:upper.lowerdim})}$. Proofs of Lemma~\ref{lem:upper.higherdim} and \ref{lem:upper.lowerdim} suggest that overestimating $C_{K_{I},K_{v},m}^{(\ref*{lem:upper.lowerdim})}$ by some constant factor doesn't deteriorate the convergence rate, so the constants $C_{K_{I},K_{v},m}^{(\ref*{lem:upper.lowerdim})}$ and $\tau_{g}$ can be replaced by any consistent estimators. Still, we have the difficulty of tuning the constant $C_{K_{I},K_{v},m}^{(\ref*{lem:upper.lowerdim})}$ and $\tau_{g}$. Also, the constant $C_{K_{I},K_{v},m}^{(\ref*{lem:upper.lowerdim})}$ is tuned to work for the worst case, so the practical performance of our dimension estimator is questionable.

\subsection{Minimax Upper Bound}
\label{subsec:upper.minimax}

As noted at the beginning of Section \ref{sec:upper}, the maximum risk of our estimator $\hat{d}_{n}$ in
	\eqref{eq:upper.maximumrisk} serves as an upper bound on the minimax risk $R_{n}$ in \eqref{eq:regular.minimax}.
Since we
assume that the intrinsic dimension is either $d_{1}$ or $d_{2}$,
Proposition~\ref{prop:upper.maximumrisk} yields that
the maximum risk of our estimator $\hat{d}_{n}$  is of order
$O\left(n^{-\left(\frac{d_{2}}{d_{1}}-1\right)n}\right)$. This
also serves as an upper bound of the minimax risk $R_{n}$, as in Proposition~\ref{prop:upper.bound}.

\begin{prop}
\label{prop:upper.bound}
Fix $\tau_{g},\,\tau_{\ell}\in(0,\infty]$, $K_{I}\in[1,\infty)$,
$K_{v}\in(0,2^{-m}]$, $K_{p}\in[(2K_{I})^{m},\infty)$,
$d_{1},\,d_{2}\in\mathbb{N}$, with $\tau_{g}\leq\tau_{\ell}$ and
$1\leq d_{1} < d_{2} \leq m$. Then
\begin{align*}
 & \underset{\hat{d}_{n}}{\inf}
\underset{P\in\mathcal{P}_{1}\cup\mathcal{P}_{2}}{\sup}
\mathbb{E}_{P^{(n)}}\left[\ell\left(\hat{d}_{n},d(P)\right)\right] \\
 & \leq \left(C_{K_{I},K_{p},K_{v},m}^{(\ref*{prop:upper.maximumrisk})}\right)^{n}
\max\left\{1,\tau_{g}^{-\left(\frac{d_{2}}{d_{1}}m+m-2d_{2}\right)n}\right\}n^{-\left(\frac{d_{2}}{d_{1}}-1\right)n},
\end{align*}
where  $C_{K_{I},K_{p},K_{v},m}^{(\ref*{prop:upper.maximumrisk})}$ is from Proposition~\ref{prop:upper.maximumrisk} and
\[
\mathcal{P}_1 =
\mathcal{P}_{\tau_{g},\tau_{\ell},K_{I},K_{v},K_{p}}^{d_{1}},\ \ \ 
\mathcal{P}_2 =
\mathcal{P}_{\tau_{g},\tau_{\ell},K_{I},K_{v},K_{p}}^{d_{2}}.
\]
\end{prop}

\begin{proof}[Proof of Proposition~\ref{prop:upper.bound}]
in Appendix \ref{sec:proof.upper}.
\end{proof}

\section{Lower Bound for Choosing Between Two Dimensions}
\label{sec:lower}

The goal of this section is to derive a lower bound for the minimax rate $R_{n}$.
As in Section
\ref{sec:upper}, we fix $d_{1},\,d_{2}\in\mathbb{N}$ with $1\leq
d_{1}<d_{2}\leq m$, and assume that the intrinsic dimension of data is either $d_{1}$ or $d_{2}$ as in
\eqref{eq:regular.distribution.binary}. This assumption is relaxed 
in Section \ref{sec:multidimension}.
All the proofs for this section are in Section \ref{sec:proof.lower}.

Our strategy is to find a subset $T\subset
I^{n}\subset (\mathbb{R}^{d})^{n}$ and two sets of distributions
$\mathcal{P}_{1}^{d_{1}}$ and $\mathcal{P}_{2}^{d_{2}}$ with 
dimensions $d_{1}$
and $d_{2}$, such that $\mathcal{P}_{1}^{d_{1}}$ and
$\mathcal{P}_{2}^{d_{2}}$ satisfy the regularity conditions in
Definition \ref{def:regular.manifold}, and whenever the sample
$X=(X_{1},\ldots,X_{n})$ lies on $T$, one cannot easily distinguish
whether the underlying distribution is from
$\mathcal{P}_{1}^{d_{1}}$ or $\mathcal{P}_{2}^{d_{2}}$. 

After constructing $T$, $\mathcal{P}_{1}^{d_{1}}$ and $\mathcal{P}_{2}^{d_{2}}$,
we derive the lower bound using
the following result, known as Le Cam's lemma.

\begin{lem} 
\label{lem:lower.LeCam}
{\em (Le Cam's Lemma)} Let $\mathcal{P}$ be a set of probability measures on
$(\Omega,\mathcal{F})$, and
$\mathcal{P}_{1},\mathcal{P}_{2}\subset\mathcal{P}$ be such that for
all $P\in\mathcal{P}_{i}$, $\theta(P)=\theta_{i}$ for $i=1,2$. For any
$Q_{i}\in co(\mathcal{P}_{i})$, where $co(\mathcal{P}_{i})$ is the convex hull of $\mathcal{P}_{i}$, let $q_{i}$ be the density of $Q_{i}$ with respect to a measure $\nu$. Then 
\begin{equation}
\label{eq:lower.lecam}
\underset{\hat{\theta}}{\inf}\underset{P\in\mathcal{P}}{\sup}\mathbb{E}_{P}[\ell(\hat{\theta},\theta(P))]\geq\frac{\ell(\theta_{1},\theta_{2})}{2}\int[q_{1}(x)\wedge q_{2}(x)]d\nu(x).
\end{equation}
\end{lem}
\begin{proof}[Proof of Lemma~\ref{lem:lower.LeCam}]
\citep[See][Chapter 29.2, Lemma 1]{Yu1997}.
\end{proof}

In above Le Cam's lemma, considering the convex hull of distributions $co(\mathcal{P}_{i})$ is critical for getting the nontrivial lower bound. Suppose we are using the basic version of Le Cam's lemma where the convex hull is not considered, i.e. $Q_i \in \mathcal{P}_{i}$. Then for two distributions $Q_{1}$ and $Q_{2}$ respectively from our $d_{1}$ and $d_{2}$ dimensional model $\mathcal{P}^{d_{1}}_{\tau_{g},\tau_{l},K_{I},K_{v},K_{p}}$ and $\mathcal{P}^{d_{2}}_{\tau_{g},\tau_{l},K_{I},K_{v},K_{p}}$, $Q_{1}$ and $Q_{2}$ are singular to each other; i.e. $q_{1}(x)\wedge q_{2}(x) = 0$ for all $x$. Hence no matter which subset $\mathcal{P}_{1}$ and $\mathcal{P}_{2}$ we choose with $d(\mathcal{P}_{1})=d_{1}$ and $d(\mathcal{P}_{2})=d_{2}$, the lower bound in \eqref{eq:lower.lecam} will be always $0$. This trivial bound can be improved by considering the convex hull of distributions $co(\mathcal{P}_{i})$ in Le Cam's lemma.

Our construction for $T$, $\mathcal{P}_{1}^{d_{1}}$, and
$\mathcal{P}_{2}^{d_{2}}$ is based on mimicking a space-filling
curve. Intuitively, this gives the lower bound
since it is difficult to differentiate a
space-filling curve and a higher dimensional
cube. In detail, we set % $\mathcal{P}_{1}^{d_{1}}$ and $\mathcal{P}_{1}^{d_{2}}$ will be defined as:
\begin{align}
\label{eq:lower.distribution.lower}
\mathcal{P}_{1}^{d_{1}}=\{ & \text{distributions supported on}\nonumber\\
&\text{a space-filling-curve like }d_{1}\text{-dimensional manifold}\},
\end{align}
and
\begin{equation}
\label{eq:lower.distribution.upper}
\mathcal{P}_{2}^{d_{2}}=\{\text{uniform distributions on }[-K_{I},\,K_{I}]^{d_{2}}\}.
\end{equation}
To apply Le Cam's lemma, we construct a set $T\subset
I^{n}$  so that, whenever $X=(X_{1},\ldots,X_{n})\in{T}$, we cannot distinguish whether $X$ is
from $\mathcal{P}_{1}^{d_{1}}$ in \eqref{eq:lower.distribution.lower}
or $\mathcal{P}_{1}^{d_{2}}$ in
\eqref{eq:lower.distribution.upper}. Then, for an  appropriately chosen
distribution $Q_{1}$ in the convex hull of $\mathcal{P}_{1}^{d_{1}}$ with
density $q_{1}$ with respect to Lebesgue measure $\lambda$ on the cube $[-K_{I},K_{I}]^{d_{2}}$, and
a density $q_{2}$ from the class
$\mathcal{P}_{2}^{d_{2}}$, $\int_{T}[q_{1}(x)\wedge
  q_{2}(x)]d\lambda(x)$ is a lower bound on the minimax rate $R_{n}$ in
\eqref{eq:regular.minimax}. Indeed, from Le Cam's Lemma~\ref{lem:lower.LeCam}, we have that
\begin{align}
\underset{\hat{\theta}}{\inf}\underset{P\in\mathcal{P}}{\sup}\mathbb{E}_{P}[\ell(\hat{\theta},\theta(P))]& \geq\frac{1}{2}\int[q_{1}(x)\wedge q_{2}(x)]d\lambda(x) \nonumber\\
& \geq\frac{1}{2}\int_{T}[q_{1}(x)\wedge q_{2}(x)]d\lambda(x).
\label{eq:lower.integrationT}
\end{align}

For constructing the class $\mathcal{P}_{1}^{d_{1}}$ in
\eqref{eq:lower.distribution.lower}, it will be sufficient to consider the case
$d_{1} = 1$.
In fact, Lemma
\ref{lem:lower.cylinder.regularity} states that the regularity
conditions in Definition \ref{def:regular.manifold} are still
preserved when the manifold $M$ is  a Cartesian product with a cube
$[-K_{I},\,K_{I}]^{\Delta d}$, as in Figure \ref{fig:lower.product}. 
Hence for constructing a $d$-dimensional ``space-filling" manifold, we first construct a $1$-dimensional space-filling curve satisfying the required regularity conditions, and then we form a Cartesian product with a cube of dimension $d-1$, which becomes a $d$-dimensional manifold satisfying the same regularity conditions by Lemma~\ref{lem:lower.cylinder.regularity}.

\begin{figure}
	\begin{center}
		\includegraphics{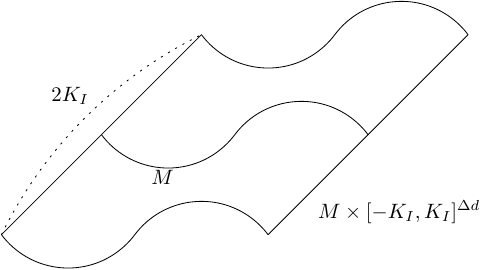}
	\end{center}
	\caption{The regularity conditions in Definition \ref{def:regular.manifold} are still preserved under the Cartesian product with a cube $[-K_{I},\,K_{I}]^{\Delta d}$. Detailed explanations are in Figure \ref{fig:proof.lower.product}.}
	\label{fig:lower.product}
\end{figure}

\begin{lem}
\label{lem:lower.cylinder.regularity}
Fix $\tau_{g},\,\tau_{\ell}\in(0,\infty]$,
$K_{I}\in[1,\infty)$, $K_{v}\in(0,2^{-m}]$, $d,\,\Delta
d\in\mathbb{N}$, with $\tau_{g}\leq\tau_{\ell}$ and $1\leq d+\Delta d
\leq m$. Let $M\in\mathcal{M}_{\tau_{g},\tau_{\ell},K_{I},K_{v}}^{d}$
be a $d$-dimensional manifold of global reach $\geq\tau_{g}$, local
reach $\geq\tau_{\ell}$, which is embedded in $\mathbb{R}^{m-\Delta d}$.
Then
\[
M\times[-K_{I},K_{I}]^{\Delta d}\in\mathcal{M}_{\tau_{g},\tau_{\ell},K_{I},K_{v}}^{d+\Delta d},
\]
which is embedded in $\mathbb{R}^{m}$.\end{lem}

\begin{proof}[Proof of Lemma~\ref{lem:lower.cylinder.regularity}]
in Appendix \ref{sec:proof.lower}.
\end{proof}

The precise construction of $\mathcal{P}_{1}^{d_{1}}$ in
\eqref{eq:lower.distribution.lower} and $T$ is detailed in Lemma~\ref{lem:lower.constructT}.
As in Figure \ref{fig:lower.Tconstruct}, we construct $T_{i}$'s that are
cylinder sets aligned 
as a zigzag in $[-K_{I},K_{I}]^{d_{2}}$, and then
$T$ is constructed as
$T=S_{n}\overset{n}{\underset{i=1}{\prod}}T_{i}$, where the
permutation group $S_{n}$ acts on
$\overset{n}{\underset{i=1}{\prod}}T_{i}$ as a coordinate
change.
Then, we show below that, for
any $x\in\prod T_{i}$, there exists a manifold
$M\in \mathcal{M}_{\tau_{g},\tau_{\ell},K_{I},K_{v}}^{d_{1}}$ that passes
through $x_{1},\ldots,x_{n}$. 
The class
$\mathcal{P}_{1}^{d_{1}}$ in \eqref{eq:lower.distribution.lower}
is finally defined as the set of distributions that are supported on such a
manifold.

\begin{figure}
	\begin{center}
		\begin{subfigure}[b]{0.5\textwidth}		
			\begin{center}
			\includegraphics[scale=0.95]{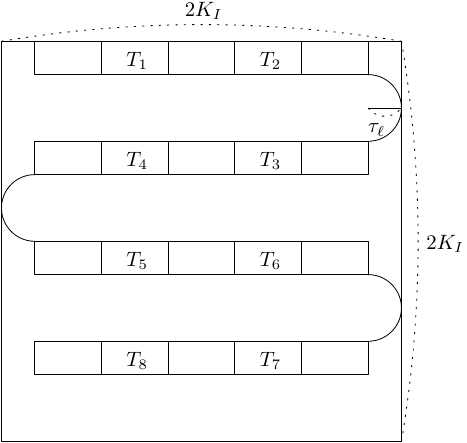}
			\end{center}
			\caption{alignment of $T_{i}$}
			\label{subfig:lower.ARTconstruct}
		\end{subfigure}
		\begin{subfigure}[b]{0.45\textwidth}
			\begin{center}
			\includegraphics[scale=0.95]{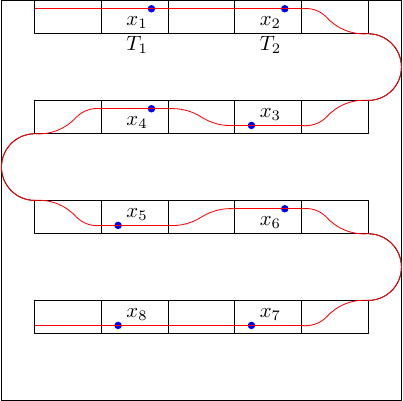}
			\end{center}
			\caption{manifold passing through $x_{i}$'s}
			\label{subfig:lower.Tconstruct.manifold}
		\end{subfigure}
	\end{center}
	
	\caption{This figure illustrates the case where $d_{1}=1$ and $d_{2}=2$. \protect\subref{subfig:lower.ARTconstruct}
		shows how $T_{i}$'s are aligned in a zigzag.
		\protect\subref{subfig:lower.Tconstruct.manifold} shows for given $x_{1}\in T_{1},\ldots,x_{n}\in
		T_{n}$(represented as blue points), how a manifold with regularity conditions(represented as a red curve) passes through $x_{1},\ldots,x_{n}$. Detailed constructions in Figure \ref{fig:proof.lower.Tconstruct}.}
	\label{fig:lower.Tconstruct}
\end{figure}

%%{\color{blue}
%%Then in Proposition~\ref{prop:lower.bound}, $\mathcal{P}_{1}$ 
%%is constructed as set of distributions that are 
%%supported on such a manifold, and $\mathcal{P}_{2}$ 
%%is a singleton set consisting of the uniform distribution on $[-K_{I},K_{I}]^{d_{2}}$.
%%}

\begin{lem}
\label{lem:lower.constructT}
Fix $\tau_{\ell}\in(0,\infty]$, $K_{I}\in[1,\infty)$,
$d_{1},\,d_{2}\in\mathbb{N}$, with $1\leq d_{1}\leq d_{2}$, and
suppose  $\tau_{\ell} < K_{I}$. Then there exist
$T_{1},\ldots,T_{n}\subset[-K_{I},K_{I}]^{d_{2}}$ such that:

(1) The $T_{i}$'s are distinct.

(2) For each $T_{i}$, there exists an isometry $\Phi_{i}$ such that
\[
T_{i}=\Phi_{i}\left([-K_{I},K_{I}]^{d_{1}-1}\times[0,a]\times B_{\mathbb{R}^{d_{2}-d_{1}}}(0,w)\right),
\]
where $c=\left\lceil \frac{K_{I}+\tau_{\ell}}{2\tau_{\ell}}\right\rceil $,  $a=\frac{K_{I}-\tau_{\ell}}{\left(d_{2}-d_{1}+\frac{1}{2}\right)\left\lceil \frac{n}{c^{d_{2}-d_{1}}}\right\rceil }$,
and  $w=\min\left\{ \tau_{\ell},\ \frac{(d_{2}-d_{1})^{2}(K_{I}-\tau_{\ell})^{2}}{2\tau_{\ell}\left(d_{2}-d_{1}+\frac{1}{2}\right)^{2}\left(\left\lceil \frac{n}{c^{d_{2}-d_{1}}}\right\rceil +1\right)^{2}}\right\} $.

(3)There exists $\mathscr{M}:\left(B_{\mathbb{R}^{d_{2}-d_{1}}}(0,w)\right)^{n}\rightarrow\mathcal{M}_{\tau_{g},\tau_{\ell},K_{I},K_{v}}^{d_{1}}$
one-to-one such that for each $y_{i}\in B_{\mathbb{R}^{d_{2}-d_{1}}}(0,w)$, $1\leq i\leq n$, $\mathscr{M}(y_{1},\ldots,y_{n})\cap T_{i}=\Phi_{i}([-K_{I},K_{I}]^{d_{1}-1}\times[0,a]\times\{y_{i}\})$.
Hence for any $x_{1}\in T_{1},\ldots,x_{n}\in T_{n}$, $\mathscr{M}(\{\Pi_{(d_{1}+1):d_{2}}^{-1}\Phi_{i}^{-1}(x_{i})\}_{1\leq i\leq n})$
passes through $x_{1},\ldots,x_{n}$.
\end{lem}

\begin{proof}[Proof of Lemma~\ref{lem:lower.constructT}]
in Appendix \ref{sec:proof.lower}.
\end{proof}

%Claim~\ref{claim:lower.probability.bound} 
Next we show that whenever
 $x=(x_{1},\ldots,x_{n})\in T$, it is difficult to tell whether
the data originated from
 $P\in \mathcal{P}_{1}^{d_{1}}$ or
 $P\in\mathcal{P}_{2}^{d_{2}}$. 
 Let $Q_{1}$ be in the convex hull of $\mathcal{P}_{1}^{d_{1}}$ and let $q_{2}$ be the density function of the uniform distribution on $[-K_{I},K_{I}]^{d_{2}}$, then from \eqref{eq:lower.integrationT}, we know that  a lower bound
is given by $\int_{T}[q_{1}(x)\wedge q_{2}(x)]d\lambda(x)$. Hence if we can choose $Q_{1}$ such that
$q_{1}(x)\geq Cq_{2}(x)$ for every $x\in T$ with $C<1$, then
$q_{1}(x)\wedge q_{2}(x)\geq Cq_{2}(x)$, so that $C\int_{T} q_{2}(x)$
can serve as a lower bound of the minimax rate. Such existence of $Q_{1}$ and the inequality
$q_{1}(x)\geq Cq_{2}(x)$ is shown in Claim
\ref{claim:lower.probability.bound}.

\begin{claim}
\label{claim:lower.probability.bound}
Let
$T=S_{n}\overset{n}{\underset{i=1}{\prod}}T_{i}$ where the $T_{i}$'s
are from Lemma~\ref{lem:lower.constructT}. Let $Q_{2}$ be the uniform distribution on $[-K_{I},K_{I}]^{d_{2}}$, and let $\mathcal{P}_{1}^{d_{1}}$ be as in \eqref{eq:lower.distribution.lower}. Then there exists $Q_{1} \in co(\mathcal{P}_{1}^{d_{1}})$ satisfying that for all $x\in intT$, there exists $r_{x}>0$ such that for
all $r<r_{x}$, 
\[
Q_{1}\left(\overset{n}{\underset{i=1}{\prod}}B_{\|\cdot\|_{\mathbb{R}^{d_{2}},\infty}}(x_{i},r)\right)\geq 2^{-n}Q_{2}\left(\overset{n}{\underset{i=1}{\prod}}B_{\|\cdot\|_{\mathbb{R}^{d_{2}},\infty}}(x_{i},r)\right).
\]
\end{claim}

\begin{proof}[Proof of Claim~\ref{claim:lower.probability.bound}]
in Appendix \ref{sec:proof.lower}.
\end{proof}

The following lower bound is then a consequence of Le Cam's lemma, 
Lemma~\ref{lem:lower.constructT}, and the previous claim. %Claim~\ref{claim:lower.probability.bound}.

\begin{prop}
\label{prop:lower.bound}
Fix $\tau_{g},\,\tau_{\ell}\in(0,\infty]$, $K_{I}\in[1,\infty)$,
$K_{v}\in(0,2^{-m}]$, $K_{p}\in[(2K_{I})^{m},\infty)$,
$d_{1},\,d_{2}\in\mathbb{N}$, with $\tau_{g}\leq\tau_{\ell}$ and
$1\leq d_{1} < d_{2} \leq m$, and suppose that $\tau_{\ell}<K_{I}$.
Then
\begin{align*}
& \underset{\hat{d}_{n}}{\inf}
\underset{P\in {\cal Q}}{\sup}
\mathbb{E}_{P^{(n)}}[\ell(\hat{d}_{n},d(P))]\\
& \geq\left(C_{d_{1},d_{2},K_{I}}^{(\ref*{prop:lower.bound})}\right)^{n}\min\left\{ \tau_{\ell}^{-2(d_{2}-d_{1}+1)}n^{-2},1\right\} ^{(d_{2}-d_{1})n},
\end{align*}
where $C_{d_{1},d_{2},K_{I}}^{(\ref*{prop:lower.bound})}\in(0,\infty)$ is a constant depending only on $d_{1}$, $d_{2}$, and $K_{I}$
and
\[
{\cal Q} = 
\mathcal{P}_{\tau_{g},\tau_{\ell},K_{I},K_{v},K_{p}}^{d_{1}}
\bigcup\mathcal{P}_{\tau_{g},\tau_{\ell},K_{I},K_{v},K_{p}}^{d_{2}}.
\]
\end{prop}

\begin{proof}[Proof of Proposition~\ref{prop:lower.bound}]
in Appendix \ref{sec:proof.lower}.
\end{proof}

\section{Upper Bound and Lower Bound for the General Case}
\label{sec:multidimension}

Now we generalize our results 
to allow the intrinsic
dimension $d$ to be any integer between 1 and
$m$. Thus the model is
$\mathcal{P}=\overset{m}{\underset{d=1}{\bigcup}}
\mathcal{P}_{\tau_{g},\tau_{\ell},K_{I},K_{v},K_{p}}^{d}$
as in \eqref{eq:regular.distribution.multi}.  
For the upper bound, we extend the dimension estimator $\hat{d}_{n}$ in \eqref{eq:upper.dim.estimator} and compute its maximum risk.
 And for the lower bound, we simply use the lower bound derived in Section
\ref{sec:lower} with $d_1 = 1$ and $d_2 = 2$.
All the proofs for this section are in Section \ref{sec:proof.multidimension}.

For the model $\mathcal{P}$ in \eqref{eq:regular.distribution.multi}, our dimension estimator $\hat{d}_{n}$ estimates the dimension as the smallest integer $1 \leq d \leq m$ that the $d$-squared length of the TSP is below a certain threshold, i.e. \eqref{eq:upper.travel} holds; that is,
\begin{align}
\label{eq:multidimension.estimator}
\hat{d}_{n}(X):=\min\Biggl\{ &  d\in[1,m]:\ \nonumber \\
& \underset{\sigma\in S_{n}}{\min}\biggl\{\overset{n-1}{\underset{i=1}{\sum}}\|X_{\sigma(i+1)}-X_{\sigma(i)}\|_{\mathbb{R}^{m}}^{d}\biggr\}\leq C_{K_{I},K_{v},m}^{(\ref*{lem:upper.lowerdim})}\max\left\{1,\tau_{g}^{d-m}\right\}\Biggr\}.
\end{align}
As a generalized result of Proposition~\ref{prop:upper.maximumrisk}, Proposition~\ref{prop:multidimension.maximumrisk} gives an upper bound for the risk of our estimator $\hat{d}_{n}$ in \eqref{eq:multidimension.estimator}. When the intrinsic dimension is $d$, our estimator $\hat{d}_{n}$ makes an error with probability of order $O\left(n^{-\frac{1}{d-1}n}\right)$.
\begin{prop}
	\label{prop:multidimension.maximumrisk}
	Fix $\tau_{g},\,\tau_{\ell}\in(0,\infty]$, $K_{I}\in[1,\infty)$, $K_{v}\in(0,2^{-m}]$, $K_{p}\in[(2K_{I})^{m},\infty)$,  with $\tau_{g}\leq\tau_{\ell}$. Let $\hat{d}_{n}$ be in \eqref{eq:multidimension.estimator}. Then:
	\begin{align*}
	& \underset{P\in\mathcal{P}_{\tau_{g},\tau_{\ell},K_{I},K_{v},K_{p}}^{d}}{\sup}\mathbb{E}_{P^{(n)}}\left[\ell\left(\hat{d}_{n},d(P)\right)\right] \nonumber \\
	& \leq 1(d>1)\left(C_{K_{I},K_{p},K_{v},m}^{(\ref*{prop:multidimension.maximumrisk})}\right)^{n}\max\left\{1,\tau_{g}^{-(dm+m-2d)n}\right\}n^{-\frac{1}{d-1}n},
	\end{align*}
	where $C_{K_{I},K_{p},K_{v},m}^{(\ref*{prop:multidimension.maximumrisk})}\in(0,\infty)$ is a constant depending only on $K_{I}$, $K_{p}$, $K_{v}$, and $m$.
\end{prop}
\begin{proof}[Proof of Proposition~\ref{prop:multidimension.maximumrisk}]
	in Appendix \ref{sec:proof.multidimension}.
\end{proof}
Then similarly to Section \ref{subsec:upper.minimax}, the maximum risk of our estimator $\hat{d}_{n}$ in \eqref{eq:multidimension.estimator} serves as an upper bound on the minimax risk $R_{n}$ in \eqref{eq:regular.minimax}. The maximum of the upper bound in Proposition~\ref{prop:multidimension.maximumrisk} over $d$ ranging from $1$ to $m$ should serve as the upper bound for the maximum risk, hence we get the upper bound of the minimax risk  $R_{n}$ in Proposition~\ref{prop:multidimension.upper} as a generalized result of Proposition~\ref{prop:upper.bound}.
\begin{prop}
\label{prop:multidimension.upper}
%{\em(Upper bound for minimax rate, in multi-dimensions)}
Fix $\tau_{g},\,\tau_{\ell}\in(0,\infty]$, $K_{I}\in[1,\infty)$, $K_{v}\in(0,2^{-m}]$, $K_{p}\in[(2K_{I})^{m},\infty)$, with $\tau_{g}\leq\tau_{\ell}$. Then:
\[
\underset{\hat{d}_{n}}{\inf}\underset{P\in\mathcal{P}}{\sup}\mathbb{E}_{P^{(n)}}\left[\ell\left(\hat{d}_{n},d(P)\right)\right]\leq\left(C_{K_{I},K_{p},K_{v},m}^{(\ref*{prop:multidimension.maximumrisk})}\right)^{n}\max\left\{1,\tau_{g}^{-(m^{2}-m)n}\right\}n^{-\frac{1}{m-1}n},
\]
where $C_{K_{I},K_{p},K_{v},m}^{(\ref*{prop:multidimension.maximumrisk})}$ is from Proposition~\ref{prop:multidimension.maximumrisk}.
\end{prop}

\begin{proof}[Proof of Proposition~\ref{prop:multidimension.upper}]
in Appendix \ref{sec:proof.multidimension}.
\end{proof}

Proposition~\ref{prop:multidimension.lower} provides a lower bound for
the minimax rate $R_{n}$ in \eqref{eq:regular.minimax}, in
multi-dimensions. It can be viewed of a generalization for the binary
dimension case in Proposition~\ref{prop:lower.bound}.

\begin{prop}
\label{prop:multidimension.lower}
%{\em(Lower bound for minimax rate, in multi-dimensions)} 
Fix $\tau_{g},\,\tau_{\ell}\in(0,\infty]$, 
$K_{I}\in[1,\infty)$, $K_{v}\in(0,2^{-m}]$, 
$K_{p}\in[(2K_{I})^{m},\infty)$,  with 
$\tau_{g}\leq\tau_{\ell}$, and suppose that $\tau_{\ell}<K_{I}$. Then,
\[
\underset{\hat{d}_{n}}{\inf}\underset{P\in\mathcal{P}}{\sup}\mathbb{E}_{P^{(n)}}[\ell(\hat{d}_{n},d(P))]\geq\left(C_{K_{I}}^{(\ref*{prop:multidimension.lower})}\right)^{n}\min\left\{ \tau_{\ell}^{-4}n^{-2},1\right\} ^{n}
\]
where $C_{K_{I}}^{(\ref*{prop:multidimension.lower})}\in(0,\infty)$ is a constant depending only on $K_{I}$.
\end{prop}

\begin{proof}[Proof of Proposition~\ref{prop:multidimension.lower}]
in Appendix \ref{sec:proof.multidimension}.
\end{proof}

\section{Conclusion}

On a logarithmic scale,
the leading terms of the lower and upper bounds for the minimax rate $R_{n}$ in 
\eqref{eq:regular.minimax} have the form
$$
-n c\log \tau
$$
for some constant $c$,
where $\tau$ is the global reach for the upper bound and the local reach for the lower bound.
This shows that the difficulty of the problem of estimating
the dimension goes to 0 rapidly with sample size,
in a way that depends on the curvature of the manifold.

There are several open problems.
The first is to tighten the bounds
so that the upper and lower bounds match.
Second, it should be possible to extend the analysis
to allow noise.
With enough noise, the minimax rate should eventually
become the same as 
the rate in \citep{Koltchinskii2000}.
Finally, it would be interesting to get very precise bounds on the
many dimension estimators that appear in the literature and compare these
bounds to the minimax bounds.

%% file: DimensionEstimator_appendix.tex
\newpage
\appendix

\section{Proofs for Section \ref{sec:definition.regular}}
\label{sec:proof.definition.regular}

\noindent \textbf{Lemma~\ref{lem:regular.bounded.volume}.} \textit{
	Fix $\tau_{g}\in(0,\infty]$, and let $M$ be a $d$-dimensional manifold
	with global reach $\geq\tau_{g}$. For $r\in(0,\tau_{g})$, let
	$M_{r}:=\{x\in\mathbb{R}^{m}:\ dist_{\mathbb{R}^{m},}(x,M)<r\}$ be
	an $r$-neighborhood of $M$ in $\mathbb{R}^{m}$. Then, the volume
	of $M$ is upper bounded as 
	\begin{equation}
	vol_{M}(M)\leq\frac{m!}{d!}r^{d-m}vol_{\mathbb{R}^{m}}(M_{r}).\label{eq:proof.regular.bounded.volume.neighborhood}
	\end{equation}
	Further, fix $\tau_{\ell}\in(0,\infty]$, $K_{I}\in[1,\infty)$,
	$K_{v}\in(0,2^{-m}]$, with $\tau_{g}\leq\tau_{\ell}$, and suppose
	$M\in\mathcal{M}_{\tau_{g},\tau_{\ell},K_{I},K_{v}}^{d}$. Then the
	volume of $M$ is upper bounded as
	\begin{equation}
	vol_{M}(M) \leq C_{K_{I},m}^{(\ref*{lem:regular.bounded.volume})}\max\left\{1,\tau_{g}^{d-m}\right\},\label{eq:proof.regular.bounded.volume}
	\end{equation}
	where $C_{K_{I},m}^{(\ref*{lem:regular.bounded.volume})}$
	is a constant depending only on $K_{I}$ and $m$.
}

\begin{figure}
\begin{center}
\includegraphics{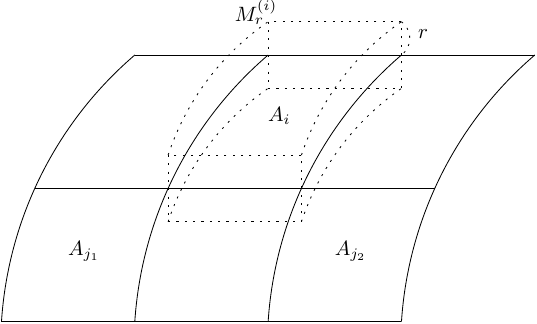}
\end{center}
\caption{$\{A_{1},\ldots,A_{l}\}$ is a disjoint cover of $M$, and each $A_{i}$ is a projection of $M_{r}^{(i)}$ on $M$.}
\label{fig:proof.regular.disjointcover}
\end{figure}

\begin{proof}[Proof of Lemma~\ref{lem:regular.bounded.volume}]

Suppose $\{A_{1},\,\ldots,\,A_{l}\}$ is a disjoint cover of $M$,
i.e. measurable subsets of $M$ such that $A_{i}\cap A_{j}=\emptyset$,
$\overset{l}{\underset{i=1}{\bigcup}}A_{i}=M$, and each $A_{i}$
is equipped with a chart map $\varphi^{(i)}:U_{i}\subset\mathbb{R}^{d}\rightarrow A_{i}$.
Such a triangulation is always possible. For each $A_{i}$, define
$M_{r}^{(i)}:=\{x\in\mathbb{R}^{m}:\ \pi_{M}(x)\in A_{i},\ dist_{\mathbb{R}^{m},||\cdot||_{1}}(x,M)\leq r\}$
so that each $A_{i}$ is a projection of $M_{r}^{(i)}$ on $M$, as
in Figure \ref{fig:proof.regular.disjointcover}. Since $\left\Vert x\right\Vert _{2}\leq\left\Vert x\right\Vert _{1}$ for all $x\in\mathbb{R}^{m}$, $\bigcup_{i=1}^{m}M_{r}^{(i)} \subset M_{r}$ holds, and hence
\begin{equation}
\label{eq:proof.regular.volumesum}
\overset{l}{\underset{i=1}{\sum}}vol_{\mathbb{R}^{m}}(M_{r}^{(i)})\leq vol_{\mathbb{R}^{m}}(M_{r}).
\end{equation}
Fix $i\in\{1,\ldots,l\}$. Then for each $u \in U_{i}$, there exists a linear isometry $R^{(i)}(u):\mathbb{R}^{m-d}\rightarrow(T_{\varphi^{(i)}(u)}M)^{\perp}$, which can be identified as
an $m\times(m-d)$ matrix with $j^{th}$ column being $R^{(i,j)}(u)$, so that $M_{r}^{(i)}$ can be parametrized as $\psi^{(i)}:U_{i}\times B_{\mathbb{R}^{m-d},\|\cdot\|_{1}}(0,r)\rightarrow M_{r}^{(i)}$ with
\begin{equation}
\label{eq:proof.regular.chartmap}
\psi^{(i)}(u,t)=\varphi^{(i)}(u)+R^{(i)}(u)t=\varphi^{(i)}(u)+\overset{m-d}{\underset{j=1}{\sum}}t_{j}R^{(i,j)}(u).
\end{equation}
Then, because $R^{(i)}$ is an isometry, 
\begin{equation}
\label{eq:proof.regular.isometry}
R^{(i)}(u)^{\top}R^{(i)}(u)=I_{m-d}.
\end{equation}
Let $\psi_{u}^{(i)}=\frac{\partial\psi^{(i)}}{\partial u}=\left(\frac{\partial\psi^{(i)}}{\partial u_{1}},\ldots,\frac{\partial\psi^{(i)}}{\partial u_{d}}\right)\in\mathbb{R}^{m\times d}$
be the partial derivative of $\psi^{(i)}$ with respect to $u$ and let  $\psi_{t}^{(i)}=\frac{\partial\psi^{(i)}}{\partial t}$ be the partial derivative of $\psi^{(i)}$ with respect to $t$. Define $\varphi_{u}^{(i)}$ and $R_{u}^{(i,j)}$ similarly. Then, since $R^{(i)}$ is an isometry, $image(R^{(i)}(u))=(T_{\varphi^{(i)}(u)}M)^{\perp}$ holds, and hence
\begin{equation}
\label{eq:proof.regular.isometry.chartmap.diffu}
R^{(i)}(u)^{\top}\varphi_{u}^{(i)}(u)=0.
\end{equation}
Also by differentiating \eqref{eq:proof.regular.isometry}, for all $j$, 
\begin{equation}
\label{eq:proof.regular.isometry.diffu}
R_{u}^{(i,j)}(u)^{\top}R^{(i)}(u)=0.
\end{equation}
Also by differentiating \eqref{eq:proof.regular.chartmap}, we get
\begin{equation}
\label{eq:proof.regular.chartmap.diffu}
\psi_{u}^{(i)}(u,t)=\varphi_{u}^{(i)}(u)+\overset{m-d}{\underset{j=1}{\sum}}t_{j}R_{u}^{(i,j)}(u),
\end{equation}
and 
\begin{equation}
\label{eq:proof.regular.chartmap.difft}
\psi_{t}^{(i)}(u,t)=R^{(i)}(u).
\end{equation}
Hence by multiplying \eqref{eq:proof.regular.chartmap.diffu} and \eqref{eq:proof.regular.chartmap.difft}, and by applying \eqref{eq:proof.regular.isometry}, \eqref{eq:proof.regular.isometry.chartmap.diffu}, and \eqref{eq:proof.regular.isometry.diffu}, we get
\begin{equation}
\label{eq:proof.regular.multiplycharts.difftdiffu}
\psi_{t}^{(i)}(u,t)^{\top}\psi_{u}^{(i)}(u,t)=R^{(i)}(u)^{\top}\varphi_{u}^{(i)}(u)+R^{(i)}(u)^{\top}R_{u}^{(i)}(u)t=0,
\end{equation}
and
\begin{equation}
\label{eq:proof.regular.multiplycharts.difftdifft}
\psi_{t}^{(i)}(u,t)^{\top}\psi_{t}^{(i)}(u,t)=R^{(i)}(u)^{\top}R^{(i)}(u)=I_{m-d}.
\end{equation}
Now let's consider $\psi_{u}^{(i)}(u,t)^{\top}\psi_{u}^{(i)}(u,t)$.
From \eqref{eq:proof.regular.isometry.diffu} and $image(R^{(i)}(u))=(T_{\varphi^{(i)}(u)}M)^{\perp}$, column space generated by $R_{u}^{(i,j)}(u)$ is contained in $T_{\varphi^{(i)}(u)}M$, i.e. 
\[
\left\langle R_{u}^{(i,j)}(u)\right\rangle\subset T_{\varphi^{(i)}(u)}(M)=span(\varphi_{u}^{(i)}(u)).
\]
Therefore, there exists $\Lambda^{(i,j)}(u)$ : $d\times d$ matrix
such that
\[
R_{u}^{(i,j)}(u)=\varphi_{u}^{(i)}(u)\Lambda^{(i,j)}(u).
\]
Then by applying this to \eqref{eq:proof.regular.chartmap.diffu}, 
\begin{equation}
\label{eq:proof.regular.twochartmaps.diffu}
\psi_{u}^{(i)}(u,t)=\varphi_{u}^{(i)}(u)\left(I+\overset{m-d}{\underset{j=1}{\sum}}t_{j}\Lambda^{(i,j)}(u)\right).
\end{equation}
Now $M$ being of global reach $\geq\tau_{g}$ implies $\psi_{u}^{(i)}(u,t)$
is of full rank for all $t\in B_{\mathbb{R}^{m-d},\|\cdot\|_{1}}(0,\tau_{g})$. From \eqref{eq:proof.regular.twochartmaps.diffu},
this implies $I+\overset{m-d}{\underset{j=1}{\sum}}t_{j}\Lambda^{(i,j)}(u)$
is invertible for all  $t\in B_{\mathbb{R}^{m-d},\|\cdot\|_{1}}(0,\tau_{g})$, and this
implies all singular values of $\Lambda^{(i,j)}(u)$ are bounded by
$\frac{1}{\tau_{g}}$. Hence for all $v\in\mathbb{R}^{d},$ 
\[
\left|v^{\top}\Lambda^{(i,j)}(u)v\right|\leq\frac{\|v\|_{2}^{2}}{\tau_{g}},
\]
and accordingly,
\begin{align*}
\left|v^{\top}\left(I+\overset{m-d}{\underset{j=1}{\sum}}t_{j}\Lambda^{(i,j)}(u)\right)v\right| & \geq\|v\|_{2}^{2}-\overset{m-d}{\underset{j=1}{\sum}}|t_{j}|\left|v^{\top}\Lambda^{(i,j)}(u)v\right|\\
& \geq\left(1-\frac{\|t\|_{1}}{\tau_{g}}\right)\|v\|_{2}^{2}.
\end{align*}
Hence any singular value $\sigma$ of $I+\overset{m-d}{\underset{j=1}{\sum}}t_{j}\Lambda^{(i,j)}(u)$ satisfies $|\sigma|\geq1-\frac{\|t\|_{1}}{\tau_{g}}$. And since $\|t\|_{1}\leq \tau_{g}$,
\[
\left|I+\overset{m-d}{\underset{j=1}{\sum}}t_{j}\Lambda^{(i,j)}(u)\right|\geq\left(1-\frac{\|t\|_{1}}{\tau_{g}}\right)^{d}.
\]
By applying this result to \eqref{eq:proof.regular.twochartmaps.diffu}, the determinant of $\psi_{u}^{(i)}(u,t)^{\top}\psi_{u}^{(i)}(u,t)$
is lower bounded as 
\begin{align}
\label{eq:proof.regular.multiplycharts.diffudiffu}
\left|\psi_{u}^{(i)}(u,t)^{\top}\psi_{u}^{(i)}(u,t)\right| & =\left|I+\overset{m-d}{\underset{j=1}{\sum}}t_{j}\Lambda^{(i,j)}(u)\right|^{2}\left|\varphi_{u}^{(i)}(u)^{\top}\varphi_{u}^{(i)}(u)\right| \nonumber\\
& \geq\left(1-\frac{\|t\|_{1}}{\tau_{g}}\right)^{2d}\left|\varphi_{u}^{(i)}(u)^{\top}\varphi_{u}^{(i)}(u)\right|.
\end{align}
Now, let $g_{ij}^{(M_{r})}$ be the Riemannian metric tensor of $M_{r}$,
and $g_{ij}^{(M)}$ be the Riemannian metric tensor of $M$. Then
from \eqref{eq:proof.regular.multiplycharts.difftdiffu}, \eqref{eq:proof.regular.multiplycharts.difftdifft}, and \eqref{eq:proof.regular.multiplycharts.diffudiffu}, the determinant of Riemannian metric tensor $g_{ij}^{(M_{r})}$
is lower bounded by 
\begin{align*}
|\det(g_{ij}^{(M_{r})})| & =\left|\left(\psi_{u}^{(i)}(u,t)\ \psi_{t}^{(i)}(u,t)\right)^{\top}\left(\psi_{u}^{(i)}(u,t)\ \psi_{t}^{(i)}(u,t)\right)\right|\\
& =\left|\begin{array}{cc}
\psi_{u}^{(i)}(u,t)^{\top}\psi_{u}^{(i)}(u,t) & \ \psi_{u}^{(i)}(u,t)^{\top}\psi_{t}^{(i)}(u,t)\\
\psi_{u}^{(i)}(u,t)^{\top}\psi_{t}^{(i)}(u,t) & \ \psi_{t}^{(i)}(u,t)^{\top}\psi_{t}^{(i)}(u,t)
\end{array}\right|\\
& =\left|\begin{array}{c}
\psi_{u}^{(i)}(u,t)^{\top}\psi_{u}^{(i)}(u,t)\end{array}\right|\\
& \geq\left(1-\frac{\|t\|_{1}}{\tau_{g}}\right)^{2d}\left|\varphi_{u}^{(i)}(u)^{\top}\varphi_{u}^{(i)}(u)\right|\\
& =\left(1-\frac{\|t\|_{1}}{\tau_{g}}\right)^{2d}|\det(g_{ij}^{(M)})|.
\end{align*}
And from this, the volume of $M_{r}^{(i)}$ is lower bounded as 
\begin{align}
\label{eq:proof.regular.volume.bound.patch}
vol_{\mathbb{R}^{m}}(M_{r}^{(i)}) & =\int_{U_{i}\times B_{\mathbb{R}^{m},\|\cdot\|_{1}}(0,r)}\sqrt{|\det(g_{ij}^{(M_{r})})|}dudt\nonumber \\
& \geq\int_{U_{i}}\int_{B_{\mathbb{R}^{m},\|\cdot\|_{1}}(0,r)}(1-\|t\|_{1}\kappa_{g})^{d}\sqrt{|\det(g_{ij}^{(M)})|}dtdu\nonumber \\
& =vol(U_{i})\int_{0}^{r}\int_{t_{1}+\cdots+t_{m-d-1}\leq s}\left(1-\frac{s}{\tau_{g}}\right)^{d}dt_{1}\cdots dt_{m-d-1}ds\nonumber \\
& =\frac{1}{(m-d-1)!}vol(U_{i})\int_{0}^{r}s^{m-d-1}\left(1-\frac{s}{\tau_{g}}\right)^{d}ds \nonumber \\
 & =\frac{1}{(m-d-1)!}r^{m-d}vol(U_{i})\int_{0}^{1}u^{m-d-1}\left(1-\frac{r}{\tau_{g}}u\right)^{d}du\nonumber \\
& \geq\frac{1}{(m-d-1)!}r^{m-d}vol(U_{i})\int_{0}^{1}u^{m-d-1}(1-u)^{d}du\nonumber \\
& =\frac{d!}{m!}r^{m-d}vol(U_{i}).
\end{align}
By applying \eqref{eq:proof.regular.volume.bound.patch} to \eqref{eq:proof.regular.volumesum}, we can lower bound the volume of $M_{r}$ as
\begin{align*}
vol_{\mathbb{R}^{m}}(M_{r}) & \geq\frac{d!}{m!}r^{m-d}\overset{l}{\underset{i=1}{\sum}}vol(U_{i})\nonumber \\
& =\frac{d!}{m!}r^{m-d}vol_{M}(M),
\end{align*}
hence rewriting this gives \eqref{eq:proof.regular.bounded.volume.neighborhood} as
\begin{equation}
vol_{M}(M) \leq \frac{m!}{d!}r^{d-m}vol_{\mathbb{R}^{m}}(M_{r}).\label{eq:proof.regular.volume.bound.metric}
\end{equation}
Now, suppose $M\in\mathcal{M}_{\tau_{g},\tau_{\ell},K_{I},K_{v}}^{d}$. With $r=\min\{\tau_{g},K_{I}\}$, $M_{r}$ is contained in $\min\{\tau_{g},K_{I}\}$-neighborhood
of $I$, hence 
\begin{equation}
vol_{\mathbb{R}^{m}}(M_{r})
\leq 2^{m}(K_{I}+\min\{\tau_{g},K_{I}\})^{m}
\leq 2^{2m}K_{I}^{m}.\label{eq:proof.regular.volume.bound.subset}
\end{equation}
By combining \eqref{eq:proof.regular.volume.bound.metric} and \eqref{eq:proof.regular.volume.bound.subset}, we get the desired upper bound of $vol_{M}(M)$ in  \eqref{eq:proof.regular.bounded.volume} as
\begin{align*}
vol_{M}(M) & \leq \frac{m!}{d!}r^{d-m}vol_{\mathbb{R}^{m}}(M_{r})\\
& \leq \frac{m!}{d!}2^{2m}K_{I}^{m}\min\{\tau_{g},K_{I}\}^{d-m}\\
& \leq C_{K_{I},m}^{(\ref*{lem:regular.bounded.volume})}\max\left\{1,\tau_{g}^{d-m}\right\},
\end{align*}
where $C_{K_{I},m}^{(\ref*{lem:regular.bounded.volume})}:=m!2^{2m}K_{I}^{m}$ is a constant depending
only on $K_{I}$ and $m$.

\end{proof}

\noindent  \textbf{Lemma~\ref{lem:regular.bounded.covernumber}.} \textit{
	Fix $\tau_{g},\,\tau_{\ell}\in(0,\infty]$, $K_{I}\in[1,\infty)$, $K_{v}\in(0,2^{-m}]$, with $\tau_{g}\leq\tau_{\ell}$. Let $M\in\mathcal{M}_{\tau_{g},\tau_{\ell},K_{I},K_{v}}^{d}$ and $r\in(0,2\sqrt{3}\tau_{g}]$. 
	Then $M$ can be covered by $N$ radius $r$ balls $B_{M}(p_{1},r)$, $\ldots$, $B_{M}(p_{N},r)$, with 
	\begin{equation}
	\label{eq:proof.regular.volumecoverbound}
	N \leq \left\lfloor \frac{2^{d}vol(M)}{K_{v} r^{d}\omega_{d}}\right\rfloor.
	\end{equation}
}
\begin{proof}[Proof of Lemma~\ref{lem:regular.bounded.covernumber}]
	We follow the strategy in \citep[][4.3.1. Lemma 3]{MaF2011}.\\
	Consider a maximal family of disjoint balls $\left\{ B_{M}(p_{1},\frac{r}{2}),\ldots,B_{M}(p_{N},\frac{r}{2})\right\} $,
	i.e. $B_{M}(p_{i},\frac{r}{2})\cap B_{M}(p_{j},\frac{r}{2})=\emptyset$
	for $i\neq j$ and for all $q\in M$, there exists $i\in[1,N]$ such
	that $B_{M}(q,\frac{r}{2})\cap B_{M}(p_{i},\frac{r}{2})\neq\emptyset$.
	Then $\left\Vert q-p_{i}\right\Vert _{2}<r$ holds, so $\left\{ B_{M}(p_{1},r),\ldots,B_{M}(p_{N},r)\right\} $
	covers $M$. Now, note that $B_{M}(p_{i},\frac{r}{2})$'s are disjoint,
	and hence 
	\begin{equation}
	\sum_{i=1}^{N}vol(B_{M}(p_{i},\frac{r}{2}))\leq vol(M).\label{eq:proof.regular.volumecoversum}
	\end{equation}
	Then since $\frac{r}{2}\leq\sqrt{3}\tau_{g}$, the condition (4) in Definition \ref{def:regular.manifold} implies
	$vol(B_{M}(p_{i},\frac{r}{2}))\geq K_{v}2^{-d}r^{d}\omega_{d}$ for
	all $i$, hence applying this to \eqref{eq:proof.regular.volumecoversum}
	yields 
	\[
	N\leq\frac{2^{d}vol(M)}{K_{v}r^{d}\omega_{d}},
	\]
	hence $M$ can be covered by $N$ radius $r$ balls with $N$ satisfying \eqref{eq:proof.regular.volumecoverbound}.
\end{proof}

\begin{lem}
\label{lem:proof.upper.Toponogov}
(Toponogov comparison theorem, 1959) Let $(M,g)$ be a complete
Riemannian manifold with sectional curvature$\geq \kappa$, and let $S_{\kappa}$
be a surface of constant Gaussian curvature $\kappa$. Given any geodesic
triangle with vertices $p,q,r\in M$ forming an angle $\alpha$ at $q$,
consider a (comparison) triangle with vertices
$\bar{p},\bar{q},\bar{r}\in S_{\kappa}$ such that
$dist_{S_{\kappa}}(\bar{p},\bar{q})=dist_{M}(p,q)$, $dist_{S_{\kappa}}(\bar{r},\bar{q})=dist_{M}(r,q)$, and $\angle\bar{p}\bar{q}\bar{r} = \angle pqr$. Then
\[
dist_{M}(\bar{p},\bar{r}) \leq dist_{S_{\kappa}}(p,r).
\]
\end{lem}

\begin{proof}[Proof of Lemma~\ref{lem:proof.upper.Toponogov}]
\citep[See][Theorem 79, p.339]{Petersen2006}. Note that for a manifold
with boundary, the complete Riemannian manifold condition can be relaxed
to requiring the existence of a geodesic path joining $p$ and $q$ whose
image lies on $int M$.
\end{proof}

\begin{lem}
\label{lem:proof.upper.hyperbolic.cosine}
(Hyperbolic law of cosines) 
Let $H_{-\kappa^{2}}$ be a hyperbolic plane
whose Gaussian curvature is $-\kappa^{2}$. Then given a hyperbolic
triangle $ABC$ with angles $\alpha$, $\beta$, $\gamma$, and side
lengths $BC=a$, $CA=b$, and $AB=c$, the following holds:
\[
\cosh(\kappa a)  = \cosh (\kappa b)\cosh (\kappa c)-\sinh (\kappa b)\sinh (\kappa c)\cos\alpha.
\]
\end{lem}

\begin{proof}[Proof of Lemma~\ref{lem:proof.upper.hyperbolic.cosine}]
\citep[See][2.13 The Law of Cosines in $M_{\kappa}^{n}$, p.24]{BridsonH1999}.
\end{proof}

\begin{claim}
	\label{claim:proof.regular.dist.fraction}
	Let $\lambda\in[0,1]$ and $a,b\in[0,\infty)$.
	Then 
	\begin{equation}
	\label{eq:proof.regular.dist.fraction}
	\frac{\cosh^{-1}\left((1-\lambda)\cosh a+\lambda\cosh b\right)}{\sqrt{(1-\lambda)a^{2}+\lambda b^{2}}}\leq\frac{\sinh\left(\max\{a,b\}/2\right)}{\max\{a,b\}/2}.
	\end{equation}
\end{claim}
\begin{proof}[Proof of Claim~\ref{claim:proof.regular.dist.fraction}]
	Without loss of generality, assume $a\leq b$. Consider two functions $F,G:[0,\infty)^{2}\times[0,1]\to\mathbb{R}$ defined
	as $F(a,b,\lambda)=f^{-1}((1-\lambda)f(a)+\lambda f(b))$ and $G(a,b,\lambda)=g^{-1}((1-\lambda)g(a)+\lambda g(b))$,
	for $0\leq a \leq b$, $\lambda\in[0,1]$, $f(t)=\cosh t$, and $g(t)=t^{2}$.
	Applying Toponogov comparison theorem in Lemma~\ref{lem:proof.upper.Toponogov} to \eqref{eq:proof.regular.bounded.length.dist.fraction} in the proof of Lemma~\ref{lem:regular.bounded.length} with $r_{1} = \frac{a+b}{2}$, $r_{2} = \frac{b-a}{2}$, $\alpha=\arccos(\sqrt{\lambda})\in[0,\frac{\pi}{2}]$
	implies
	\[
	F(a,b,\lambda)\geq G(a,b,\lambda),
	\]
	and $f$ and $g$ being
	strictly increasing function implies $a\leq G(a,b,\lambda)\leq F(a,b,\lambda)\leq b$.
	Also differentiating the log fraction $\frac{\partial}{\partial a}\log\frac{F(a,b,\lambda)}{G(a,b,\lambda)}$
	gives 
	\begin{align}
	\label{eq:proof.regular.dist.fraction.diffa}
	\frac{\partial}{\partial a}\log\frac{F(a,b,\lambda)}{G(a,b,\lambda)} & =\frac{(1-\lambda)f'(a)}{f'(F(a,b,\lambda))F(a,b,\lambda)}-\frac{(1-\lambda)g'(a)}{g'(G(a,b,\lambda))G(a,b,\lambda)}\nonumber \\
	& =\frac{1-\lambda}{F(a,b,\lambda)}\exp\left(-\int_{a}^{F(a,b,\lambda)}(\log f')'(t)dt\right)\nonumber \\
	& \quad-\frac{1-\lambda}{G(a,b,\lambda)}\exp\left(-\int_{a}^{G(a,b,\lambda)}(\log g')'(t)dt\right).
	\end{align}
	Then applying $(\log f')'(t)=\coth t>\frac{1}{t}=(\log g')'(t)$ for $t>0$
	and $F(a,b,\lambda)\geq G(a,b,\lambda)$ to \eqref{eq:proof.regular.dist.fraction.diffa}
	implies 
	\[
	0<\forall a<b,\ \frac{\partial}{\partial a}\log\frac{F(a,b,\lambda)}{G(a,b,\lambda)}<0,
	\]
	and hence 
	\[
	\frac{F(a,b,\lambda)}{G(a,b,\lambda)}\leq\frac{F(0,b,\lambda)}{G(0,b,\lambda)}.
	\]
	By expanding $F$ and $G$ from this, we get 
	\begin{align*}
	\frac{\cosh^{-1}\left((1-\lambda)\cosh a+\lambda\cosh b\right)}{\sqrt{(1-\lambda)a^{2}+\lambda b^{2}}} & \leq\frac{\cosh^{-1}\left(\lambda\cosh b+(1-\lambda)\right)}{\sqrt{\lambda b^{2}}}\\
	& =\frac{\cosh^{-1}\left(1+2\lambda\sinh^{2}\left(\frac{b}{2}\right)\right)}{b\sqrt{\lambda}}\\
	& \leq\frac{2\sinh\left(\frac{b}{2}\right)}{b},
	\end{align*}
	where the last line is coming from $1+x\leq\cosh\sqrt{2x}\Longrightarrow\cosh^{-1}\left(1+x\right)\leq\sqrt{2x}$ for all $x\geq 0$.
	Hence we get \eqref{eq:proof.regular.dist.fraction}. 
	
\end{proof}

\noindent \textbf{Lemma~\ref{lem:regular.bounded.length}.} \textit{Fix $\tau_{g},\,\tau_{\ell}\in(0,\infty]$, $K_{I}\in[1,\infty)$, $K_{v}\in(0,2^{-m}]$, with $\tau_{g}\leq\tau_{\ell}$. Let $M\in\mathcal{M}_{\tau_{g},\tau_{\ell},K_{I},K_{v}}^{d}$ and let
$\exp_{p_{k}}:\mathcal{E}_{k} \subset \mathbb{R}^{m} \rightarrow {\mathcal{M}}$
be an exponential map, where  $\mathcal{E}_{k}$ is the domain of the exponential
map $\exp_{p_{k}}$ and $T_{p_{k}}M$ is identified with $\mathbb{R}^{d}$.
For all
$v,w\in \mathcal{E}_{k}$, let $R_{k}:=\max\{||v||,||w||\}$. Then 
\begin{equation}
\label{eq:proof.regular.bounded.length}
\|\exp_{p_{k}}(v)-\exp_{p_{k}}(w)\|_{\mathbb{R}^{m}}\leq
\frac{\sinh(\sqrt{2}R_{k}/\tau_{\ell})}{\sqrt{2}R_{k}/\tau_{\ell}}\|v-w\|_{\mathbb{R}^{d}}.
\end{equation}
}

\begin{figure}
	\begin{center}
		\begin{subfigure}[b]{0.4\textwidth}
			\begin{center}
			\includegraphics{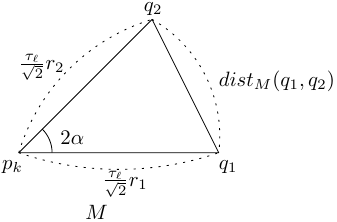}
			\end{center}
			\caption{triangle $\triangle p_{k} q_{1} q_{2}$ in $M$}
			\label{subfig:proof.upper.triangle.manifold}
		\end{subfigure}
		\begin{subfigure}[b]{0.5\textwidth}
			\begin{center}
			\includegraphics{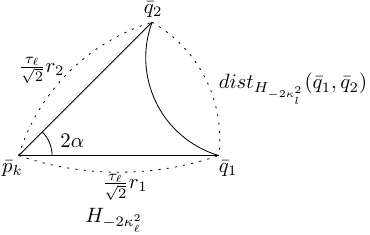}
			\end{center}
			\caption{comparison triangle $\triangle\bar{p}_{k}\bar{q}_{1}\bar{q}_{2}$ in $H_{-2\kappa_{\ell}^{2}}$}
			\label{subfig:proof.upper.triangle.comparison}
		\end{subfigure}
	\end{center}
	
	\caption{\protect\subref{subfig:proof.upper.triangle.manifold} A triangle $\triangle p_{k} q_{1} q_{2}$ in $M$ formed by $p_{k}$, $q_{1}$, $q_{2}$, and
		\protect\subref{subfig:proof.upper.triangle.comparison} its comparison triangle $\triangle\bar{p}_{k}\bar{q}_{1}\bar{q}_{2}$ in $H_{-2\kappa_{\ell}}^{2}$.}
	\label{fig:proof.upper.regularity.triangle}
\end{figure}

\begin{proof}[Proof of Lemma~\ref{lem:regular.bounded.length}]
Let $q_{1}=\exp_{p_{k}}(v)$ and $q_{2}=\exp_{p_{k}}(w)$. Let $r_{1}:=\frac{\sqrt{2}\left\Vert v\right\Vert}{\tau_{\ell}}$ and $r_{2}:=\frac{\sqrt{2}\left\Vert w\right\Vert}{\tau_{\ell}}$, so that 
$dist_{M}(p_{k},q_{1})=\frac{\tau_{\ell}}{\sqrt{2}}r_{1}$ and $dist_{M}(p_{k},q_{2})=\frac{\tau_{\ell}}{\sqrt{2}}r_{2}$, and let $\alpha:=\frac{1}{2}\angle q_{1}p_{k}q_{2}\in[0,\frac{\pi}{2}]$ so that 
$\angle q_{1}p_{k}q_{2}=2\alpha$, as in Figure \ref{fig:proof.upper.regularity.triangle}\subref{subfig:proof.upper.triangle.manifold}. Then
\begin{align}
\label{eq:proof.regular.bounded.length.dist.tangent}
\|v-w\|_{\mathbb{R}^{d}} & = \frac{\tau_{\ell}}{\sqrt{2}}\sqrt{r_{1}^{2}+r_{2}^{2}-2r_{1}r_{2}\cos2\alpha}\nonumber\\
& = \frac{\tau_{\ell}}{\sqrt{2}}\sqrt{(r_{1}+r_{2})^{2}\sin^{2}\alpha+(r_{1}-r_{2})^{2}\cos^{2}\alpha}.
\end{align}
Let $\kappa_{\ell}:=\frac{1}{\tau_{\ell}}$,  $H_{-2\kappa_{\ell}^{2}}$ be a surface of constant sectional curvature
$-2\kappa_{\ell}^{2}$, and let $\bar{p}_{k},\bar{q}_{1},\bar{q}_{2}\in H_{-2\kappa_{\ell}^{2}}$ be such that
 $dist_{H_{-2\kappa_{\ell}^{2}}}(\bar{p}_{k},\bar{q}_{1})=dist_{M}(p_{k},q_{1})$, $dist_{H_{-2\kappa_{\ell}^{2}}}(\bar{p}_{k},\bar{q}_{2})=dist_{M}(p_{k},q_{2})$, and $\angle\bar{q}_{1}\bar{p}_{k}\bar{q}_{2}=\angle q_{1}p_{k}q_{2}$, so
that $\triangle\bar{p}_{k}\bar{q}_{1}\bar{q}_{2}$ becomes a comparison
triangle of ${p}_{k}{q}_{1}{q}_{2}$, as in Figure
\ref{fig:proof.upper.regularity.triangle}\subref{subfig:proof.upper.triangle.comparison}. Then since $(\text{sectional curvature of }M)\geq-2\kappa_{\ell}^{2}$ by \cite[Proposition A.1 (iii)]{AamariKCMRW2017}, from the Toponogov comparison theorem in Lemma
\ref{lem:proof.upper.Toponogov},
\begin{equation}
\label{eq:proof.regular.bounded.length.toponogov}
dist_{M}(q_{1},q_{2})\leq dist_{H_{-2\kappa_{\ell}^{2}}}(\bar{q}_{1},\bar{q}_{2}).
\end{equation}
Also, by applying the hyperbolic law of cosines in Lemma~\ref{lem:proof.upper.hyperbolic.cosine} to the comparison triangle $\triangle\bar{p}_{k}\bar{q}_{1}\bar{q}_{2}$ in Figure \ref{fig:proof.upper.regularity.triangle}\subref{subfig:proof.upper.triangle.manifold},
\begin{align}
\cosh\left(\frac{\sqrt{2}}{\tau_{\ell}}dist_{H_{-2\kappa_{\ell}^{2}}}(\bar{q}_{1},\bar{q}_{2})\right) 
& = \cosh r_{1}\cosh r_{2}-\sinh r_{1}\sinh r_{2}\cos2\alpha \nonumber\\
 & = (\sin^{2}\alpha)\cosh(r_{1}+r_{2})+(\cos^{2}\alpha)\cosh(r_{1}-r_{2}).\label{eq:proof.regular.bounded.length.dist.comparison}
\end{align}
From \eqref{eq:proof.regular.bounded.length.dist.tangent} and \eqref{eq:proof.regular.bounded.length.dist.comparison}, we can expand the fraction of the distances $\frac{dist_{H_{-2\kappa_{\ell}^{2}}}(\bar{q}_{1},\bar{q}_{2})}{\|v-w\|_{\mathbb{R}^{d}}}$ as 
\begin{equation}
\label{eq:proof.regular.bounded.length.dist.fraction}
\frac{dist_{H_{-2\kappa_{\ell}^{2}}}(\bar{q}_{1},\bar{q}_{2})}{\|v-w\|_{\mathbb{R}^{d}}}=\frac{\cosh^{-1}\left(\sin^{2}\alpha\cosh(r_{1}+r_{2})+\cos^{2}\alpha\cosh(r_{1}-r_{2})\right)}{\sqrt{(\sin^{2}\alpha)(r_{1}+r_{2})^{2}+(\cos^{2}\alpha)(r_{1}-r_{2})^{2}}}.
\end{equation}
Then we can upper bound the fraction of the distances $\frac{dist_{H_{-2\kappa_{\ell}^{2}}}(\bar{q}_{1},\bar{q}_{2})}{\|v-w\|_{\mathbb{R}^{d}}}$ by plugging in $a=|r_{1}-r_{2}|$, $b=r_{1}+r_{2}$, $\lambda=\sin^{2}\alpha$ to Claim~\ref{claim:proof.regular.dist.fraction} as
\begin{equation}
\label{eq:proof.regular.bounded.length.claimapplied}
\frac{\cosh^{-1}\left(\sin^{2}\alpha\cosh(r_{1}+r_{2})+\cos^{2}\alpha\cosh(r_{1}-r_{2})\right)}{\sqrt{(\sin^{2}\alpha)(r_{1}+r_{2})^{2}+(\cos^{2}\alpha)(r_{1}-r_{2})^{2}}}\leq\frac{\sinh\left(\frac{r_{1}+r_{2}}{2}\right)}{(r_{1}+r_{2})/2}.
\end{equation}
Then since $t\mapsto\frac{\sinh t}{t}$ is an increasing function of $t$ and $\frac{r_{1}+r_{2}}{2}\leq \sqrt{2} R_{k}/\tau_{\ell}$, so 
\begin{equation}
\label{eq:proof.regular.bounded.length.unifupper}
\frac{\sinh\left(\frac{r_{1}+r_{2}}{2}\right)}{(r_{1}+r_{2})/2}\leq\frac{\sinh(\sqrt{2}R_{k}/\tau_{\ell})}{\sqrt{2}R_{k}/\tau_{\ell}}.
\end{equation}
Combining \eqref{eq:proof.regular.bounded.length.dist.fraction}, \eqref{eq:proof.regular.bounded.length.claimapplied}, and \eqref{eq:proof.regular.bounded.length.unifupper}, we have an upper bound of the fraction of the distances $\frac{dist_{H_{-2\kappa_{\ell}^{2}}}(\bar{q}_{1},\bar{q}_{2})}{\|v-w\|_{\mathbb{R}^{d}}}$ as
\begin{equation}
\label{eq:proof.regular.bounded.length.distfraction.unifupper}
\frac{dist_{H_{-2\kappa_{\ell}^{2}}}(\bar{q}_{1},\bar{q}_{2})}{\|v-w\|_{\mathbb{R}^{d}}}\leq\frac{\sinh(\sqrt{2}R_{k}/\tau_{\ell})}{\sqrt{2}R_{k}/\tau_{\ell}}.
\end{equation}
And finally, combining \eqref{eq:proof.regular.bounded.length.toponogov} and \eqref{eq:proof.regular.bounded.length.distfraction.unifupper}, we get the desired upper bound of $\|\exp_{p_{k}}(v)-\exp_{p_{k}}(w)\|_{\mathbb{R}^{m}}$ in \eqref{eq:proof.regular.bounded.length} as
\begin{align*}
\|\exp_{p_{k}}(v)-\exp_{p_{k}}(w)\|_{\mathbb{R}^{m}} & \leq dist_{M}(q_{1},q_{2})\\
& \leq dist_{H_{-2\kappa_{\ell}^{2}}}(\bar{q}_{1},\bar{q}_{2})\\
& \leq\frac{\sinh(\sqrt{2}R_{k}/\tau_{\ell})}{\sqrt{2}R_{k}/\tau_{\ell}}\|v-w\|_{\mathbb{R}^{d}}.
\end{align*}

\end{proof}

%\begin{lem}
%Let $M\in\mathcal{M}_{\tau_{g},\tau_{\ell},N}^{d_{1}}$ and $U_{p_{k}}$
%be a normal neighborhood of $p_{k}$ with chart map $\varphi_{k}:U_{p_{k}}\rightarrow\mathbb{R}^{d_{1}}$.
%Then there exists $K'_{\kappa_{g},K_{I}}>0$ which depends only on
%$\kappa_{g}$ and $K_{I}$ such that for all $p,q\in U_{k}$, 
%\begin{equation}
%\|p-q\|_{\mathbb{R}^{m}}^{d_{1}}\leq K'_{\kappa_{g},K_{I}}\|\varphi_{k}(p)-\varphi_{k}(q)\|_{\mathbb{R}^{d_{1}}}^{d_{1}}.
%\end{equation}
%\end{lem}
%\begin{proof}
%This is since in normal neighborhood $U_{p_{k}}$, metric tensor $g_{ij}(x)$
%at $x\in U_{p_{k}}$ is continuous function of curvature $R(,,,)$
%and distance $d_{M}(x,p_{k})$. And from Lemma 3, $d_{M}(x,p_{k})\leq C_{1,m}\kappa_{g}^{m-1}2^{2m-1}\left(K_{I}+\frac{1}{2}\tau_{g}\right)^{m}$.
%Hence there exists $K'_{\kappa_{g},K_{I}}$ such that $|g_{ij}|\leq\left(\frac{K'_{\kappa_{g},K_{l}}}{d_{1}^{2}}\right)^{d_{1}}$.
%And hence result follows.\end{proof}

\section{Proofs for Section \ref{sec:upper}}
\label{sec:proof.upper}

\begin{claim}
\label{claim:proof.upper.conditionalcdf}
	
	Fix $\tau_{g},\,\tau_{\ell}\in(0,\infty]$, $K_{I}\in[1,\infty)$,
		$K_{v}\in(0,2^{-m}]$, $K_{p}\in[(2K_{I})^{m},\infty)$, $d_{1},\,d_{2}\in\mathbb{N}$,
		with $\tau_{g}\leq\tau_{\ell}$ and $1\leq d_{1}<d_{2}\leq m$. Let
		$X_{1},\ldots,X_{n}\sim P\in\mathcal{P}_{\tau_{g},\tau_{\ell},K_{I},K_{v},K_{p}}^{d_{2}}$.
		Then for all $y\in[0,\infty)$,		
	\begin{equation}
	\label{eq:proof.upper.conditionalcdf}
	P^{(n)}\left(||X_{n}-X_{n-1}||_{\mathbb{R}^{m}}^{d_{1}}\leq y|X_{1},\ldots,X_{n-1}\right)\leq C_{K_{I},K_{p},m}^{(\ref*{claim:proof.upper.conditionalcdf})}\left\{1,\tau_{g}^{d_{2}-m}\right\}y^{\frac{d_{2}}{d_{1}}},
	\end{equation}
	where $C_{K_{I},K_{p},m}^{(\ref*{claim:proof.upper.conditionalcdf})}$ is a constant depending
		only on $K_{I}$, $K_{p}$, and $m$.
	
\end{claim} 

\begin{proof}[Proof of Claim~\ref{claim:proof.upper.conditionalcdf}]
	
	Let $p_{X_{n}}$ be the pdf of $X_{n}$. Then the conditional cdf of $||X_{n}-X_{n-1}||_{\mathbb{R}^{m}}^{d_{1}}$
	given $X_{1},\ldots,X_{n-1}$ is upper bounded by the volume of a ball in the manifold $M$ as 
	\begin{align}
	\label{eq:proof.upper.conditionalcdf.bound.volume}
	& P^{(n)}\left(||X_{n}-X_{n-1}||_{\mathbb{R}^{m}}^{d_{1}}\leq y|X_{1},\ldots,X_{n-1}\right)\nonumber \\ 
	& =P^{(n)}\left(X_{n}\in B_{\mathbb{R}^{m}}\left(X_{n-1},y^{\frac{1}{d_{1}}}\right)|\ X_{1},\ldots,X_{n-1}\right)\nonumber \\
	& =\int_{M\cap\left(B_{\mathbb{R}^{m}}\left(X_{n-1},y^{\frac{1}{d_{1}}}\right)\right)}p_{X_{n}}\left(x_{n}\right)dvol_{M}(x_{n})\nonumber \\
	& \leq K_{p}vol_{M}\left(M\cap B\left(X_{n-1},y^{\frac{1}{d_{1}}}\right)\right),
	\end{align}
	where the last inequality is coming from the condition (6) in Definition \ref{def:regular.manifold}.
	And by applying Lemma~\ref{lem:regular.bounded.volume}, $vol_{M}\left(M\cap B\left(X_{n-1},y^{\frac{1}{d_{1}}}\right)\right)$
	can be further bounded as 
	\begin{align}
	\label{eq:proof.upper.volume.bound}
	& vol_{M}\left(M\cap B\left(X_{n-1},y^{\frac{1}{d_{1}}}\right)\right)\nonumber \\
	& \leq \frac{m!}{d_{2}!}\min\left\{ y^{\frac{1}{d_{1}}},\tau_{g}\right\} ^{d_{2}-m}vol_{\mathbb{R}^{m}}\left(B\left(X_{n-1},\,y^{\frac{1}{d_{1}}}+\min\left\{ y^{\frac{1}{d_{1}}},\tau_{g}\right\} \right)\right)\ (\text{Lemma }\ref{lem:regular.bounded.volume})\nonumber \\
	& =\frac{m!}{d_{2}!}\omega_{m}\left(y^{\frac{d_{2}}{d_{1}}}2^{m}1(y^{\frac{1}{d_{1}}}\leq\tau_{g})+y^{\frac{d_{2}}{d_{1}}}\left(\frac{\tau_{g}}{y^{\frac{1}{d_{1}}}}\right)^{d_{2}-m}\left(1+\left(\frac{\tau_{g}}{y^{\frac{1}{d_{1}}}}\right)\right)^{m}1(y^{\frac{1}{d_{1}}}>\tau_{g})\right)\nonumber \\
	& \leq \frac{m!}{d_{2}!}\omega_{m}2^{m}\left(y^{\frac{d_{2}}{d_{1}}}1(y^{\frac{1}{d_{1}}}\leq\tau_{g})+y^{\frac{d_{2}}{d_{1}}}\left(\frac{\tau_{g}}{(2K_{I}\sqrt{m})^{\frac{1}{d_{1}}}}\right)^{d_{2}-m}1(y^{\frac{1}{d_{1}}}>\tau_{g})\right)\nonumber \\
	& \leq C_{K_{I},m}^{(\ref*{claim:proof.upper.conditionalcdf},1)}\max\left\{1,\tau_{g}^{d_{2}-m}\right\}y^{\frac{d_{2}}{d_{1}}},
	\end{align}
	where $C_{K_{I},m}^{(\ref*{claim:proof.upper.conditionalcdf},1)}=m!\omega_{m}2^{m}\left(2K_{I}\sqrt{m}\right)^{m}$.
	By applying \eqref{eq:proof.upper.conditionalcdf.bound.volume} and \eqref{eq:proof.upper.volume.bound}, we get the upper bound on the conditional cdf of $||X_{n}-X_{n-1}||_{\mathbb{R}^{m}}^{d_{1}}$
	given $X_{1},\ldots,X_{n-1}$ in \eqref{eq:proof.upper.conditionalcdf} as 
	\begin{align}
	P^{(n)}\left(||X_{n}-X_{n-1}||_{\mathbb{R}^{m}}^{d_{1}}\leq y|X_{1},\ldots,X_{n-1}\right) & \leq K_{p}C_{K_{I},m}^{(\ref*{claim:proof.upper.conditionalcdf},1)}\max\left\{1,\tau_{g}^{d_{2}-m}\right\}y^{\frac{d_{2}}{d_{1}}}\nonumber \\
	& \leq C_{K_{I},K_{p},m}^{(\ref*{claim:proof.upper.conditionalcdf})}\max\left\{1,\tau_{g}^{d_{2}-m}\right\}y^{\frac{d_{2}}{d_{1}}},
	\end{align}
	where $C_{K_{I},K_{p},m}^{(\ref*{claim:proof.upper.conditionalcdf})}=K_{p}C_{K_{I},m}^{(\ref*{claim:proof.upper.conditionalcdf},1)}=m!K_{p}\omega_{m}2^{m}\left(2K_{I}\sqrt{m}\right)^{m}$ is a constant depending only on $K_{I}$, $K_{p}$, and $m$.
	
\end{proof}

\noindent \textbf{Lemma~\ref{lem:upper.higherdim}.} \textit{Fix $\tau_{g},\,\tau_{\ell}\in(0,\infty]$, $K_{I}\in[1,\infty)$,
	$K_{v}\in(0,2^{-m}]$, $K_{p}\in[(2K_{I})^{m},\infty)$,
	$d_{1},\,d_{2}\in\mathbb{N}$, with $\tau_{g}\leq\tau_{\ell}$ and
	$1\leq d_{1} < d_{2} \leq m$. Let $X_{1},\ldots,X_{n}\sim
	P\in\mathcal{P}_{\tau_{g},\tau_{\ell},K_{I},K_{v},K_{p}}^{d_{2}}$.
	Then for all $L>0$,
	\begin{equation}
	\label{eq:proof.upper.higherdim}
	P^{(n)}\left[\overset{n-1}{\underset{i=1}{\sum}}\|X_{i+1}-X_{i}\|_{\mathbb{R}^{m}}^{d_{1}}\leq L\right]\leq\frac{\left(C_{K_{I},K_{p},m}^{(\ref*{lem:upper.higherdim})}\right)^{n-1}L^{\frac{d_{2}}{d_{1}}(n-1)}\max\left\{1,\tau_{g}^{(d_{2}-m)(n-1)}\right\}}{(n-1)^{\left(\frac{d_{2}}{d_{1}}-1\right)(n-1)}(n-1)!},
	\end{equation}
	where $C_{K_{I},K_{p},m}^{(\ref*{lem:upper.higherdim})}$ is a constant depending only on $K_{I}$, $K_{p}$, and $m$.
}

\begin{proof}[Proof of Lemma~\ref{lem:upper.higherdim}] Let $Y_{i}:=\|X_{i+1}-X_{i}\|_{\mathbb{R}^{m}}^{d_{1}}$,
	$i=1,\ldots,n-1$, and let $P^{(n)}_{\overset{n-2}{\underset{i=1}{\sum}}Y_{i}}$ be the cumulative distribution function of $\overset{n-2}{\underset{i=1}{\sum}}Y_{i}$. Then from Claim~\ref{claim:proof.upper.conditionalcdf}, probability of the $d_{1}$-squared
	length of the path being bounded by $L$, $P^{(n)}\left(\overset{n-1}{\underset{i=1}{\sum}}Y_{i}\leq L\right)$,
	is upper bounded as 
	\begin{align*}
	& P^{(n)}\left(\overset{n-1}{\underset{i=1}{\sum}}Y_{i}\leq L\right)\\
	& =\int_{0}^{L}P^{(n)}\left(Y_{n-1}\leq y_{n-1}|\ \overset{n-2}{\underset{i=1}{\sum}}Y_{i}=L-y_{n-1}\right)dP^{(n)}_{\overset{n-2}{\underset{i=1}{\sum}}Y_{i}}(L-y_{n-1})\\
	& \leq C_{K_{I},K_{p},m}^{(\ref*{claim:proof.upper.conditionalcdf})}\max\left\{1,\tau_{g}^{d_{2}-m}\right\}\int_{0}^{L}y_{n-1}^{\frac{d_{2}}{d_{1}}}dP^{(n)}_{\overset{n-2}{\underset{i=1}{\sum}}Y_{i}}(L-y_{n-1})\ (\text{Claim }\ref{claim:proof.upper.conditionalcdf})\\
	& =C_{K_{I},K_{p},m}^{(\ref*{claim:proof.upper.conditionalcdf})}\max\left\{1,\tau_{g}^{d_{2}-m}\right\}\\
	& \quad\times\left(\left[-y_{n-1}^{\frac{d_{2}}{d_{1}}}P\left(\overset{n-2}{\underset{i=1}{\sum}}Y_{i}\leq L-y_{n-1}\right)\right]_{0}^{L}+\int_{0}^{L}P\left(\overset{n-2}{\underset{i=1}{\sum}}Y_{i}\leq L-y_{n-1}\right)d\left(y_{n-1}^{\frac{d_{2}}{d_{1}}}\right)\right)\\
	& =C_{K_{I},K_{p},m}^{(\ref*{claim:proof.upper.conditionalcdf})}\max\left\{1,\tau_{g}^{d_{2}-m}\right\}\int_{0}^{L}P\left(\overset{n-2}{\underset{i=1}{\sum}}Y_{i}\leq L-y_{n-1}\right)\frac{d_{2}}{d_{1}}y_{n-1}^{\frac{d_{2}-d_{1}}{d_{1}}}dy_{n-1}.
	\end{align*}
	By repeating this argument, we get an upper bound of $P^{(n)}\left(\overset{n-1}{\underset{i=1}{\sum}}Y_{i}\leq L\right)$
	as 
	\[
	P^{(n)}\left(\overset{n-1}{\underset{i=1}{\sum}}Y_{i}\leq L\right)\leq\left(\frac{d_{2}}{d_{1}}C_{K_{I},K_{p},m}^{(\ref*{claim:proof.upper.conditionalcdf})}\max\left\{1,\tau_{g}^{d_{2}-m}\right\}\right)^{n-1}\int_{\overset{n-1}{\underset{i=1}{\sum}}y_{i}\leq L}\overset{n-1}{\underset{i=1}{\prod}}y_{i}^{\frac{d_{2}-d_{1}}{d_{1}}}dy.
	\]
	Hence we get a further upper bound of $P^{(n)}\left(\overset{n-1}{\underset{i=1}{\sum}}\|X_{i+1}-X_{i}\|_{\mathbb{R}^{m}}^{d_{1}}\leq L\right)$
	in \eqref{eq:proof.upper.higherdim} with applying the AM-GM inequality as 
	\begin{align*}
	& P^{(n)}\left(\overset{n-1}{\underset{i=1}{\sum}}\|X_{i+1}-X_{i}\|_{\mathbb{R}^{m}}^{d_{1}}\leq L\right)\\
	& \leq\left(\frac{d_{2}}{d_{1}}C_{K_{I},K_{p},m}^{(\ref*{claim:proof.upper.conditionalcdf})}\max\left\{1,\tau_{g}^{d_{2}-m}\right\}\right)^{n-1}\int_{\overset{n-1}{\underset{i=1}{\sum}}y_{i}\leq L}\overset{n-1}{\underset{i=1}{\prod}}y_{i}^{\frac{d_{2}-d_{1}}{d_{1}}}dy\\
	& \leq\left(C_{K_{I},K_{p},m}^{(\ref*{lem:upper.higherdim})}\right)^{n-1}L^{\frac{d_{2}}{d_{1}}(n-1)}\max\left\{1,\tau_{g}^{(d_{2}-m)(n-1)}\right\}\\
	& \quad\times\int_{\overset{n-1}{\underset{i=1}{\sum}}y_{i}\leq1}\left(\frac{1}{n-1}\overset{n-1}{\underset{i=1}{\sum}}y_{i}\right)^{\frac{(d_{2}-d_{1})(n-1)}{d_{1}}}dy_{n-1}\cdots dy_{1}\ \text{(by AM-GM inequality)}\\
	& =\frac{\left(C_{K_{I},K_{p},m}^{(\ref*{lem:upper.higherdim})}\right)^{n-1}L^{\frac{d_{2}}{d_{1}}(n-1)}\max\left\{1,\tau_{g}^{(d_{2}-m)(n-1)}\right\}}{(n-1)^{\left(\frac{d_{2}}{d_{1}}-1\right)(n-1)}}\\
	& \quad\times\int_{0}^{1}\int_{\overset{n-2}{\underset{i=1}{\sum}}y_{i}\leq z}z^{\frac{(d_{2}-d_{1})(n-1)}{d_{1}}}dy_{n-2}\cdots dy_{1}dz\\
	& =\frac{\left(C_{K_{I},K_{p},m}^{(\ref*{lem:upper.higherdim})}\right)^{n-1}L^{\frac{d_{2}}{d_{1}}(n-1)}\max\left\{1,\tau_{g}^{(d_{2}-m)(n-1)}\right\}}{(n-1)^{\left(\frac{d_{2}}{d_{1}}-1\right)(n-1)}(n-2)!}\int_{0}^{1}z^{\frac{d_{2}(n-1)}{d_{1}}-1}dz\\
	& \leq\frac{\left(C_{K_{I},K_{p},m}^{(\ref*{lem:upper.higherdim})}\right)^{n-1}L^{\frac{d_{2}}{d_{1}}(n-1)}\max\left\{1,\tau_{g}^{(d_{2}-m)(n-1)}\right\}}{(n-1)^{\left(\frac{d_{2}}{d_{1}}-1\right)(n-1)}(n-1)!},
	\end{align*}
	where $C_{K_{I},K_{p},m}^{(\ref*{lem:upper.higherdim})}= m C_{K_{I},K_{p},m}^{(\ref*{claim:proof.upper.conditionalcdf})}$ is a constant depending only on $K_{I}$, $K_{p}$, and $m$.
	
\end{proof}

\begin{lem} 
\label{lem:proof.upper.spacefilling}
{\em (Space-filling curve)} There exists a surjective map
 $\psi_{d}:\ [0,1]\rightarrow[0,1]^{d}$ which is H{\"o}lder
continuous of order $1/d$, i.e.
\begin{equation}
0\leq\forall s,t\leq1,\ \|\psi_{d}(s)-\psi_{d}(t)\|_{\mathbb{R}^{d}}\leq2\sqrt{d+3}|s-t|^{1/d}.
\end{equation}
Such a map is called a space-filling curve.
\end{lem}

\begin{proof}[Proof of Lemma~\ref{lem:proof.upper.spacefilling}]
\citep[See][Chapter 2.1.6]{Buchin2008.ch2}.
\end{proof}

\noindent \textbf{Lemma~\ref{lem:upper.lowerdim}.} \textit{Fix $\tau_{g},\,\tau_{\ell}\in(0,\infty]$, $K_{I}\in[1,\infty)$, $K_{v}\in(0,2^{-m}]$, $d_{1}\in\mathbb{N}$, with $\tau_{g}\leq\tau_{\ell}$. Let
	$M\in\mathcal{M}_{\tau_{g},\tau_{\ell},K_{p},K_{v}}^{d_{1}}$ and
	$X_{1},\ldots,X_{n}\in M$. Then 
	\begin{equation}
	\label{eq:proof.upper.lowerdim}
	\underset{\sigma\in S_{n}}{\min} \overset{n-1}{\underset{i=1}{\sum}}
	\|X_{\sigma(i+1)}-X_{\sigma(i)}\|_{\mathbb{R}^{m}}^{d_{1}} \leq 
	C_{K_{I},K_{v},m}^{(\ref*{lem:upper.lowerdim})}\max\left\{1,\tau_{g}^{d_{1}-m}\right\},
	\end{equation}
	where $C_{K_{I},K_{v},m}^{(\ref*{lem:upper.lowerdim})}$ is a constant depending only on $K_{I}$, $K_{v}$, and $m$.
}

\begin{proof}[Proof of Lemma~\ref{lem:upper.lowerdim}]
When $d_{1}=1$, the length of TSP path is bounded by the length of the curve
$vol_{M}(M)$ as in Figure \ref{fig:upper.curve}, and Lemma
\ref{lem:regular.bounded.volume} implies 
$
vol_{M}(M)\leq C_{K_{I},m}^{(\ref*{lem:regular.bounded.volume})}\max\left\{1,\tau_{g}^{1-m}\right\},
$
hence
$C_{K_{I},K_{v},m}^{(\ref*{lem:upper.lowerdim})}$ can be set as
$C_{K_{I},K_{v},m}^{(\ref*{lem:upper.lowerdim})}=C_{K_{I},m}^{(\ref*{lem:regular.bounded.volume})}$, as described
before.

Consider $d_{1}>1$, and let $r:=2\sqrt{3}\tau_{g}$. By scaling the space-filling curve in Lemma
\ref{lem:proof.upper.spacefilling}, there exists a surjective map
$\psi_{d_{1}}:[0,1]\rightarrow[-r,r]^{d_{1}}$ and
$\psi_{m}:[0,1]\rightarrow[-K_{I},K_{I}]^{m}$ that satisfies
\begin{align}
\label{eq:proof.upper.spacefilling.d1} 0\leq\forall s,t\leq1, & \  \|\psi_{d_{1}}(s)-\psi_{d_{1}}(t)\|_{\mathbb{R}^{d_{1}}}\leq4r\sqrt{d_{1}+3}|s-t|^{1/d_{1}}\\
\label{eq:proof.upper.spacefilling.m} 0\leq\forall s,t\leq1, & \  \|\psi_{m}(s)-\psi_{m}(t)\|_{\mathbb{R}^{m}}\leq4K_{I}\sqrt{m+3}|s-t|^{1/m}
\end{align}

Now, from Lemma
\ref{lem:regular.bounded.covernumber}, $M$ can be covered by $N$ balls 
of radius $r$, denoted by
\begin{equation}
\label{eq:proof.upper.covering}
B_{M}(p_{1},r),\,\ldots,\,B_{M}(p_{N},r),
\end{equation}
with $N\leq\left\lfloor \frac{2^{d_{1}}vol_{M}(M)}{K_{v}r^{d_{1}}\omega_{d_{1}}}\right\rfloor$. 
Since $\psi_{m}:[0,1]\rightarrow[-K_{I},K_{I}]^{m}$ in \eqref{eq:proof.upper.spacefilling.m} is surjective,
we can find a right inverse
$\Psi_{m}:[-K_{I},K_{I}]^{m}\rightarrow[0,1]$ that satisfies
$\psi_{m}(\Psi_{m}(p))=p$, i.e.
\begin{equation}
\label{eq:proof.upper.rightinverse.m}
\xymatrix{
[0,1]\ar@/^/[rr]^{\psi_{m}} & & [-K_{I},K_{I}]^{m}.\ar@/^/[ll]^{\Psi_{m}}
}
\end{equation}
Reindex $p_{k}$ with respect to $\Psi_{m}$ so that
\begin{equation}
\label{eq:proof.upper.order.m}
\Psi_{m}(p_{1})<\cdots<\Psi_{m}(p_{N}).
\end{equation}
Now fix $k$, and consider the ball $B_{M}(p_{k},r)$ in the covering in \eqref{eq:proof.upper.covering}. Then for all $p\in B_{M}(p_{k},r)$, since
$d_{M}(p_{k},p)<r$, the condition (3) in Definition \ref{def:regular.manifold} implies that we can find $\varphi_{k}(p)\in
B_{\mathbb{R}^{d_{1}}}(0,r)$ such that
$\exp_{p_{k}}(\varphi_{k}(p))=p$. So this shows
\[
B_{M}(p_{k},r)\subset\exp_{p_{k}}\left(B_{\mathbb{R}^{d_{1}}}(0,r)\right).
\]
Now consider the composition of the exponential map $\exp_{p_{k}}$ and $\psi_{d_{1}}$ in \eqref{eq:proof.upper.spacefilling.d1},  $\exp_{p_{k}}\circ\psi_{d_{1}}:[0,1]\rightarrow M$. Then 
\[
B_{M}(p_{k},r)\subset\exp_{p_{k}}\left(B_{\mathbb{R}^{d_{1}}}(0,r)\right)\subset\exp_{p_{k}}\left([-r,r]^{d_{1}}\right)=\exp_{p_{k}}\circ\psi_{d_{1}}\left([0,1]\right),
\]
where the last equality is from that $\psi_{d_{1}}$ in \eqref{eq:proof.upper.spacefilling.d1} is surjective. So $\exp_{p_{k}}\circ\psi_{d_{1}}:[0,1]\rightarrow M$ is surjective on $B_{M}(p,r)$, so we can find right inverse $\Psi_{k}:B_{M}(p_{k},r)\rightarrow[0,1]$ that satisfies $(\exp_{p_{k}}\circ\psi_{d_{1}})(\Psi_{k}(p))=p$, i.e.
\begin{equation}
\label{eq:proof.upper.rightinverse.k}
\xymatrix{
[0,1]\ar@/^/[rr]^{\psi_{d_{1}}} & & [-r,r]\ar@/^/[rr]^{\exp_{p_{k}}} & & M \supset B_{M}(p_{k},r).\ar@/^1pc/[llll]^{\Psi_{k}}
}
\end{equation}
Then, reindex $X_{1},\ldots,X_{n}$ with respect to $\Psi_{m}$ and
$\Psi_{k}$ as $\{X_{k,j}\}_{1\leq k\leq N,\ 1\leq j\leq n_{k}}$,
where $X_{k,1},\ldots,X_{k,n_{k}}\in B_{M}(p_{k},r)$ and 
\begin{equation}
\label{eq:proof.upper.order.k}
\Psi_{k}(X_{k,1})<\cdots<\Psi_{k}(X_{k,n_{k}}).
\end{equation}
Let $\sigma\in S_{n}$ be the corresponding order of index, so that the
$d_{1}$-squared length of the path $\overset{n-1}{\underset{i=1}{\sum}}\|X_{\sigma(i+1)}-X_{\sigma(i)}\|_{\mathbb{R}^{m}}^{d_{1}}$
is factorized as 
\begin{equation}
\label{eq:proof.upper.path.factorize}
\overset{n-1}{\underset{i=1}{\sum}}\|X_{\sigma(i+1)}-X_{\sigma(i)}\|_{\mathbb{R}^{m}}^{d_{1}}=\overset{N}{\underset{k=1}{\sum}}\overset{n_{k}-1}{\underset{j=1}{\sum}}\|X_{k,j+1}-X_{k,j}\|_{\mathbb{R}^{m}}^{d_{1}}+\overset{N-1}{\underset{k=1}{\sum}}\|X_{k+1,1}-X_{k,n_{k}}\|_{\mathbb{R}^{m}}^{d_{1}}.
\end{equation}
First, consider the first term $\overset{N}{\underset{k=1}{\sum}}\overset{n_{k}-1}{\underset{j=1}{\sum}}\|X_{k,j+1}-X_{k,j}\|_{\mathbb{R}^{m}}^{d_{1}}$
in \eqref{eq:proof.upper.path.factorize}. For all $1\leq k\leq N$, by applying Lemma~\ref{lem:regular.bounded.length}, $\overset{n_{k}-1}{\underset{j=1}{\sum}}\|X_{k,j+1}-X_{k,j}\|_{\mathbb{R}^{m}}^{d_{1}}$
is upper bounded as 
\begin{align*}
& \overset{n_{k}-1}{\underset{j=1}{\sum}}\|X_{k,j+1}-X_{k,j}\|_{\mathbb{R}^{m}}^{d_{1}}\\
& \leq\overset{n_{k}-1}{\underset{j=1}{\sum}}\|(\exp_{p_{k}}\circ\psi_{d_{1}})(\Psi_{k}(X_{k,j+1}))-(\exp_{p_{k}}\circ\psi_{d_{1}})(\Psi_{k}(X_{k,j}))\|_{\mathbb{R}^{m}}^{d_{1}}\ (\text{from }\eqref{eq:proof.upper.rightinverse.k})\\
& \leq\left(\frac{\sinh(\sqrt{2}r/\tau_{\ell})}{\sqrt{2}r/\tau_{\ell}}\right)^{d_{1}}\overset{n_{k}-1}{\underset{j=1}{\sum}}\|\psi_{d_{1}}(\Psi_{k}(X_{k,j+1}))-\psi_{d_{1}}(\Psi_{k}(X_{k,j}))\|_{\mathbb{R}^{d_{1}}}^{d_{1}}\ (\text{Lemma }\ref{lem:regular.bounded.length})\\
& \leq\left(\frac{2\sqrt{2(d_{1}+3)}\sinh(\sqrt{2}r/\tau_{\ell})}{r/\tau_{\ell}}\right)^{d_{1}}r^{d_{1}}\overset{n_{k}-1}{\underset{j=1}{\sum}}|\Psi_{k}(X_{k,j+1})-\Psi_{k}(X_{k,j})|\ (\text{from }\eqref{eq:proof.upper.spacefilling.d1})\\
& \leq\left(\frac{2\sqrt{2(d_{1}+3)}\sinh(\sqrt{2}r/\tau_{\ell})}{r/\tau_{\ell}}\right)^{d_{1}}r^{d_{1}}\ (\text{from }\eqref{eq:proof.upper.order.k}).
\end{align*}
Then, by applying the fact that $r = 2\sqrt{3}\tau_{g} \leq 2\sqrt{3}\tau_{\ell}$ and that $t\mapsto\frac{\sinh t}{t}$ is an increasing function on $t\geq0$
to this, we have an upper bound of $\overset{n_{k}-1}{\underset{j=1}{\sum}}\|X_{k,j+1}-X_{k,j}\|_{\mathbb{R}^{m}}^{d_{1}}$
as
\begin{equation}
\label{eq:proof.upper.path.factorize.withinball}
\overset{n_{k}-1}{\underset{j=1}{\sum}}\|X_{k,j+1}-X_{k,j}\|_{\mathbb{R}^{m}}^{d_{1}} \leq\left(\frac{\sqrt{2(d_{1}+3)}\sinh2\sqrt{6}}{\sqrt{3}}\right)^{d_{1}}r^{d_{1}}.
\end{equation}
And then, the second term  $\overset{N-1}{\underset{k=1}{\sum}}\|X_{k+1,1}-X_{k,n_{k}}\|_{\mathbb{R}^{m}}^{d_{1}}$
in \eqref{eq:proof.upper.path.factorize} is upper bounded as 
\begin{align}
\label{eq:proof.upper.path.factorize.betweenball}
& \overset{N-1}{\underset{k=1}{\sum}}\|X_{k+1,1}-X_{k,n_{k}}\|_{\mathbb{R}^{m}}^{d_{1}}\nonumber \\
& \leq 3^{d_{1}-1}\overset{N-1}{\underset{k=1}{\sum}}\left(\|X_{k+1,1}-p_{k+1}\|_{\mathbb{R}^{m}}^{d_{1}}+\|p_{k+1}-p_{k}\|_{\mathbb{R}^{m}}^{d_{1}}+\|p_{k}-X_{k,n_{k}}\|_{\mathbb{R}^{m}}^{d_{1}}\right)\nonumber \\
& \leq 2\cdot 3^{d_{1}-1}(N-1)r^{d_{1}}+3^{d_{1}-1}\overset{N-1}{\underset{k=1}{\sum}}\|\psi_{m}(\Psi_{m}(p_{k+1}))-\psi_{m}(\Psi_{m}(p_{k}))\|_{\mathbb{R}^{d_{1}}}^{d_{1}}\ (\text{from }\eqref{eq:proof.upper.rightinverse.m})\nonumber \\
& < 3^{d_{1}}(N-1)r^{d_{1}}+2\cdot 3^{d_{1}}\sqrt{m+3}K_{I}\overset{N-1}{\underset{k=1}{\sum}}|\Psi_{m}(p_{k+1})-\Psi_{m}(p_{k})|^{\frac{d_{1}}{m}}\ (\text{from }\eqref{eq:proof.upper.spacefilling.m})\nonumber \\
& \leq 3^{d_{1}}(N-1)r^{d_{1}}+\nonumber \\
& \qquad 2 \cdot 3^{d_{1}}\sqrt{m+3}K_{I}\left(\overset{N-1}{\underset{k=1}{\sum}}|\Psi_{m}(p_{k+1})-\Psi_{m}(p_{k})|^{\frac{d_{1}}{m}\times\frac{m}{d_{1}}}\right)^{\frac{d_{1}}{m}}\left(\overset{N-1}{\underset{k=1}{\sum}}1^{\frac{m}{m-d_{1}}}\right)^{\frac{m-d_{1}}{m}}\nonumber \\
&\qquad(\text{using H{\"o}lder's inequality})\nonumber \\
& \leq 3^{d_{1}}(N-1)r^{d_{1}}+2\cdot 3^{d_{1}}\sqrt{m+3}K_{I}(N-1)^{1-\frac{d_{1}}{m}}\ (\text{from }\eqref{eq:proof.upper.order.m}).
\end{align}
Hence, by plugging in \eqref{eq:proof.upper.path.factorize.withinball} and \eqref{eq:proof.upper.path.factorize.betweenball} to \eqref{eq:proof.upper.path.factorize}, $\overset{n-1}{\underset{i=1}{\sum}}\|X_{\sigma(i+1)}-X_{\sigma(i)}\|_{\mathbb{R}^{m}}^{d_{1}}$ is upper bounded 
as 
\begin{align*}
& \overset{n-1}{\underset{i=1}{\sum}}\|X_{\sigma(i+1)}-X_{\sigma(i)}\|_{\mathbb{R}^{m}}^{d_{1}}\\
& <\left(\left(\frac{\sqrt{2(d_{1}+3)}\sinh2\sqrt{6}}{\sqrt{3}}\right)^{d_{1}}+3^{d_{1}}\right)r^{d_{1}}N+2\cdot 3^{d_{1}}\sqrt{m+3}K_{I}N^{1-\frac{d_{1}}{m}}\\
& <\frac{\left(2\sqrt{d_{1}+3}\sinh2\sqrt{6}\right)^{d_{1}}+6^{d_{1}}}{K_{v}\omega_{d_{1}}}vol_{M}(M) +\frac{2\cdot 3^{\frac{d_{1}}{2}}\sqrt{m+3}K_{I}}{\left(K_{v}\omega_{d_{1}}\right)^{1-\frac{d_{1}}{m}}}\tau_{g}^{d_{1}\left(\frac{d_{1}}{m}-1\right)}\left(vol_{M}(M)\right)^{1-\frac{d_{1}}{m}}\\
& \leq\frac{\left(2(\sinh2\sqrt{6})\sqrt{m+3}\right)^{d_{1}}2K_{I}}{\min\left\{ 1,K_{v}\omega_{d_{1}}\right\} }\times \\
& \quad\left(C_{K_{I},m}^{(3)}\max\left\{ 1,\tau_{g}^{d_{1}-m}\right\} +\tau_{g}^{d_{1}\left(\frac{d_{1}}{m}-1\right)}\left(C_{K_{I},m}^{(3)}\max\left\{ 1,\tau_{g}^{d_{1}-m}\right\} \right)^{1-\frac{d_{1}}{m}}\right)\ (\text{from Lemma~\ref{lem:regular.bounded.volume}})\\
& \leq C_{K_{I},K_{v},m}^{(\ref*{lem:upper.lowerdim})}\max\left\{1,\tau_{g}^{d_{1}-m}\right\},
\end{align*}
with some constant $C_{K_{I},K_{v},m}^{(\ref*{lem:upper.lowerdim})}$
which depends only on $m$, $K_{v}$, and $K_{I}$.
Hence we have the same upper bound for $\underset{\sigma\in S_{n}}{\min}\overset{n-1}{\underset{i=1}{\sum}}\|X_{\sigma(i+1)}-X_{\sigma(i)}\|_{\mathbb{R}^{m}}^{d_{1}}$
as well, as in \eqref{eq:proof.upper.lowerdim}. \end{proof}

\noindent \textbf{Proposition~\ref{prop:upper.maximumrisk}.} \textit{Fix $\tau_{g},\,\tau_{\ell}\in(0,\infty]$, $K_{I}\in[1,\infty)$,
	$K_{v}\in(0,2^{-m}]$, $K_{p}\in[(2K_{I})^{m},\infty)$,
	$d_{1},\,d_{2}\in\mathbb{N}$, with $\tau_{g}\leq\tau_{\ell}$ and
	$1\leq d_{1} < d_{2} \leq m$. Let $\hat{d}_{n}$ be in \eqref{eq:upper.dim.estimator}. Then either for $d=d_{1}$ or $d=d_{2}$,
	\begin{align}
	\label{eq:proof.upper.maximumrisk}
	& \underset{P\in\mathcal{P}_{\tau_{g},\tau_{\ell},K_{I},K_{v},K_{p}}^{d}}{\sup}
	\mathbb{E}_{P^{(n)}}\left[\ell\left(\hat{d}_{n},d(P)\right)\right]\nonumber \\
	& \leq 1(d=d_{2})\left(C_{K_{I},K_{p},K_{v},m}^{(\ref*{prop:upper.maximumrisk})}\right)^{n}
	\max\left\{1,\tau_{g}^{-\left(\frac{d_{2}}{d_{1}}m+m-2d_{2}\right)n}\right\}n^{-\left(\frac{d_{2}}{d_{1}}-1\right)n},
	\end{align}
	where $C_{K_{I},K_{p},K_{v},m}^{(\ref*{prop:upper.maximumrisk})}\in(0,\infty)$ is a constant depending only on $K_{I}$, $K_{p}$, $K_{v}$, and $m$.
}
\begin{proof}[Proof of Proposition~\ref{prop:upper.maximumrisk}]
Consider first the case $d = d_{1}$.
Then for all $P\in\mathcal{P}_{\tau_{g},\tau_{\ell},K_{I},K_{v},K_{p}}^{d_{1}}$ and $X_{1},\ldots,X_{n}\sim P$, by Lemma~\ref{lem:upper.lowerdim}, 
\[
\underset{\sigma\in S_{n}}{\min}\left\{ \overset{n-1}{\underset{i=1}{\sum}}\|X_{\sigma(i+1)}-X_{\sigma(i)}\|_{\mathbb{R}^{m}}^{d_{1}}\right\}\leq C_{K_{I},K_{v},m}^{(\ref*{lem:upper.lowerdim})}\max\left\{1,\tau_{g}^{d_{1}-m}\right\},
\]
hence $\hat{d}_{n}$ in \eqref{eq:upper.dim.estimator} always satisfies $\hat{d}_{n}(X)=d_{1}=d(P)$,
i.e. the risk of $\hat{d}_{n}$ satisfies
\begin{equation}
\label{eq:proof.upper.estimator.upperdim}
P^{(n)}\left[\hat{d}_{n}(X_{1},\ldots,X_{n})=d_{2}\right]=0.
\end{equation}
For the case when $d=d_{2}$, for all $P\in\mathcal{P}_{\tau_{g},\tau_{\ell},K_{I},K_{v},K_{p}}^{d_{2}}$, the risk of $\hat{d}_{n}$ in \eqref{eq:upper.dim.estimator} is upper bounded as
\begin{align}
\label{eq:proof.upper.estimator.lowerdim}
 & P^{(n)}\left[\hat{d}_{n}(X_{1},\ldots,X_{n})=d_{1}\right]\nonumber \\
 & = P\left[\underset{\sigma\in S_{n}}{\bigcup}\overset{n-1}{\underset{i=1}{\sum}}|X_{\sigma(i+1)}-X_{\sigma(i)}|\leq C_{K_{I},K_{v},m}^{(\ref*{lem:upper.lowerdim})}\max\left\{1,\tau_{g}^{d_{1}-m}\right\} \right]\nonumber\\
 & \leq \underset{\sigma\in S_{n}}{\sum}P\left[\overset{n-1}{\underset{i=1}{\sum}}|X_{\sigma(i+1)}-X_{\sigma(i)}|\leq C_{K_{I},K_{v},m}^{(\ref*{lem:upper.lowerdim})}\max\left\{1,\tau_{g}^{d_{1}-m}\right\} \right]\nonumber\\
 & = n!P\left[\overset{n-1}{\underset{i=1}{\sum}}|X_{i+1}-X_{i}|\leq C_{K_{I},K_{v},m}^{(\ref*{lem:upper.lowerdim})}\max\left\{1,\tau_{g}^{d_{1}-m}\right\} \right]\nonumber\\
 & = \frac{n\left(C_{K_{I},K_{p},m}^{(\ref*{lem:upper.higherdim})}\right)^{n-1}\left(C_{K_{I},K_{v},m}^{(\ref*{lem:upper.lowerdim})}\max\left\{1,\tau_{g}^{d_{1}-m}\right\} \right)^{\frac{d_{2}}{d_{1}}(n-1)}\max\left\{1,\tau_{g}^{(d_{2}-m)(n-1)}\right\}}{(n-1)^{\left(\frac{d_{2}}{d_{1}}-1\right)(n-1)}},
\end{align}
where the last line is implied by Lemma~\ref{lem:upper.higherdim}. Therefore, by combining \eqref{eq:proof.upper.estimator.upperdim} and \eqref{eq:proof.upper.estimator.lowerdim}, the risk is upper bounded as in \eqref{eq:proof.upper.maximumrisk}, as
\begin{align*}
 & \underset{P\in\mathcal{P}_{\tau_{g},\tau_{\ell},K_{I},K_{v},K_{p}}^{d}}{\sup}
\mathbb{E}_{P^{(n)}}\left[\ell\left(\hat{d}_{n},d(P)\right)\right]\\
 & \leq 1(d=d_{2})\frac{n\left(C_{K_{I},K_{p},m}^{(\ref*{lem:upper.higherdim})}\left(C_{K_{I},K_{v},m}^{(\ref*{lem:upper.lowerdim})}\right)^{\frac{d_{2}}{d_{1}}}\right)^{n-1}\max\left\{1,\tau_{g}^{-\left(\frac{d_{2}}{d_{1}}m+m-2d_{2}\right)(n-1)}\right\}}{(n-1)^{\left(\frac{d_{2}}{d_{1}}-1\right)(n-1)}}\\
 & \leq 1(d=d_{2})\left(C_{K_{I},K_{p},K_{v},m}^{(\ref*{prop:upper.maximumrisk})}\right)^{n}\max\left\{1,\tau_{g}^{-\left(\frac{d_{2}}{d_{1}}m+m-2d_{2}\right)n}\right\}n^{-\left(\frac{d_{2}}{d_{1}}-1\right)n},
\end{align*}
for some $C_{K_{I},K_{p},K_{v},m}^{(\ref*{prop:upper.maximumrisk})}$ that depends only on $K_{I}$, $K_{p}$, $K_{v}$, and $m$.
\end{proof}

\noindent \textbf{Proposition~\ref{prop:upper.bound}.} \textit{Fix $\tau_{g},\,\tau_{\ell}\in(0,\infty]$, $K_{I}\in[1,\infty)$,
	$K_{v}\in(0,2^{-m}]$, $K_{p}\in[(2K_{I})^{m},\infty)$,
	$d_{1},\,d_{2}\in\mathbb{N}$, with $\tau_{g}\leq\tau_{\ell}$ and
	$1\leq d_{1} < d_{2} \leq m$. Then
	\begin{align}
	\label{eq:proof.upper.bound}
	& \underset{\hat{d}_{n}}{\inf}
	\underset{P\in\mathcal{P}_{1}\cup\mathcal{P}_{2}}{\sup}
	\mathbb{E}_{P^{(n)}}\left[\ell\left(\hat{d}_{n},d(P)\right)\right]\nonumber \\
	& \leq \left(C_{K_{I},K_{p},K_{v},m}^{(\ref*{prop:upper.maximumrisk})}\right)^{n}
	\max\left\{1,\tau_{g}^{-\left(\frac{d_{2}}{d_{1}}m+m-2d_{2}\right)n}\right\}n^{-\left(\frac{d_{2}}{d_{1}}-1\right)n},
	\end{align}
	where $C_{K_{I},K_{p},K_{v},m}^{(\ref*{prop:upper.maximumrisk})}$ is from Proposition~\ref{prop:upper.maximumrisk} and
	\[
	\mathcal{P}_1 =
	\mathcal{P}_{\tau_{g},\tau_{\ell},K_{I},K_{v},K_{p}}^{d_{1}},\ \ \ 
	\mathcal{P}_2 =
	\mathcal{P}_{\tau_{g},\tau_{\ell},K_{I},K_{v},K_{p}}^{d_{2}}.
	\]
}

\begin{proof}[Proof of Proposition~\ref{prop:upper.bound}]
	Applying Proposition~\ref{prop:upper.maximumrisk} to \eqref{eq:upper.maximumrisk.bound} yields
	\begin{align*}
	& \underset{\hat{d}_{n}}{\inf}
	\underset{P\in\mathcal{P}_{\tau_{g},\tau_{\ell},K_{I},K_{v},K_{p}}^{d_{1}}\cup\mathcal{P}_{\tau_{g},\tau_{\ell},K_{I},K_{v},K_{p}}^{d_{2}}}{\sup}
	\mathbb{E}_{P^{(n)}}\left[\ell\left(\hat{d}_{n},d(P)\right)\right] \\
	& \leq
	\underset{P\in\mathcal{P}_{\tau_{g},\tau_{\ell},K_{I},K_{v},K_{p}}^{d_{1}}\cup\mathcal{P}_{\tau_{g},\tau_{\ell},K_{I},K_{v},K_{p}}^{d_{2}}}{\sup}
	\mathbb{E}_{P^{(n)}}\left[\ell\left(\hat{d}_{n},d(P)\right)\right] \\
	& \leq \left(C_{K_{I},K_{p},K_{v},m}^{(\ref*{prop:upper.maximumrisk})}\right)^{n}\max\left\{1,\tau_{g}^{-\left(\frac{d_{2}}{d_{1}}m+m-2d_{2}\right)n}\right\}n^{-\left(\frac{d_{2}}{d_{1}}-1\right)n}.
	\end{align*}
	Hence the minimax rate $R_{n}$ in \eqref{eq:regular.minimax} is upper bounded as in \eqref{eq:proof.upper.bound}.
\end{proof}

\section{Proofs for Section \ref{sec:lower}}
\label{sec:proof.lower}

\noindent \textbf{Lemma~\ref{lem:lower.cylinder.regularity}.} \textit{Fix $\tau_{g},\,\tau_{\ell}\in(0,\infty]$,
	$K_{I}\in[1,\infty)$, $K_{v}\in(0,2^{-m}]$, $d,\,\Delta
	d\in\mathbb{N}$, with $\tau_{g}\leq\tau_{\ell}$ and $1\leq d+\Delta d
	\leq m$. Let $M\in\mathcal{M}_{\tau_{g},\tau_{\ell},K_{I},K_{v}}^{d}$
	be a $d$-dimensional manifold of global reach $\geq\tau_{g}$, local
	reach $\geq\tau_{\ell}$, which is embedded in $\mathbb{R}^{m-\Delta d}$.
	Then
	\begin{equation}
		\label{eq:proof.lower.cylinder.regularity}
		M\times[-K_{I},K_{I}]^{\Delta d}\in\mathcal{M}_{\tau_{g},\tau_{\ell},K_{I},K_{v}}^{d+\Delta d},
	\end{equation}
	which is embedded in $\mathbb{R}^{m}$.
}

\begin{proof}[Proof of Lemma~\ref{lem:lower.cylinder.regularity}]
	For showing \eqref{eq:proof.lower.cylinder.regularity}, we need to
	show 4 conditions in Definition \ref{def:regular.manifold}. 
	The other
	conditions are rather obvious and the critical condition is (2),
	i.e. the global reach condition and the local reach condition. Showing the local
	reach condition is almost identical to showing the global reach condition,
	so we will focus on the global reach condition. From the definition of the
	global reach in Definition \ref{def:definition.reach}, we need to show
	that for all $x\in\mathbb{R}^{m}$ with
	$dist_{\mathbb{R}^{m}}(x,M\times[-K_{I},K_{I}]^{\Delta d})<\tau_{g}$,
	$x$ has the unique closest point $\pi_{M\times[-K_{I},K_{I}]^{\Delta  d}}(x)$ on $M\times[-K_{I},K_{I}]$.
	
	Let $x\in\mathbb{R}^{m}$ be satisfying $dist_{\mathbb{R}^{m}}(x,M\times[-K_{I},K_{I}]^{\Delta d})<\tau_{g}$,
	and let $y\in M\times[-K_{I},K_{I}]^{\Delta d}$. Then the distance
	between $x$ and $y$ can be factorized as their distance on first
	$m-\Delta d$ coordinates and last $\Delta d$ coordinates, 
	\begin{align}
		\label{eq:proof.lower.cylinderdistance.factorize}
		& dist_{\mathbb{R}^{m}}\left(x,\,y\right)\nonumber\\
		& =\sqrt{dist_{\mathbb{R}^{m-\Delta d}}\left(\Pi_{1:m-\Delta d}(x),\,
		\Pi_{1:m-\Delta d}\left(y\right)\right)^{2}+
		dist_{\mathbb{R}^{\Delta d}}\left(\Pi_{(m-\Delta d+1):m}(x),\,\Pi_{(m-\Delta d+1):m}(y)\right)^{2}}.
	\end{align}
	For the first term in
	\eqref{eq:proof.lower.cylinderdistance.factorize}, note that the
	projection map $\Pi_{1:m-\Delta   d}:\mathbb{R}^{m}\rightarrow\mathbb{R}^{m-\Delta d}$ is
	a contraction, i.e. for all $x,y\in\mathbb{R}^{m}$,
	$dist_{\mathbb{R}^{m-\Delta d}}(\Pi_{1:m-\Delta
		d}(x),\,\Pi_{1:m-\Delta d}(y))\leq dist_{\mathbb{R}^{m}}(x,\,y)$
	holds, so $\Pi_{1:m-\Delta d}(x)$ is also within
	a $\tau_{g}$-neighborhood of $M=\Pi_{1:m-\Delta
		d}(M\times[-K_{I},K_{I}]^{\Delta d})$, i.e.
	\begin{align*}
		dist_{\mathbb{R}^{m-\Delta d}}\left(\Pi_{1:m-\Delta d}(x),\,M\right) & =dist_{\mathbb{R}^{m-\Delta d}}\left(\Pi_{1:m-\Delta d}(x),\,\Pi_{1:m-\Delta d}(M\times[-K_{I},K_{I}]^{\Delta d})\right)\\
		& \leq dist_{\mathbb{R}^{m}}(x,\,M\times[-K_{I},K_{I}]^{\Delta d})<\tau_{g}.
	\end{align*}
	Hence from the definition of the global reach in Definition \ref{def:definition.reach}, $\pi_{M}\left(\Pi_{1:m-\Delta d}(x)\right)\in M$ uniquely exists. And from $\Pi_{1:m-\Delta d}(y)\in M$,
	the distance between $\Pi_{1:m-\Delta d}(x)$ and $\Pi_{1:m-\Delta d}\left(y\right)$
	is lower bounded by the distance between $\Pi_{1:m-\Delta d}(x)$ and $M$,
	i.e. 
	\begin{align}
		\label{eq:proof.lower.cylinderdistance.factorize.first}
		dist_{\mathbb{R}^{m-\Delta d}}\left(\Pi_{1:m-\Delta d}(x),\,\Pi_{1:m-\Delta d}\left(y\right)\right) & \geq dist_{\mathbb{R}^{m-\Delta d}}\left(\Pi_{1:m-\Delta d}(x),\,\pi_{M}\left(\Pi_{1:m-\Delta d}(x)\right)\right)\nonumber \\
		& =dist_{\mathbb{R}^{m-\Delta d}}\left(\Pi_{1:m-\Delta d}(x),\,M\right),
	\end{align}
	and the equality holds if and only if $\Pi_{1:m-\Delta d}\left(y\right)=\pi_{M}\left(\Pi_{1:m-\Delta d}(x)\right)$.
	
	The second term in \eqref{eq:proof.lower.cylinderdistance.factorize} is trivially lower bounded by $0$, i.e. 
	\begin{equation}
		\label{eq:proof.lower.cylinderdistance.factorize.second}
		dist_{\mathbb{R}^{\Delta d}}\left(\Pi_{(m-\Delta d+1):m}(x),\,\Pi_{(m-\Delta d+1):m}(y)\right)\geq 0,
	\end{equation}
	and the equality holds if and only if $\Pi_{(m-\Delta d+1):m}(x)=\Pi_{(m-\Delta d+1):m}(y)$.
	
	Hence by applying \eqref{eq:proof.lower.cylinderdistance.factorize.first} and \eqref{eq:proof.lower.cylinderdistance.factorize.second} to \eqref{eq:proof.lower.cylinderdistance.factorize}, $dist_{\mathbb{R}^{m}}\left(x,\,y\right)$
	is lower bounded by the distance between $\Pi_{1:m-\Delta d}(x)$ and $M$, i.e.
	\begin{align*}
		& dist_{\mathbb{R}^{m}}\left(x,\,y\right)\\
		& =\sqrt{dist_{\mathbb{R}^{m-\Delta d}}\left(\Pi_{1:m-\Delta d}(x),\,\Pi_{1:m-\Delta d}\left(y\right)\right)^{2}+dist_{\mathbb{R}^{\Delta d}}\left(\Pi_{(m-\Delta d+1):m}(x),\,\Pi_{(m-\Delta d+1):m}(y)\right)^{2}}\\
		& \geq dist_{\mathbb{R}^{m-\Delta d}}\left(\Pi_{1:m-\Delta d}(x),\,M\right),
	\end{align*}
	and the equality holds if and only if $\Pi_{1:m-\Delta d}\left(y\right)=\pi_{M}\left(\Pi_{1:m-\Delta d}(x)\right)$
	and $\Pi_{(m-\Delta d+1):m}(x)=\Pi_{(m-\Delta d+1):m}(y)$, i.e. when
	 $y=\left(\pi_{M}\left(\Pi_{1:m-\Delta d}(x)\right),\,\Pi_{(m-\Delta d+1):m}(x)\right)$.
	Hence $x$ has the unique closest point $\pi_{M\times[-K_{I},K_{I}]^{\Delta d}}(x)$
	on $M\times[-K_{I},K_{I}]$ as 
	\[
	\pi_{M\times[-K_{I},K_{I}]^{\Delta d}}(x)=\left(\pi_{M}\left(\Pi_{1:m-\Delta d}(x)\right),\ \Pi_{(m-\Delta d+1):m}(x)\right),
	\]
	as in Figure \ref{fig:proof.lower.product}.
\end{proof}

\begin{figure}
	\begin{center}
		\includegraphics{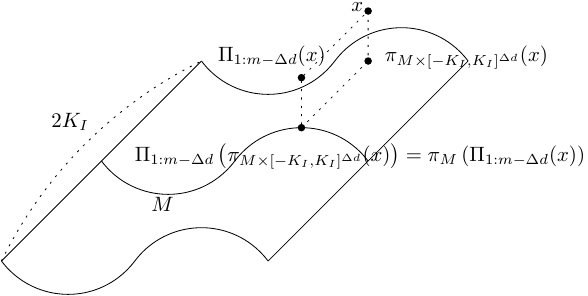}
	\end{center}
	\caption{$\pi_{M\times[-K_{I},K_{I}]^{\Delta d}}(x)$ satisfies $\Pi_{1:m-\Delta d}\left(\pi_{M\times[-K_{I},K_{I}]^{\Delta d}}(x)\right)=\pi_{M}\left(\Pi_{1:m-\Delta d}(x)\right)$.}
	\label{fig:proof.lower.product}
\end{figure}

%\begin{lem*}
%\ref{lem:lower.bounded.second.regularity}.
%Let $\gamma:[-K_{\delta},K_{\delta}]\rightarrow I$
%be a parametrized curve which is $C^{1}$ and piecewise $C^{2}$.
%Suppose that, for all $t\in[-K_{\delta},K_{\delta}]$, 
%\begin{equation}
%\|\gamma''(t)\|<\|\gamma'(t)\|_{2}^{2}\kappa_{\ell}.
%\end{equation}
%Then $image(\gamma)$ is of local reach $\geq\tau_{\ell}$.
%\end{lem*}
%
%\begin{proof}
%$\forall p\in image(\gamma)$, let $\epsilon>0$ be sufficiently small
%and $U_{p}=B(p,\epsilon)\cap image(\gamma)$ be an $\epsilon$-neighborhood
%of $p$. Let $U_{p}=\gamma(a,\,b)$, and $x\in\mathbb{R}^{m}$ be
%such that $d(x,\,U_{p})<\tau_{\ell}-\epsilon$. Then $\forall t\in(a,b)$,
%if $\gamma''(t)$ exists,
%
%\begin{align}
%\frac{d}{dt}(\gamma(t)-x)^{\top}(\gamma(t)-x) & =\gamma'(t)^{\top}(\gamma(t)-x)\\
%\frac{d^{2}}{dt^{2}}(\gamma(t)-x)^{\top}(\gamma(t)-x)|_{t=t_{0}} & =\gamma''(t_{0})^{\top}(\gamma(t_{0})-x)+\|\gamma'(t_{0})\|_{2}^{2}\\
%& >-\|\gamma''(t_{0})\|R^{l}+\|\gamma'(t_{0})\|_{2}^{2}>0
%\end{align}
%Since $\gamma$ is piecewise $C^{2}$, $\|\gamma(t)-x\|_{2}^{2}$
%is strictly convex function of $t\in(a,\,b)$. Hence a unique global
%minimizer $t_{0}$ exists, $\gamma(t_{0})=\pi_{U_{p}}(x)$, which
%is the unique projection of $x$ to $U_{p}$. Therefore, $image(\gamma)$
%is of local reach $\geq\tau_{\ell}-\epsilon$, for all $\epsilon>0$.
%And this asserts that $image(\gamma)$ is of local reach $\geq\tau_{\ell}$.
%\end{proof}

\noindent \textbf{Lemma~\ref{lem:lower.constructT}.} \textit{Fix $\tau_{\ell}\in(0,\infty]$, $K_{I}\in[1,\infty)$,
	$d_{1},\,d_{2}\in\mathbb{N}$, with $1\leq d_{1}\leq d_{2}$, and
	suppose $\tau_{\ell} < K_{I}$. Then there exist
	$T_{1},\cdots,T_{n}\subset[-K_{I},K_{I}]^{d_{2}}$ such that:\\
	(1) The $T_{i}$'s are distinct.\\
	(2) For each $T_{i}$, there exists an isometry $\Phi_{i}$ such that
	\begin{equation}
		T_{i}=\Phi_{i}\left([-K_{I},K_{I}]^{d_{1}-1}\times[0,a]\times B_{\mathbb{R}^{d_{2}-d_{1}}}(0,w)\right),
	\end{equation}
	where $c=\left\lceil \frac{K_{I}+\tau_{\ell}}{2\tau_{\ell}}\right\rceil $, $a=\frac{K_{I}-\tau_{\ell}}{\left(d_{2}-d_{1}+\frac{1}{2}\right)\left\lceil \frac{n}{c^{d_{2}-d_{1}}}\right\rceil }$,
	and  $w=\min\left\{ \tau_{\ell},\ \frac{(d_{2}-d_{1})^{2}(K_{I}-\tau_{\ell})^{2}}{2\tau_{\ell}\left(d_{2}-d_{1}+\frac{1}{2}\right)^{2}\left(\left\lceil \frac{n}{c^{d_{2}-d_{1}}}\right\rceil +1\right)^{2}}\right\} $.\\
	(3)There exists $\mathscr{M}:\left(B_{\mathbb{R}^{d_{2}-d_{1}}}(0,w)\right)^{n}\rightarrow\mathcal{M}_{\tau_{g},\tau_{\ell},K_{I},K_{v}}^{d_{1}}$
	one-to-one such that for each $y_{i}\in B_{\mathbb{R}^{d_{2}-d_{1}}}(0,w)$, $1\leq i\leq n$, $\mathscr{M}(y_{1},\ldots,y_{n})\cap T_{i}=\Phi_{i}([-K_{I},K_{I}]^{d_{1}-1}\times[0,a]\times\{y_{i}\})$.
	Hence for any $x_{1}\in T_{1},\ldots,x_{n}\in T_{n}$, $\mathscr{M}(\{\Pi_{(d_{1}+1):d_{2}}^{-1}\Phi_{i}^{-1}(x_{i})\}_{1\leq i\leq n})$
	passes through $x_{1},\ldots,x_{n}$.
}

\begin{proof}[Proof of Lemma~\ref{lem:lower.constructT}]
	By Lemma~\ref{lem:lower.cylinder.regularity}, we only need to show the
	case for $d_{1}=1$. This is since for $d_{1}>1$ case, we can
	build the set of manifolds in
	$\mathcal{M}_{\tau_{g},\tau_{\ell},K_{I},K_{v}}^{d_{1}}$ by forming a
	Cartesian product of the manifold with the cube as in Lemma~\ref{lem:lower.cylinder.regularity}.
	
	Let
	 $b=\frac{2(d_{2}-d_{1})(K_{I}-\tau_{\ell})}{\left(d_{2}-d_{1}+\frac{1}{2}\right)\left(\left\lfloor
		\frac{n}{c^{d_{2}-d_{1}}}\right\rfloor +1\right)}$, so that
\[
	b\geq2\sqrt{2w\tau_{\ell}}\quad\text{and}\quad2\tau_{\ell}+\left\lfloor\frac{n}{c^{d_{2}-d_{1}}}\right\rfloor a+\left(\left\lfloor	\frac{n}{c^{d_{2}-d_{1}}}\right\rfloor +1\right)b=2K_{I}.
\]
With such values of $a$, $b$, and $w$, align $T_{i}$, $R_{i}$, and $A_{i}$ in a zigzag way, as in Figure \ref{fig:proof.lower.Tconstruct}\subref{subfig:proof.lower.ARTconstruct}.

	\begin{figure}
		\begin{center}
			\begin{subfigure}[b]{0.5\textwidth}
				\begin{center}
				\includegraphics[scale=0.95]{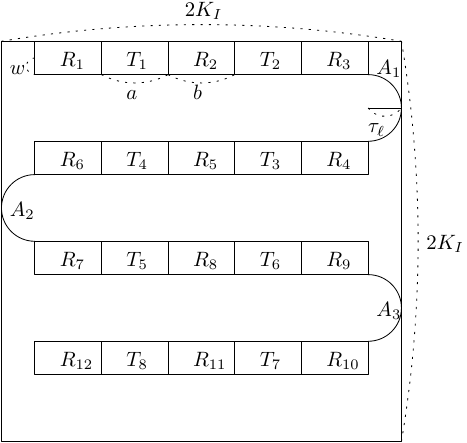}
				\end{center}
				\caption{alignment of $T_{i}$, $R_{i}$, and $A_{i}$}
				\label{subfig:proof.lower.ARTconstruct}
			\end{subfigure}
			\begin{subfigure}[b]{0.45\textwidth}
				\begin{center}
				\includegraphics[scale=0.95]{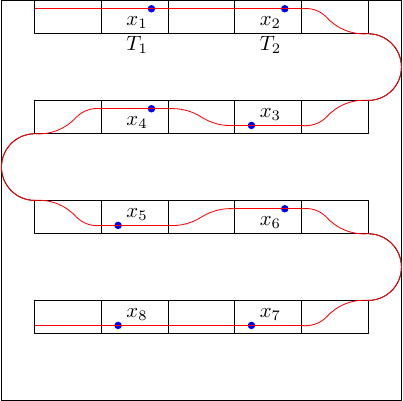}
				\end{center}
				\caption{manifold passing through $X_{i}$'s}
				\label{subfig:proof.lower.Tconstruct.manifold}
			\end{subfigure}
		\end{center}
		
		\caption{This figure illustrates the case where $d_{1}=1$ and $d_{2}=2$. \protect\subref{subfig:proof.lower.ARTconstruct}
			shows how $T_{i}$, $R_{i}$, and $A_{i}$'s are aligned in a zigzag.
			\protect\subref{subfig:proof.lower.Tconstruct.manifold} shows for given $x_{1}\in T_{1},\ldots,x_{n}\in T_{n}$ (represented as blue points), how
			 $\mathscr{M}(\{\Pi_{(d_{1}+1):d_{2}}^{-1}\Phi_{i}^{-1}(x_{i})\}_{1\leq i\leq n})$ (represented as a red curve) passes through
			 $x_{1},\ldots,x_{n}$}.
		\label{fig:proof.lower.Tconstruct}
	\end{figure}

	Then from the definition of $T_{i}$, (1) the
	$T_{i}$'s are distinct and (2) for each $T_{i}$, there exists an isometry
	$\Phi_{i}$ such that
	$T_{i}=\Phi_{i}\left([-K_{I},K_{I}]^{d_{1}-1}\times[0,a]\times
	B_{\mathbb{R}^{d_{2}-d_{1}}}(0,w)\right).$ There exists an isometry
	$\Psi_{i}$ such that
	$R_{i}=\Psi_{i}\left([-K_{I},K_{I}]^{d_{1}-1}\times[0,b]\times
	B_{\mathbb{R}^{d_{2}-d_{1}}}(0,w)\right)$ as well. Hence the conditions (1) and (2) are satisfied.
	
	We are left to define $\mathscr{M}$ that satisfies the condition (3). Now define a map from a set of points to a set of manifolds $\mathscr{M}:\left(B_{\mathbb{R}^{d_{2}-d_{1}}}(0,w)\right)^{n}\rightarrow\mathcal{M}_{\tau_{g},\tau_{\ell},K_{I},K_{v}}^{d_{1}}$
	as follows. For each $y_{i}\in B_{\mathbb{R}^{d_{2}-d_{1}}}(0,w)$, $1\leq i\leq n$, $\overset{4}{\underset{i=1}{\bigcup}}A_{i}\subset\mathscr{M}(y_{1},\ldots,y_{n})\subset\left(\overset{4}{\underset{i=1}{\bigcup}}A_{i}\right)\bigcup\left(\underset{i=1}{\bigcup}T_{i}\right)\bigcup\left(\underset{i=1}{\bigcup}R_{i}\right)$.
	The intersection of $\mathscr{M}(y_{1},\ldots,y_{n})$ and $T_{i}$
	is a line segment $\Phi_{i}([-K_{I},K_{I}]^{d_{1}-1}\times[0,a]\times\{y_{i}\})$, as in Figure \ref{fig:proof.lower.Tconstruct}\subref{subfig:proof.lower.Tconstruct.manifold}.
	Our goal is to make $\mathscr{M}(y_{1},\ldots,y_{n})$ be $C^{1}$
	and piecewise $C^{2}$.
	
	See Figure \ref{fig:proof.lower.Rconstruct} for the construction of the intersection of $\mathscr{M}(y_{1},\ldots,y_{n})$ and $R_{i}$. Given that
	 $\mathscr{M}(y_{1},\ldots,y_{n})\cap\left(\left(\overset{4}{\underset{i=1}{\bigcup}}A_{i}\right)\bigcup\left(\underset{i=1}{\bigcup}T_{i}\right)\right)$
	is determined, two points on
	 $\mathscr{M}(y_{1},\ldots,y_{n})\cap\partial R_{i}$ are already
	determined. By translation and rotation if necessary, for all $p,q$
	with $-w \leq q \leq p \leq w$, we need to find a $C^{2}$ curve with
	 reach $\geq \tau_{\ell}$ that starts from $(0,p)\in \mathbb{R}^{2}$,
	ends at $(b,q)\in \mathbb{R}^{2}$, and the velocities at both endpoints are
	parallel to $(1,0)\in \mathbb{R}^{2}$, as in Figure \ref{fig:proof.lower.Rconstruct}\subref{subfig:proof.lower.Rconstruct.initial}.
	
	\begin{figure}
		\begin{center}
			\begin{subfigure}[b]{0.48\textwidth}
				\begin{center}
				\includegraphics[width=\textwidth]{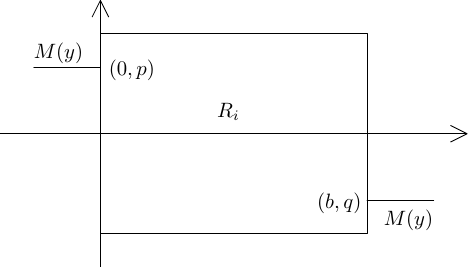}
				\end{center}
				\caption{}
				\label{subfig:proof.lower.Rconstruct.initial}
			\end{subfigure}
			\begin{subfigure}[b]{0.48\textwidth}
				\begin{center}
				\includegraphics[width=\textwidth]{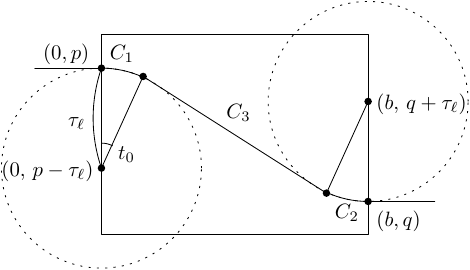}
				\end{center}
				\caption{}
				\label{subfig:proof.lower.Rconstruct.arcs}
			\end{subfigure}
		\end{center}
		\caption{\protect\subref{subfig:proof.lower.Rconstruct.initial} We need to find a $C^{2}$ curve with local reach $\geq \tau_{\ell}$ that starts
			from $(0,p)\in \mathbb{R}^{2}$, ends at $(b,q)$, and the velocities at both endpoints are parallel to $(1,0)$. \protect\subref{subfig:proof.lower.Rconstruct.arcs} $C_{1}$ and $C_{2}$ are arcs of circles of radius ${R}_{l}$, and $C_{3}$ is the cotangent segment of two circles.}
		\label{fig:proof.lower.Rconstruct}
	\end{figure}

	Let 
	\begin{equation}
		\label{eq:proof.lower.t0}
		t_{0}=\cos^{-1}\left(\frac{2\tau_{\ell}\left(2\tau_{\ell}-(p-q)\right)+b\sqrt{b^{2}-(p-q)\left(4\tau_{\ell}-(p-q)\right)}}{b^{2}+\left(2\tau_{\ell}-(p-q)\right)^{2}}\right),
	\end{equation}
	and let 
	\[
	C_{1}=\left\{ \left(0,p-\tau_{\ell}\right)+\tau_{\ell}\left(\sin t,\cos t\right)\ |\ 0\leq t\leq t_{0}\right\} .
	\]
	Then $C_{1}$ is an arc of a circle of which center is $\left(0,p-\tau_{\ell}\right)$,
	and starts at $(0,p)$ when $t=0$ and ends at $\left(\tau_{\ell}\sin t_{0},\ p-\tau_{\ell}(1-\cos t_{0})\right)$
	when $t=t_{0}.$ Also, the normalized velocities of $C_{1}$ at endpoints
	are 
	\begin{equation}
		\label{eq:proof.lower.C1velocity}
		(1,\,0)\text{ at }(0,\,p),\qquad(\cos t_{0},\,-\sin t_{0})\text{ at }\left(\tau_{\ell}\sin t_{0},\ p-\tau_{\ell}(1-\cos t_{0})\right).
	\end{equation}
	Similarly, let 
	\[
	C_{2}=\left\{ \left(b,q+\tau_{\ell}\right)-\tau_{\ell}\left(\sin t,\cos t\right)\ |\ 0\leq t\leq t_{0}\right\} .
	\]
	Then $C_{2}$ is an arc of a circle of whose center is $\left(b,q+\tau_{\ell}\right)$,
	and starts at $(b,q)$ when $t=0$ and ends at $\left(b-\tau_{\ell}\sin t_{0},\ q+\tau_{\ell}\left(1-\cos t_{0}\right)\right)$
	when $t=t_{0}$. Also, the normalized velocities of $C_{2}$ at endpoints
	are 
	\begin{equation}
		\label{eq:proof.lower.C2velocity}
		(-1,\,0)\text{ at }(b,\,q),\qquad(-\cos t_{0},\,\sin t_{0})\text{ at }\left(b-\tau_{\ell}\sin t_{0},\ q+\tau_{\ell}\left(1-\cos t_{0}\right)\right).
	\end{equation}
	Let 
	\begin{align*}
	C_{3} & =\Bigl\{(1-s)\left(\tau_{\ell}\sin t_{0},\ p-\tau_{\ell}(1-\cos t_{0})\right)+s\left(b-\tau_{\ell}\sin t_{0},\ q+\tau_{\ell}\left(1-\cos t_{0}\right)\right)\\
	& \qquad|\ 0\leq s\leq1\Bigr\},
	\end{align*}
	so that $C_{3}$ is a segment joining $\left(\tau_{\ell}\sin t_{0},\ p-\tau_{\ell}(1-\cos t_{0})\right)$
	(when $s=0$) and $(b-\tau_{\ell}\sin t_{0},\ q+\tau_{\ell}(1-\cos t_{0}))$
	(when $s=1$). Also, its velocity vector is 
	\begin{equation}
		\label{eq:proof.lower.C3velocity}
		\left(b-\tau_{\ell}\sin t_{0},\ q+\tau_{\ell}\left(1-\cos t_{0}\right)\right)\text{ for all }s\in[0,1].
	\end{equation}
	Then from definition of $t_{0}$ in \eqref{eq:proof.lower.t0}, 
	\[
	\cos t_{0}\left(q-p+2\tau_{\ell}\left(1-\cos t_{0}\right)\right)+\sin t_{0}\left(b-2\tau_{\ell}\sin t_{0}\right)=0,
	\]
	and this implies that $\left(b-2\tau_{\ell}\sin t_{0},\ q-p+2\tau_{\ell}\left(1-\cos t_{0}\right)\right)$
	is parallel to $\left(\cos t_{0},-\sin t_{0}\right)$. Hence the velocity
	vector of $C_{3}$ in \eqref{eq:proof.lower.C3velocity} is parallel to the velocity vector of $C_{1}$
	in \eqref{eq:proof.lower.C1velocity} at $\left(\tau_{\ell}\sin t_{0},\ p-\tau_{\ell}(1-\cos t_{0})\right)$ and the velocity vector of $C_{2}$ in \eqref{eq:proof.lower.C2velocity} at $(b-\tau_{\ell}\sin t_{0},\ q+\tau_{\ell}(1-\cos t_{0}))$, i.e. $C_{3}$ is cotangent to both $C_{1}$
	and $C_{2}$. See Figure \ref{fig:proof.lower.Rconstruct}\subref{subfig:proof.lower.Rconstruct.arcs}.
	
	Now we check whether  is of global
	reach $\geq\tau_{\ell}$, which implies both global reach $\geq\tau_{g}$ and local reach $\geq\tau_{\ell}$ since $\tau_{g}\leq \tau_{\ell}$.
	From \cite[Theorem 3.4]{AamariKCMRW2017}, the reach $\tau(M)$ of a manifold $M$ is realized in either the global case or the local case, where the global case refers to that there exist two points $q_{1},q_{2}\in M$ with $B(\frac{q_{1}+q_{2}}{2},\tau(M))\cap M = \emptyset$, and the local case refers to that there exists an arc-length parametrized geodesic $\gamma$ such that $||\gamma''(0)||_{2}=\frac{1}{\tau(M)}$. Now from the construction, any $q_{1},q_{2}\in \mathscr{M}(y_{1},\ldots,y_{n})$ with $B(\frac{q_{1}+q_{2}}{2},\tau)\cap \mathscr{M}(y_{1},\ldots,y_{n}) = \emptyset$ can only happen when $\tau \geq \tau_{\ell}$, so it suffices to check whether any arc-length parametrized geodesics $\gamma$ satisfies $||\gamma''(0)||_{2} \leq \frac{1}{\tau_{\ell}}$. And this is satisfied since $\mathscr{M}(y_{1},\ldots,y_{n})$ is piecewise either a straight line segment or an arc of a circle of radius $\tau_{\ell}$. Hence $\mathscr{M}(y_{1},\ldots,y_{n})$
	is of global reach $\geq\tau_{\ell}$.
	%Therefore from Corollary \ref{Lowerbound_LineCircleRegularity}, $C_{1}\cup C_{2}\cup C_{3}$
	%is of local curvature $\leq \kappa_{\ell}$. Refer to Figure \ref{fig:proof.lower.Rconstruct}\subref{subfig:proof.lower.Rconstruct.arcs}.
	%
	%Hence by defining $\mathscr{M}(Y_{1},\ldots,Y_{n})\cap R_{i}$ as
	%appropriate translation and rotation of $C_{1}\cup C_{2}$, $\mathscr{M}(Y_{1},\ldots,Y_{n})$
	%is of local curvature $\leq\kappa_{\ell}$.
\end{proof}

\noindent \textit{Claim} \ref{claim:lower.probability.bound}.
Let
$T=S_{n}\overset{n}{\underset{i=1}{\prod}}T_{i}$ where the $T_{i}$'s
are from Lemma~\ref{lem:lower.constructT}. Let $Q_{2}$ be the uniform distribution on $[-K_{I},K_{I}]^{d_{2}}$, and let $\mathcal{P}_{1}^{d_{1}}$ be as in \eqref{eq:lower.distribution.lower}. Then there exists $Q_{1} \in co(\mathcal{P}_{1}^{d_{1}})$ satisfying that for all $x\in intT$, there exists $r_{x}>0$ such that for
all $r<r_{x}$, 
\begin{equation}
\label{eq:proof.lower.probability.bound}
Q_{1}\left(\overset{n}{\underset{i=1}{\prod}}B_{\|\cdot\|_{\mathbb{R}^{d_{2}},\infty}}(x_{i},r)\right)\geq 2^{-n}Q_{2}\left(\overset{n}{\underset{i=1}{\prod}}B_{\|\cdot\|_{\mathbb{R}^{d_{2}},\infty}}(x_{i},r)\right).
\end{equation}

\begin{proof}[Proof of Claim~\ref{claim:lower.probability.bound}]
	Let $Q_{1}$ be from \eqref{eq:proof.lower.Q1andQ2} in Proposition~\ref{prop:lower.bound}. By symmetry, we can assume that $x\in\overset{n}{\underset{i=1}{\prod}}T_{i}$,
	i.e. $x_{1}\in T_{1},\ldots,x_{n}\in T_{n}$. Choose $r_{x}$ small
	enough so that $B(x,r_{x})\subset int T$. Then for all $r<r_{x}$,
	from the definition of $Q_{1}$ in \eqref{eq:proof.lower.Q1andQ2}, 
	\begin{align}
		\label{eq:proof.lower.Q1ball}
		Q_{1}\left(\overset{n}{\underset{i=1}{\prod}}B_{\|\cdot\|_{\mathbb{R}^{d_{2}},\infty}}(x_{i},r)\right) & =\int_{\mathcal{P}_{1}}P^{(n)}\left(\overset{n}{\underset{i=1}{\prod}}B_{\|\cdot\|_{\mathbb{R}^{d_{2}},\infty}}(x_{i},r)\right)d\mu_{1}(P)\nonumber\\
		& =\int_{C^{n}}\Phi(y)^{(n)}\left(\overset{n}{\underset{i=1}{\prod}}B_{\|\cdot\|_{\mathbb{R}^{d_{2}},\infty}}(x_{i},r)\right)\lambda_{C^{n}}(y)\nonumber\\
		& =\int_{C^{n}}\overset{n}{\underset{i=1}{\prod}}\lambda_{\mathscr{M}(y)}\left(B_{\|\cdot\|_{\mathbb{R}^{d_{2}},\infty}}(x_{i},r)\right)\lambda_{C^{n}}(y).
	\end{align}
	Then from the condition (3) in Lemma~\ref{lem:lower.constructT}, $\mathscr{M}(y)\cap T_{i}=\Phi_{i}\left([-K_{I},K_{I}]^{d_{1}-1}\times[0,a]\times\{y_{i}\}\right)$
	holds, hence 
	\begin{align*}
		& \mathscr{M}(y)\cap B_{\|\cdot\|_{\mathbb{R}^{d_{2}},\infty}}(x_{i},r)\\
		& \begin{cases}
			= \Phi_{i}\left(B_{\|\cdot\|_{\mathbb{R}^{d_{1}},\infty}}\left(\Pi_{1:d_{1}}(\Phi_{i}^{-1}(x_{i})),\ r\right)\times\{y_{i}\}\right), & \ \text{if }\left\Vert y_{i}-\Pi_{(d_{1}+1):d_{2}}(\Phi_{i}^{-1}(x_{i}))\right\Vert _{\mathbb{R}^{d_{2}-d_{1}}}<r,\\
			\supset \emptyset, & \ \text{otherwise}.
		\end{cases}
	\end{align*}
	And hence the volume of $\mathscr{M}(y)\cap B_{\|\cdot\|_{\mathbb{R}^{d_{2}},\infty}}(x_{i},r)$
	can be lower bounded as 
	\[
	\lambda_{\mathscr{M}(y)}\left(B_{\|\cdot\|_{\mathbb{R}^{d_{2}},\infty}}(x_{i},r)\right)\geq\frac{r^{d_{1}}}{2K_{I}^{d_{1}-1}an}I\left(\left\Vert y_{i}-\Pi_{(d_{1}+1):d_{2}}(\Phi_{i}^{-1}(x_{i}))\right\Vert _{\mathbb{R}^{d_{2}-d_{1}},\infty}<r\right).
	\]
	By applying this to \eqref{eq:proof.lower.Q1ball}, $Q_{1}\left(\overset{n}{\underset{i=1}{\prod}}B_{\|\cdot\|_{\mathbb{R}^{d_{2}},\infty}}(x_{i},r)\right)$
	can be lower bounded as 
	\begin{align}
	\label{eq:proof.lower.Q1ball.lowerbound}
	& Q_{1}\left(\overset{n}{\underset{i=1}{\prod}}B_{\|\cdot\|_{\mathbb{R}^{d_{2}},\infty}}(x_{i},r)\right)\nonumber \\
	& \geq\int_{C^{n}}\overset{n}{\underset{i=1}{\prod}}\frac{r^{d_{1}}}{2K_{I}{}^{d_{1}-1}an}I\left(\left\Vert y_{i}-\Pi_{(d_{1}+1):d_{2}}(\Phi_{i}^{-1}(x_{i}))\right\Vert _{\mathbb{R}^{d_{2}-d_{1}},\infty}<r\right)\lambda_{C^{n}}(y)\nonumber \\
	& =\frac{r^{d_{1}n}}{2^{n}K_{I}^{(d_{1}-1)n}(an)^{n}}\overset{n}{\underset{i=1}{\prod}}\int_{C}I\left(\left\Vert y_{i}-\Pi_{(d_{1}+1):d_{2}}(\Phi_{i}^{-1}(x_{i}))\right\Vert _{\mathbb{R}^{d_{2}-d_{1}},\infty}<r\right)\lambda_{C}(y_{i})\nonumber \\
	& =\frac{r^{d_{1}n}}{2^{n}K_{I}^{(d_{1}-1)n}(an)^{n}}\left(\frac{(2r)^{d_{2}-d_{1}}}{w^{d_{2}-d_{1}}\omega_{d_{2}-d_{1}}}\right)^{n}\nonumber \\
	& =\frac{2^{(d_{2}-d_{1}-1)n}r^{d_{2}n}}{K_{I}^{(d_{1}-1)n}w^{(d_{2}-d_{1})n}(an)^{n}\omega_{d_{2}-d_{1}}^{n}}\nonumber \\
	& \geq\frac{2^{(d_{2}-d_{1}-1)n}r^{d_{2}n}}{K_{I}^{d_{2}n}\omega_{d_{2}-d_{1}}^{n}},
	\end{align}
	where the last inequality uses $an\leq c^{d_{2}-d_{1}}K_{I}\leq\frac{K_{I}^{d_{2}-d_{1}+1}}{\tau_{\ell}^{d_{2}-d_{1}}}$
	and $w\leq \tau_{\ell}$.
	
	On the other hand, $Q_{2}\left(\overset{n}{\underset{i=1}{\prod}}B_{\|\cdot\|_{\mathbb{R}^{d_{2}},\infty}}(x_{i},r)\right)=\left(\frac{2r}{2K_{I}}\right)^{d_{2}n}=\frac{r^{d_{2}n}}{K_{I}^{d_{2}n}}$,
	so from this and \eqref{eq:proof.lower.Q1ball.lowerbound}, we get
	\eqref{eq:proof.lower.probability.bound} as 
	\begin{align*}
	Q_{1}\left(\overset{n}{\underset{i=1}{\prod}}B_{\|\cdot\|_{\mathbb{R}^{d_{2}},\infty}}(x_{i},r)\right) & \geq\frac{2^{(d_{2}-d_{1}-1)n}}{\omega_{d_{2}-d_{1}}^{n}}Q_{2}\left(\overset{n}{\underset{i=1}{\prod}}B_{\|\cdot\|_{\mathbb{R}^{d_{2}},\infty}}(x_{i},r)\right)\\
	& \geq 2^{-n}Q_{2}\left(\overset{n}{\underset{i=1}{\prod}}B_{\|\cdot\|_{\mathbb{R}^{d_{2}},\infty}}(x_{i},r)\right).
	\end{align*}
\end{proof}

\noindent \textbf{Proposition~\ref{prop:lower.bound}.} \textit{Fix $\tau_{g},\,\tau_{\ell}\in(0,\infty]$, $K_{I}\in[1,\infty)$,
	$K_{v}\in(0,2^{-m}]$, $K_{p}\in[(2K_{I})^{m},\infty)$,
	$d_{1},\,d_{2}\in\mathbb{N}$, with $\tau_{g}\leq\tau_{\ell}$ and
	$1\leq d_{1} < d_{2} \leq m$, and suppose that $\tau_{\ell}<K_{I}$.
	Then
	\begin{align}
		\label{eq:proof.lower.bound}
		& \underset{\hat{d}_{n}}{\inf}
		\underset{P\in {\cal Q}}{\sup}
		\mathbb{E}_{P^{(n)}}[\ell(\hat{d}_{n},d(P))]\nonumber\\
		& \geq\left(C_{d_{1},d_{2},K_{I}}^{(\ref*{prop:lower.bound})}\right)^{n}\min\left\{ \tau_{\ell}^{-2(d_{2}-d_{1}+1)}n^{-2},1\right\} ^{(d_{2}-d_{1})n},
	\end{align}
	where $C_{d_{1},d_{2},K_{I}}^{(\ref*{prop:lower.bound})}\in(0,\infty)$ is a constant depending only on $d_{1}$, $d_{2}$, and $K_{I}$
	and
	\[
	{\cal Q} = 
	\mathcal{P}_{\tau_{g},\tau_{\ell},K_{I},K_{v},K_{p}}^{d_{1}}
	\bigcup\mathcal{P}_{\tau_{g},\tau_{\ell},K_{I},K_{v},K_{p}}^{d_{2}}.
	\]
}

\begin{proof}[Proof of Proposition~\ref{prop:lower.bound}]
	Let $J=[-K_{I},K_{I}]^{d_{2}}$. Let $S_{n}$ be the
	permutation group, and $S_{n}\curvearrowright J^{n}$ by coordinate
	change, i.e. $\sigma\in S_{n}$, $x\in J^{n}$, $\sigma x:=(x_{\sigma(1)},\ldots,x_{\sigma(n)})$.
	For any set $A\subset J^{n}$, let $S_{n}A:=\{\sigma x\in J^{n}:\ \sigma\in S_{n},\ x\in A\}$.
	
	Let $T_{i}$ be $T_{i}$'s from Lemma~\ref{lem:lower.constructT}.
	Let $T:=S_{n}\overset{n}{\underset{i=1}{\prod}}T_{i}$, and $V:=\overset{n}{\underset{i=1}{\bigcup}}T_{i}=\Pi_{1:d_{2}}(T)$.
	Intuitively, $T$ is the set of points $x=(x_{1},\ldots,x_{n})$ where
	$x_{i}$ lies on one of the $T_{j}$.
	
	Let $C=B_{\mathbb{R}^{d_{2}-d_{1}}}(0,w)$ where $w$ is from Lemma
	\ref{lem:lower.constructT}, and precisely define a set of $d_{1}$-dimensional
	distribution $\mathcal{P}_{1}$ in \eqref{eq:lower.distribution.lower} and a set of $d_{2}$-dimensional distribution $\mathcal{P}_{2}$ in \eqref{eq:lower.distribution.upper} as 
	\begin{align}
		\label{eq:proof.lower.P1andP2}
		\mathcal{P}_{1} & =\{P\in\mathcal{P}_{\tau_{g},\tau_{\ell},K_{I},K_{v},K_{p}}^{d_{1}}:\text{ there exists }M\in\mathscr{M}(C^{n})\text{ such that }P\text{ is uniform on }M\},\nonumber \\
		\mathcal{P}_{2} & =\{\lambda_{J}\}\subset\mathcal{P}_{\tau_{g},\tau_{\ell},K_{I},K_{v},K_{p}}^{d_{2}}.
	\end{align}

	Define a map $\Phi:C^{n}\rightarrow\mathcal{P}_{1}$ by $\Phi(y_{1},\ldots,y_{n})=\lambda_{\mathscr{M}(y_{1},\ldots,y_{n})}$,
	i.e. the uniform measure on $\mathscr{M}(y_{1},\ldots,y_{n})$. Impose
	a topology and probability measure structure on $\mathcal{P}_{1}$
	by the pushforward topology and the uniform measure on $C^{n}$, i.e.
	$\mathcal{P}'\subset\mathcal{P}_{1}$ is open if and only if $\Phi^{-1}(\mathcal{P}')$
	is open in $C^{n}$, $\mathcal{P}'\subset\mathcal{P}_{1}$ is measurable
	if and only if $\Phi^{-1}(\mathcal{P}')\in\mathcal{B}(C^{n})$, and $\mu_{1}(\mathcal{P}')=\lambda_{C^{n}}(\Phi^{-1}(\mathcal{P}'))$.
	
	Define a probability measure $Q_{1}$, $Q_{2}$ on $(J^{n},\mathcal{B}(J^{n}))$
	by 
	\begin{equation}
		\label{eq:proof.lower.Q1andQ2}
		Q_{1}(A):=\int_{\mathcal{P}_{1}}P^{(n)}(A)d\mu_{1}(P)\qquad\text{ and }\qquad Q_{2}=\lambda_{J^{n}}.
	\end{equation}
	Fix $P\in\mathcal{P}_{1}$, let $x=\Phi^{-1}(P)$. Then $P^{(n)}(A)=\lambda_{\mathscr{M}(x)}^{(n)}(A)$
	is a measurable function of $x$ and $\Phi$ is a homeomorphism. Hence,
	$p^{(n)}(A)$ is measurable function and $Q_{1}(A)$ is well defined.
	Define $\nu=Q_{1}+\lambda_{J}$. Then $Q_{1},\ Q_{2}\ll\nu$, so there
	exist densities $q_{1}=\frac{dQ_{1}}{d\nu},\ q_{2}=\frac{dQ_{2}}{d\nu}$
	with respect to $\nu$.
	
	Then by applying Le Cam's Lemma (Lemma~\ref{lem:lower.LeCam}) with $\theta(P)=d(P)$,
	$\mathcal{P}_{1}$ and $\mathcal{P}_{2}$ from \eqref{eq:proof.lower.P1andP2}, and $Q_{1}$
	and $Q_{2}$ in \eqref{eq:proof.lower.Q1andQ2}, the minimax rate $\underset{\hat{d}_{n}}{\inf}\underset{P\in\mathcal{P}_{1}\cup\mathcal{P}_{2}}{\sup}\mathbb{E}_{P}\left[\ell(\hat{d}_{n},\,d(P))\right]$
	can be lower bounded as 
	\begin{align}
	\label{eq:proof.lower.LeCambound}
	\underset{\hat{d}_{n}}{\inf}\underset{P\in\mathcal{P}_{1}\cup\mathcal{P}_{2}}{\sup}\mathbb{E}_{P}\left[\ell(\hat{d}_{n},\,d(P))\right] & \geq\frac{\ell(d_{1},\,d_{2})}{2}\int_{J^{n}}q_{1}(x)\wedge q_{2}(x)d\nu(x)\nonumber \\
	& =\frac{1}{2}\int_{J^{n}}q_{1}(x)\wedge q_{2}(x)d\nu(x).
	\end{align}
	Then from Claim~\ref{claim:lower.probability.bound}, for all $x\in intT$,
	there exists $r_{x}>0$ s.t. for all $r<r_{x}$, 
	\[
	Q_{1}\left(\overset{n}{\underset{i=1}{\prod}}B_{\|\cdot\|_{\mathbb{R}^{d_{2}},\infty}}(x_{i},r)\right)\geq 2^{-n}Q_{2}\left(\overset{n}{\underset{i=1}{\prod}}B_{\|\cdot\|_{\mathbb{R}^{d_{2}},\infty}}(x_{i},r)\right).
	\]
	Hence $q_{1}(x)$ is lower bounded by $q_{2}(x)$ whenever $x\in intT$
	as
	\[
	q_{1}(x)\geq 2^{-n}q_{2}(x)\text{ if }x\in intT,
	\]
	and  $q_{1}(x)\wedge q_{2}(x)$ is
	correspondingly lower bounded by $q_{2}(x)$ as 
	\[
	q_{1}(x)\wedge q_{2}(x)\geq 2^{-n}q_{2}(x)1(x\in intT).
	\]
	Hence the integration of $q_{1}(x)\wedge q_{2}(x)$ over $T$ is lower
	bounded as 
	\begin{equation}
		\label{eq:proof.lower.q1wedgeq2.bound}
		\frac{1}{2}\int_{T}q_{1}(x)\wedge q_{2}(x) d\nu(x)\geq 2^{-n-1}\lambda_{J^{n}}(T).
	\end{equation}
	Then from $a=\frac{K_{I}-\tau_{\ell}}{\left(d_{2}-d_{1}+\frac{1}{2}\right)\left\lceil \frac{n}{c^{d_{2}-d_{1}}}\right\rceil }$
	and $w=\min\left\{ \tau_{\ell},\ \frac{(d_{2}-d_{1})^{2}(K_{I}-\tau_{\ell})^{2}}{2\tau_{\ell}\left(d_{2}-d_{1}+\frac{1}{2}\right)^{2}\left(\left\lceil \frac{n}{c^{d_{2}-d_{1}}}\right\rceil +1\right)^{2}}\right\} $,
	$\lambda_{J^{n}}(T)$ can be lower bounded as 
	\begin{align}
		\label{eq:proof.lower.Tvolume.bound}
		\lambda_{J^{n}}\left(S_{n}\overset{n}{\underset{i=1}{\prod}}T_{i}\right) & =n!\lambda_{J^{1}}(T_{1})^{n}\nonumber \\
		& =n!\left(\frac{(2K_{I})^{d_{1}-1}\omega_{d_{2}-d_{1}}aw^{d_{2}-d_{1}}}{(2K_{I})^{d_{2}}}\right)^{n}\nonumber \\
		& \geq\left(C_{d_{1},d_{2},K_{I}}^{(\ref*{prop:lower.bound},1)}\right)^{n}\min\left\{ \tau_{\ell}^{-2(d_{2}-d_{1}+1)}n^{-2},1\right\} ^{(d_{2}-d_{1})n},
	\end{align}
	for some constant $C_{d_{1},d_{2},K_{I}}^{(\ref*{prop:lower.bound},1)}$ that depends
	only on $d_{1}$, $d_{2}$, and $K_{I}$. Hence by combining \eqref{eq:proof.lower.LeCambound}, \eqref{eq:proof.lower.q1wedgeq2.bound}, and \eqref{eq:proof.lower.Tvolume.bound}, the minimax rate $\underset{\hat{d}_{n}}{\inf}\underset{P\in\mathcal{P}_{1}\cup\mathcal{P}_{2}}{\sup}\mathbb{E}_{P}\left[\ell(\hat{d}_{n},\,d(P))\right]$
	can be lower bounded as 
	\[
	\underset{\hat{d}_{n}}{\inf}\underset{P\in\mathcal{P}_{1}\cup\mathcal{P}_{2}}{\sup}\mathbb{E}_{P}\left[\ell(\hat{d}_{n},\,d(P))\right]\geq\left(C_{d_{1},d_{2},K_{I}}^{(\ref*{prop:lower.bound})}\right)^{n}\min\left\{ \tau_{\ell}^{-2(d_{2}-d_{1}+1)}n^{-2},1\right\} ^{(d_{2}-d_{1})n},
	\]
	for some constant $C_{d_{1},d_{2},K_{I}}^{(\ref*{prop:lower.bound})}$
	that depends only on $d_{1}$, $d_{2}$, and $K_{I}$. Then since
	$\mathcal{P}_{1}\subset\mathcal{P}_{\tau_{g},\tau_{\ell},K_{I},K_{v},K_{p}}^{d_{1}}$
	and $\mathcal{P}_{2}\subset\mathcal{P}_{\tau_{g},\tau_{\ell},K_{I},K_{v},K_{p}}^{d_{2}}$,
	the minimax rate $R_{n}$ in \eqref{eq:regular.minimax} can be lower bounded by the minimax
	rate $\underset{\hat{d}_{n}}{\inf}\underset{P\in\mathcal{P}_{1}\cup\mathcal{P}_{2}}{\sup}\mathbb{E}_{P}\left[\ell(\hat{d}_{n},\,d(P))\right]$,
	i.e. 
	\[
	\underset{\hat{d}_{n}}{\inf}\underset{P\in\mathcal{P}_{\tau_{g},\tau_{\ell},K_{I},K_{v},K_{p}}^{d_{1}}\cup\mathcal{P}_{\tau_{g},\tau_{\ell},K_{I},K_{v},K_{p}}^{d_{2}}}{\sup}\mathbb{E}_{P}[\ell(\hat{d}_{n},d(P))]\geq\underset{\hat{d}_{n}}{\inf}\underset{P\in\mathcal{P}_{1}\cup\mathcal{P}_{2}}{\sup}\mathbb{E}_{P}[\ell(\hat{d}_{n},d(P))],
	\]
	which completes the proof of showing \eqref{eq:proof.lower.bound}.
\end{proof}

\section{Proofs For Section \ref{sec:multidimension}}
\label{sec:proof.multidimension}

\noindent \textbf{Proposition~\ref{prop:multidimension.maximumrisk}.} \textit{
	Fix $\tau_{g},\,\tau_{\ell}\in(0,\infty]$, $K_{I}\in[1,\infty)$, $K_{v}\in(0,2^{-m}]$, $K_{p}\in[(2K_{I})^{m},\infty)$,  with $\tau_{g}\leq\tau_{\ell}$. Let $\hat{d}_{n}$ be in \eqref{eq:multidimension.estimator}. Then:
	\begin{align}
	& \underset{P\in\mathcal{P}_{\tau_{g},\tau_{\ell},K_{I},K_{v},K_{p}}^{d}}{\sup}\mathbb{E}_{P^{(n)}}\left[\ell\left(\hat{d}_{n},d(P)\right)\right] \nonumber \\
	& \leq 1(d>1)\left(C_{K_{I},K_{p},K_{v},m}^{(\ref*{prop:multidimension.maximumrisk})}\right)^{n}\max\left\{1,\tau_{g}^{-(dm+m-2d)n}\right\}n^{-\frac{1}{d-1}n}, \label{eq:proof.multidimension.maximumrisk}
	\end{align}
	where $C_{K_{I},K_{p},K_{v},m}^{(\ref*{prop:multidimension.maximumrisk})}\in(0,\infty)$ is a constant depending only on $K_{I}$, $K_{p}$, $K_{v}$, and $m$.
}
\begin{proof}[Proof of Proposition~\ref{prop:multidimension.maximumrisk}] Note that for all $P\in\mathcal{P}_{\tau_{g},\tau_{\ell},K_{I},K_{v},K_{p}}^{d}$
and $X_{1},\ldots,X_{n}\sim P$, by Lemma~\ref{lem:upper.lowerdim}, 
\[
\underset{\sigma\in S_{n}}{\min}\left\{ \overset{n-1}{\underset{i=1}{\sum}}\|X_{\sigma(i+1)}-X_{\sigma(i)}\|_{\mathbb{R}^{m}}^{d}\right\}\leq C_{K_{I},K_{v},m}^{(\ref*{lem:upper.lowerdim})}\max\left\{1,\tau_{g}^{d-m}\right\},
\]
hence $\hat{d_{n}}$ in  \eqref{eq:multidimension.estimator} always satisfies 
\begin{equation}
\label{eq:proof.multidimension.underestimate}
\hat{d}_{n}(X)\leq d=d(P).
\end{equation}
Hence when $d = 1$, the risk of $\hat{d}_{n}$ is $0$. When $d > 1$, from \eqref{eq:proof.multidimension.underestimate} and Proposition~\ref{prop:upper.bound}, the risk of $\hat{d}_{n}$ in \eqref{eq:multidimension.estimator} is upper bounded as 
\begin{align*}
& P^{(n)}\left[\hat{d}_{n}(X_{1},\cdots,X_{n})\neq d\right]\\
& =P^{(n)}\Biggl[\max\left\{ k\in[1,m]:\ \underset{\sigma\in S_{n}}{\min}\left\{ \overset{n-1}{\underset{i=1}{\sum}}\|X_{\sigma(i+1)}-X_{\sigma(i)}\|_{\mathbb{R}^{m}}^{k}\right\} \leq C_{K_{I},K_{v},m}^{(\ref*{lem:upper.lowerdim})}\max\left\{1,\tau_{g}^{k-m}\right\}\right\} \\
& \qquad\qquad<d\Biggr](\text{from }\eqref{eq:proof.multidimension.underestimate})\\
& \leq\overset{d-1}{\underset{k=1}{\sum}}P^{(n)}\left[\underset{\sigma\in S_{n}}{\min}\left\{ \overset{n-1}{\underset{i=1}{\sum}}\|X_{\sigma(i+1)}-X_{\sigma(i)}\|_{\mathbb{R}^{m}}^{k}\right\} \leq C_{K_{I},K_{v},m}^{(\ref*{lem:upper.lowerdim})}\max\left\{1,\tau_{g}^{k-m}\right\}\right]\\
& \leq\overset{d-1}{\underset{k=1}{\sum}}\left(C_{K_{I},K_{p},K_{v},m}^{(\ref*{prop:upper.maximumrisk})}\right)^{n}\max\left\{1,\tau_{g}^{-\left(\frac{d}{k}m+m-2d\right)n}\right\}n^{-\left(\frac{d}{k}-1\right)n}\ (\text{Proposition }\ref{prop:upper.bound})\\
& \leq\left(C_{K_{I},K_{p},K_{v},m}^{(\ref*{prop:multidimension.maximumrisk})}\right)^{n}\max\left\{1,\tau_{g}^{-(dm+m-2d)n}\right\}n^{-\frac{1}{d-1}n},
\end{align*}
where $C_{K_{I},K_{p},K_{v},m}^{(\ref*{prop:multidimension.maximumrisk})} =m C_{K_{I},K_{p},K_{v},m}^{(\ref*{prop:upper.maximumrisk})}$  is a constant depending only on $K_{I}$, $K_{p}$, $K_{v}$, and $m$. Therefore, the  risk is upper bounded as in \eqref{eq:proof.multidimension.maximumrisk}, as
\begin{align*}
& \underset{P\in\mathcal{P}_{\tau_{g},\tau_{\ell},K_{I},K_{v},K_{p}}^{d}}{\sup}\mathbb{E}_{P^{(n)}}\left[\ell\left(\hat{d}_{n},d(P)\right)\right] \\
& \leq 1(d>1)\left(C_{K_{I},K_{p},K_{v},m}^{(\ref*{prop:multidimension.maximumrisk})}\right)^{n}\max\left\{1,\tau_{g}^{-(dm+m-2d)n}\right\}n^{-\frac{1}{d-1}n}.
\end{align*}
\end{proof}

\noindent \textbf{Proposition~\ref{prop:multidimension.upper}.} \textit{
	%{\em(Upper bound for minimax rate, in multi-dimensions)}
	Fix $\tau_{g},\,\tau_{\ell}\in(0,\infty]$, $K_{I}\in[1,\infty)$, $K_{v}\in(0,2^{-m}]$, $K_{p}\in[(2K_{I})^{m},\infty)$, with $\tau_{g}\leq\tau_{\ell}$. Then:
	\begin{equation}
	\label{eq:proof.multidimension.upper}
	\underset{\hat{d}_{n}}{\inf}\underset{P\in\mathcal{P}}{\sup}\mathbb{E}_{P^{(n)}}\left[\ell\left(\hat{d}_{n},d(P)\right)\right]\leq\left(C_{K_{I},K_{p},K_{v},m}^{(\ref*{prop:multidimension.maximumrisk})}\right)^{n}\max\left\{1,\tau_{g}^{-(m^{2}-m)n}\right\}n^{-\frac{1}{m-1}n},
	\end{equation}
	where $C_{K_{I},K_{p},K_{v},m}^{(\ref*{prop:multidimension.maximumrisk})}$ is from Proposition~\ref{prop:multidimension.maximumrisk}.
}
\begin{proof}[Proof of Proposition~\ref{prop:multidimension.upper}]
	Note that \eqref{eq:upper.maximumrisk.bound} still holds when $\mathcal{P}$ is as in \eqref{eq:regular.distribution.multi}. Hence applying Proposition~\ref{prop:multidimension.maximumrisk} to \eqref{eq:upper.maximumrisk.bound} yields
	\begin{align*}
	& \underset{\hat{d}_{n}}{\inf}\underset{P\in\mathcal{P}}{\sup}\mathbb{E}_{P^{(n)}}\left[\ell\left(\hat{d}_{n},d(P)\right)\right]\\
	& \leq\max_{1\leq d\leq n}\left\{ \underset{P\in\mathcal{P}_{\tau_{g},\tau_{\ell},K_{I},K_{v},K_{p}}^{d}}{\sup}\mathbb{E}_{P^{(n)}}\left[\ell\left(\hat{d}_{n},d(P)\right)\right]\right\} \\
	& \leq\left(C_{K_{I},K_{p},K_{v},m}^{(\ref*{prop:multidimension.maximumrisk})}\right)^{n}\max\left\{1,\tau_{g}^{-(m^{2}-m)n}\right\}n^{-\frac{1}{m-1}n}.
	\end{align*}
	Hence the minimax rate $R_{n}$ in \eqref{eq:regular.minimax} is upper bounded as in \eqref{eq:proof.multidimension.upper}.
\end{proof}

\noindent \textbf{Proposition~\ref{prop:multidimension.lower}.} \textit{
	%{\em(Lower bound for minimax rate, in multi-dimensions)} 
	Fix $\tau_{g},\,\tau_{\ell}\in(0,\infty]$, 
	$K_{I}\in[1,\infty)$, $K_{v}\in(0,2^{-m}]$, 
	$K_{p}\in[(2K_{I})^{m},\infty)$, with 
	$\tau_{g}\leq\tau_{\ell}$ and suppose that $\tau_{\ell}<K_{I}$. Then,
	\begin{equation}
	\underset{\hat{d}_{n}}{\inf}\underset{P\in\mathcal{P}}{\sup}\mathbb{E}_{P^{(n)}}[\ell(\hat{d}_{n},d(P))]\geq\left(C_{K_{I}}^{(\ref*{prop:multidimension.lower})}\right)^{n}\min\left\{ \tau_{\ell}^{-4}n^{-2},1\right\} ^{n}
	\end{equation}
	where $C_{K_{I}}^{(\ref*{prop:multidimension.lower})}\in(0,\infty)$ is a constant depending only on $K_{I}$.
}
	
\begin{proof}[Proof of Proposition~\ref{prop:multidimension.lower}]
For any $d_{1}$ and $d_{2}$, from Proposition~\ref{prop:lower.bound},
\begin{align*}
 & \underset{\hat{d}_{n}}{\inf}\underset{P\in\mathcal{P}}{\sup}\mathbb{E}_{P^{(n)}}[\ell(\hat{d}_{n},d(P))]\\
 & \geq\underset{\hat{d}_{n}}{\inf}\underset{P\in\mathcal{P}_{\tau_{g},\tau_{\ell},K_{I},K_{v},K_{p}}^{d_{1}}\cup\mathcal{P}_{\tau_{g},\tau_{\ell},K_{I},K_{v},K_{p}}^{d_{2}}}{\sup}\mathbb{E}_{P^{(n)}}[\ell(\hat{d}_{n},d(P))]\\
 & \geq\left(C_{d_{1},d_{2},K_{I}}^{(\ref*{prop:lower.bound})}\right)^{n}\min\left\{ \tau_{\ell}^{-2(d_{2}-d_{1}+1)}n^{-2},1\right\} ^{(d_{2}-d_{1})n}\\
\end{align*}
Hence by plugging in $d_{1}=1$ and $d_{2}=2$, the minimax rate $R_{n}$ in \eqref{eq:regular.minimax} is lower bounded as in \eqref{eq:proof.multidimension.upper}, as
\[
\underset{\hat{d}_{n}}{\inf}\underset{P\in\mathcal{P}}{\sup}\mathbb{E}_{P^{(n)}}[\ell(\hat{d}_{n},d(P))] \geq\left(C_{K_{I}}^{(\ref*{prop:multidimension.lower})}\right)^{n}\min\left\{ \tau_{\ell}^{-4}n^{-2},1\right\} ^{n}
\]
with $C_{K_{I}}^{(\ref*{prop:multidimension.lower})}=C_{d_{1}=1,d_{2}=2,K_{I}}^{(\ref*{prop:lower.bound})}$.

\end{proof}